\documentclass{amsart}

\usepackage{amsmath}
\usepackage{amsthm}
\usepackage{amsfonts}
\usepackage{amssymb}
\usepackage{enumerate}
\usepackage{graphicx}

\usepackage{hyperref}
\hypersetup{
    colorlinks,%
    citecolor=black,%
    filecolor=black,%
    linkcolor=black,%
    urlcolor=black
}

\newcommand{\E}{\mathbb{E}}
\newcommand{\Prob}{\mathbb{P}}

\newcommand{\R}{\mathbb{R}}

\newcommand{\eps}{\varepsilon}

\newcommand{\CM}{\mathcal{M}} 
\newcommand{\med}{\mathbb{M}}
\newcommand{\BJ}{\mathbf{J}}
\newcommand{\Bj}{\mathbf{1}}

\newcommand{\indicator}[1]{\ensuremath{\mathbf{1}_{\{#1\}}}}
\newcommand{\oindicator}[1]{\ensuremath{\mathbf{1}_{{#1}}}}

\DeclareMathOperator{\tr}{tr}

\DeclareMathOperator{\Span}{Span}
\DeclareMathOperator{\range}{range}
\DeclareMathOperator{\rank}{rank}

\DeclareMathOperator{\var}{Var}

\theoremstyle{plain}
  \newtheorem{theorem}{Theorem}[section]

  \newtheorem{lemma}[theorem]{Lemma}
  \newtheorem{corollary}[theorem]{Corollary}

\theoremstyle{definition}
  \newtheorem{definition}[theorem]{Definition}
  
  \newtheorem{remark}[theorem]{Remark}

\begin{document}
\title{Eigenvectors of random matrices: A survey}

\author[S. O'Rourke]{Sean O'Rourke}
\address{Department of Mathematics, University of Colorado at Boulder, Boulder, CO 80309  }
\email{sean.d.orourke@colorado.edu}

\author{Van Vu}
\thanks{V. Vu is supported by research grants DMS-0901216 and AFOSAR-FA-9550-09-1-0167.}
\address{Department of Mathematics, Yale University, New Haven, CT 06520, USA}
\email{van.vu@yale.edu}

\author{Ke Wang}
\address{Jockey Club Institute for Advanced Study,
Hong Kong University of Science and Technology, Hong Kong, China}
\address{Computing \& Mathematical Sciences, California Institute of Technology, Pasadena, CA 91125, USA}
\email{kewang@ust.hk}

\begin{abstract}
Eigenvectors of large matrices (and graphs) play an essential role in combinatorics and theoretical computer science. 
The goal of this survey is to provide an up-to-date account on properties of eigenvectors when the matrix (or graph) is random. 

\end{abstract}

\maketitle

\tableofcontents

\section{Introduction}

Eigenvectors of large matrices (and graphs) play an essential role in combinatorics and theoretical computer science. For instance, many properties of a graph can be deduced or estimated from its eigenvectors.  In recent years, many algorithms have been developed which take advantage of this relationship to study various problems including spectral clustering \cite{SM, vonL}, spectral partitioning \cite{McSherry, PSL}, PageRank \cite{PBMW}, and community detection \cite{Newman, Newman1}.  

The goal of this survey is to study basic properties of eigenvectors when the matrix (or graph) is random.  As this survey is written with combinatorics/theoretical computer science readers in mind, we try to formalize the results in  forms which
are closest to their interest and give references for further extensions.  Some of the results presented in this paper are new with proofs included, while many others have appeared in very recent papers.

 We focus on the following models of random matrices.   
 
 
\begin{definition}[Wigner matrix] \label{def:wigner}
Let $\xi, \zeta$ be real random variables with mean zero.  We say $W$ is a \emph{Wigner matrix} of size $n$ with atom variables 
$\xi,\zeta$ if $W = (w_{ij})_{i,j=1}^n$ is a random real symmetric $n \times n$ matrix that satisfies the following conditions.
\begin{itemize}
\item (independence) $\{w_{ij} : 1 \leq i \leq j \leq n\}$ is a collection of independent random variables.
\item (off-diagonal entries) $\{w_{ij} : 1 \leq i < j \leq n\}$ is a collection of independent and identically distributed (iid) copies of $\xi$.
\item (diagonal entries) $\{w_{ii} : 1 \leq i \leq n\}$ is a collection of iid copies of $\zeta$.  
\end{itemize}
If $\xi$ and $\zeta$ have the same distribution, we say $W$ is a Wigner matrix with atom variable $\xi$.  In this case, $W$ is a real symmetric matrix whose entries on and above the diagonal are iid copies of $\xi$. 
\end{definition}

One can similarly define complex Hermitian Wigner matrices whose off-diagonal entries are complex-valued random variables.  For the purposes of this survey, we focus on real symmetric Wigner matrices.

Throughout the paper, we consider various assumptions on the atom variables $\xi$ and $\zeta$.  We will always assume that $\xi$ is \emph{non-degenerate}, namely that there is no value $c$ such that $\Prob (\xi =c) =1$.

\begin{definition}[$K$-bounded] A random variable $\xi$ is \emph{$K$-bounded} if $|\xi| \le K$ almost surely. For combinatorial applications, the entries usually take on values $\{0, \pm 1\}$ and are 
$1$-bounded.  In general, we allow $K$ to depend on $n$.  
\end{definition}

Occasionally, we will assume the atom variable $\xi$ is symmetric.  Recall that a random variable $\xi$ is \emph{symmetric} if $\xi$ and $-\xi$ have the same distribution.  The Rademacher random variable, which take the values $\pm 1$ with equal probability is an example of a symmetric random variable.  

\begin{definition}[Sub-exponential]
A random variable $\xi$ is called \emph{sub-exponential} with exponent $\alpha > 0$  if there exists  a constant $\beta  > 0$ such that 
\begin{equation} \label{eq:subgauss}
	\Prob (|\xi| > t) \leq \beta \exp(- t^{\alpha} / \beta ) \quad \text{for all } t > 0.
\end{equation}  If $\alpha=2$, then $\xi$ is called \emph{sub-gaussian}, and $1 / \beta$ is the \emph{sub-gaussian moment} of $\xi$.  
\end{definition}

The prototypical example of a Wigner real symmetric matrix is the \emph{Gaussian orthogonal ensemble} (GOE).  
The GOE is defined as a Wigner random matrix with atom variables $\xi, \zeta$, where $\xi$ is a standard normal random variable and $\zeta$ is a normal random variable with mean zero and variance $2$.  The GOE is widely-studied in random matrix theory and mathematical physics.  However, due to its continuous nature, the GOE has little use in combinatorial applications.  

A case of principal interest in combinatorial applications is when both $\xi$ and $\zeta$ are Bernoulli random variables.  Let $0 < p < 1$, and take $\xi$ to be the random variable
\begin{equation} \label{eq:def:Bernoullip} 
	\xi := \left\{
     \begin{array}{ll}
       1-p, & \text{ with probability } p, \\
       -p, & \text{ with probability } 1-p.
     \end{array}
   \right. 
\end{equation}
In particular, $\xi$ has mean zero by construction.  Let $B_n(p)$ denote the $n \times n$ Wigner matrix with atom variable $\xi$ (i.e. the entries on and above the diagonal are iid copies of $\xi$).  
  We refer to $B_n(p)$ as a \emph{symmetric Bernoulli matrix} (with parameter $p$).   
The most interesting case is when $p=1/2$. In this case, $2 B_n(p)$ is the random symmetric Rademacher matrix, whose entries are $\pm 1$ with probability $1/2$. 

In applications, one often considers the adjacency matrix of a random graph.  We let $G(n,p)$ denote the Erd\"os--R\'enyi random graph on $n$ vertices with edge density $p$.  That is, $G(n,p)$ is a
 simple graph on $n$ vertices such that each edge $\{i,j\}$ is in $G(n,p)$ with probability $p$, independent of other edges.  We let $A_n(p)$ be the zero-one adjacency matrix of $G(n,p)$.  
$A_n(p)$ is not a Wigner matrix since its entries do not have mean zero.  

For the sake of simplicity, we will sometimes consider random graphs with loops (thus the diagonals of the adjacency matrix are also random).  Let $\tilde{G}(n,p)$ denote the Erd\"os--R\'enyi random graph with loops on $n$ vertices with edge density $p$.  That is, $\tilde{G}(n,p)$ is a graph on $n$ vertices such that each edge $\{i,j\}$ (including the case when $i = j$) is in $\tilde{G}(n,p)$ with probability $p$, independent of other edges.  We let $\tilde{A}_n(p)$ denote the zero-one adjacency matrix of $\tilde{G}(n,p)$.  Technically, $\tilde{A}_n(p)$ is not a Wigner random matrix because its entries do not have mean zero.  However, we can view $\tilde{A}_n(p)$ as a low rank deterministic perturbation of a Wigner matrix.  That is, we can write $\tilde{A}_n(p)$ as 
$$ \tilde{A}_n(p) =  p \BJ_n +  B_n(p), $$  where $\BJ_n$ is the all-ones matrix.  We also observe that $A_n(p)$ can be formed from $\tilde{A}_n(p)$ by replacing the diagonal entries  with zeros.  
In this note, we focus on the case when $p$ is a constant, independent of the dimension $n$, but will also give references to the case when $p$ decays with $n$.

For an $n \times n$ Hermitian matrix $M$, we let 
$ \lambda_1(M) \leq \cdots \leq \lambda_n(M) $
denote the ordered eigenvalues of $M$ (counted with multiplicity) with corresponding unit eigenvectors $v_1(M), \ldots, v_n(M)$.  It is important to notice that 
 the eigenvectors of $M$ are {\it not uniquely determined}.  In general, we let $v_1(M), \ldots, v_n(M)$ denote any orthonormal basis of eigenvectors of $M$ such that
 $$ M v_i(M) = \lambda_i(M) v_i(M), \quad 1 \leq i \leq n.$$
On the other hand, it is well known that   if the  spectrum of $M$ is \emph{simple} (i.e. all eigenvalues have multiplicity one) 
 then     the unit eigenvectors $v_1(M), \ldots, v_n(M)$ are  determined uniquely up to phase.  In this case, to avoid ambiguity, we always assume that the eigenvectors are taken so that their first non-zero coordinate 
 is positive.  Theorem \ref{thm:simplespectrum} below shows that, with high probability, the eigenvalues of many matrices under discussion 
 are simple and the coordinates of all eigenvectors are non-zero.

\begin{theorem}[\cite{TVsimple}]    \label{thm:simplespectrum}
Let $W$ be an $n \times n$ Wigner matrix with atom variables $\xi, \zeta$.  
\begin{itemize}
\item If $\xi$ is non-degenerate, then, for every $\alpha > 0$, there exists $C > 0$ (depending on $\alpha$ and $\xi$) such that the spectrum of $W$ is simple with probability at least $1 - C n^{-\alpha}$.  
\item Moreover, if $\xi$ and $\zeta$ are sub-gaussian, then, for every $\alpha > 0$, there exists $C > 0$ (depending on $\alpha$ and $\xi$) such that every coordinate of every eigenvector of $W$ is non-zero with probability at least $1 - C n^{-\alpha}$.  
\end{itemize}
In addition, the conclusions above also hold for the matrices $A_n(p)$ and $\tilde{A}_n(p)$ when $p \in (0,1)$ is fixed, independent of $n$.  
\end{theorem} 
\begin{remark}
In many cases, the bound $1 - C n^{-\alpha }$ appearing in Theorem \ref{thm:simplespectrum} can be improved to $1 - C\exp (n^{-c} )$ for some constants $C, c > 0$ (see \cite{NTV}). 
\end{remark}

Consequently, in many theorems  we can assume that the 
spectrum is simple and the eigenvectors are uniquely defined (using the convention that the first coordinate is positive).


\subsection{Overview and outline}  
Although this survey examines several models of random matrices, we mostly focus on Wigner matrices, specifically Wigner matrices whose atom variables have light tails (e.g. sub-exponential atom variables).  
In this case, the main message we would like to communicate 
 is that an eigenvector of a Wigner matrix behaves like a random vector uniformly distributed on the unit sphere. 
 While this concept is natural and intuitive, it is usually non-trivial to prove quantitatively. 
  We present several estimates to quantify this statement from different aspects.  In particular, we address
\begin{itemize} 
\item the joint distribution of many coordinates (Section \ref{sec:universality}),
\item the largest coordinate, i.e. the $\ell^{\infty}$-norm (Section \ref{sec:deloc}),
\item the smallest coordinate (Section \ref{sec:deloc}),
\item the $\ell^p$-norm (Section \ref{sec:mass}),
\item the amount of mass on a subset of coordinates (Section \ref{sec:mass}). 
\end{itemize} 
For comparison, we also discuss other models of random matrices, such as heavy-tailed and band random matrices, whose eigenvectors do not behave like random vectors uniformly distributed on the unit sphere.

The paper is organized as follows.  In Section \ref{sec:GOE}, we present results for the eigenvectors of a matrix drawn from the GOE (defined above). 
 This is the ``ideal'' situation, where  the eigenvectors are  indeed uniformly distributed on the unit sphere, thanks to the rotational invariance of the ensemble. 
 Next, in Section \ref{sec:universality}, we discuss {\it universality} results which give a direct comparison between eigenvectors of general Wigner matrices with those of the GOE.  In Section \ref{sec:deloc}, we present bounds on the largest coordinate (i.e. the $\ell^\infty$-norm) and smallest coordinate of eigenvectors of Wigner matrices.   Section \ref{sec:mass} gives more information about the mass distribution on the coordinates of an eigenvector, such as the magnitude of the smallest coordinates, or the amount of mass one can have on any subset of coordinates of linear size.  In  Section \ref{sec:nonzero}, we present results for deterministic perturbations of Wigner matrices.   In particular, this section contains results for the adjacency matrices $A_n(p)$ and $\tilde{A}_n(p)$.   In Section \ref{sec:ht}, we change direction and review two ensembles of random matrices whose eigenvectors do not behave like random vectors uniformly distributed on the unit sphere.  In Section \ref{sec:others} and Section \ref{sec:rrg}, we discuss results concerning some other models of random matrices, which have not been 
 mentioned in the introduction, such as
 random non-symmetric matrices or the adjacency matrix of a random regular graph. In the remaining sections, we represent proofs of the new results mentioned in the previous sections. The appendix contains a number of technical lemmas.

\subsection{Notation}
The \emph{spectrum} of an $n \times n$ real symmetric matrix $M$ is the multiset $\{ \lambda_1(M), \ldots, \lambda_n(M)\}$.  We use the phrase {\it bulk} of the spectrum to refer to the eigenvalues $\lambda_i(M)$ with $\eps n \le i \le (1-\eps) n$, where $\eps$ is a small positive constant. The remaining eigenvalues form the {\it edge} of the spectrum.  

For a vector $v = (v_i)_{i=1}^n \in \mathbb{R}^n$, we let
$$ \|v \|_{\ell^p} := \left( \sum_{i=1}^n |v_i|^p \right)^{1/p} $$
denote the $\ell^p$-norm of $v$.  We let $\|v\| := \|v\|_{\ell^2}$ be the Euclidean norm of $v$.    For $S \subset [n] := \{1, \ldots, n\}$, we denote
$$ \|v\|_S := \left( \sum_{i \in S} |v_i|^2 \right)^{1/2}. $$
It follows that $\|v\|_{[n]} = \|v\|$.  Let
$$ \| v \|_{\ell^\infty} := \max_{1 \leq i \leq n} |v_i| $$
denote the $\ell^\infty$-norm of $v$.  We introduce the notation
\begin{equation} \label{eq:def:min}
	\|v \|_{\min} := \min_{1 \leq i \leq n} |v_i| 
\end{equation}
to denote the minimal coordinate (in magnitude) of $v$; notice that this is {\it not} a norm. 
  For two vectors $u, v \in \mathbb{R}^n$, we let $u \cdot v = u^\mathrm{T} v$ be the dot product between $u$ and $v$.  $S^{n-1}$ is the unit sphere in $\mathbb{R}^n$.   

For an $m \times n$ matrix $M$ with real entries, we let $\|M\|$ denote the spectral norm of $M$:
$$ \| M \| := \max_{v \in S^{n-1}} \| M v \|. $$

We let $E^c$ denote the complement of the event $E$.  For a set $S$, $|S|$ is the cardinality of $S$.  For two random variables (or vectors) $\xi$ and $\zeta$, we write $\xi \sim \zeta$ if $\xi$ and $\zeta$ have the same distribution.  We let $N(\mu, \sigma^2)$ denote the normal distribution with mean $\mu$ and variance $\sigma^2$.  In particular, $\xi \sim N(0,1)$ means $\xi$ is a standard normal random variable.  In addition, $|N(0,1)|$ is the distribution which arises when one take the absolute value of a standard normal random variable.  

We use asymptotic notation under the assumption that $n$ tends to infinity.  We write $X = o(Y)$ if $|X| \leq c_n Y$ for some $c_n$ which converges to zero as $n \to \infty$.  In particular, $o(1)$ denotes a term which tends to zero as $n$ tends to infinity.

\section {A toy case: The Gaussian orthogonal ensemble} \label{sec:GOE}

The GOE (defined above) is a special example of a Wigner matrix, enjoying the property that it is  invariant under orthogonal transformations.  By the spectral theorem, any $n \times n$ real symmetric matrix $M$ can be decomposed as $M = U D U^\mathrm{T}$, where $U$ is an orthogonal matrix whose columns are the eigenvectors of $M$ and $D$ is a diagonal matrix whose entries are the eigenvalues of $M$.  However, if $M$ is drawn from the GOE, the property of being invariant under orthogonal transformations implies that $U$ and $D$ are independent.  It follows from \cite[Section 2.5.1]{AGZ} that the eigenvectors $v_1(M), \ldots, v_n(M)$ are uniformly distributed on 
$$ S^{n-1}_+ := \{ x = (x_1, \ldots, x_n) \in S^{n-1} : x_1 > 0 \}, $$
and $(v_1(M), \ldots, v_n(M))$ is distributed like a sample of Haar measure on the orthogonal group $O(n)$, with each column multiplied by a norm one scalar so that the columns all belong to $S_+^{n-1}$.  Additionally, the eigenvalues $(\lambda_1(M), \ldots, \lambda_n(M))$ have joint density
$$ p_n(\lambda_1, \ldots, \lambda_n) := \left\{
     \begin{array}{ll}
       \frac{1}{Z_n} \prod_{1 \leq i < j \leq n} |\lambda_i - \lambda_j| \prod_{i=1}^n e^{-\lambda_i^2/4}, & \text{ if } \lambda_1 \leq \cdots \leq \lambda_n, \\
       0, & \text{ otherwise}.
     \end{array}
   \right. $$
where $Z_n$ is a normalization constant.  We refer the interested reader to \cite[Section 2.5.1]{AGZ} for a further discussion of these results as well as additional references.  

In the following statements, we gather information about a unit eigenvector $v$ of a $n \times n$ matrix drawn from the GOE.  As discussed above, this is equivalent to studying a random vector $v$ uniformly distributed on $S^{n-1}_+$.  In fact, since all of the properties discussed below are invariant under scaling by a norm one scalar, we state all of the results in this section for a unit vector $v$ uniformly distributed over the unit sphere $S^{n-1}$ of $\mathbb{R}^n$.

\begin{theorem}[Smallest and largest coordinates] \label{thm:GOEnorm} 
Let $v$ be a random vector uniformly distributed on the unit sphere $S^{n-1}$.   Then, for any $C > 1$,  with probability at least $1-2 n^{1 - C}-\exp\left(-\frac{(C-1)^2}{4C^2}n\right)$, 
\begin{equation} \label{eq:GOEmax}
	\| v \|_{\ell^\infty} \leq \sqrt {\frac{2C^3 \log n}{n} }. 
\end{equation}
In addition, for $n \geq 2$, any $0 \leq c < 1$, and any $a>1$,  
\begin{equation} \label{eq:GOEmin}
	\| v \|_{\min} \geq \frac{c}{a} \frac{1}{n^{3/2}}
\end{equation}
with probability at least $\exp \left( -2c \right) - \exp \left( - \frac{ a^2 - \sqrt{2 a^2 -1} }{2} n \right)$.  
\end{theorem} 
Next, we obtain the order of the $\ell^p$-norm of such a vector.  
\begin{theorem}[$\ell^p$-norm] \label{thm:GOElp}
Let $v$ be a random vector uniformly distributed on the unit sphere $S^{n-1}$.   Then, for any $p \geq 1$, there exists $c_p > 0$, such that almost surely
$$ \| v \|_{\ell^p}^p = n^{1 - p/2} c_p + o(n^{1-p/2}). $$
\end{theorem}

We now consider the distribution of mass over the coordinates of such a vector.  Recall that the beta distribution with shape parameters $\alpha, \beta > 0$ is the continuous probability distribution supported on the interval $[0,1]$ with probability density function
$$ f(x; \alpha, \beta) := \left\{
     \begin{array}{ll}
       \frac{1}{B(\alpha, \beta)} x^{\alpha-1}(1-x)^{\beta - 1}, & \text{if } 0 \leq x \leq 1 \\
       0, & \text{otherwise},
     \end{array}
   \right. $$
where $B(\alpha, \beta)$ is the normalization constant.  A random variable $\xi$ beta-distributed with shape parameters $\alpha, \beta > 0$ will be denoted by $\xi \sim \mathrm{Beta}(\alpha, \beta)$.

\begin{theorem}[Distribution of mass] \label{thm:distr}
Let $S$ be a proper nonempty subset of $[n]$, and let $v$ be a random vector uniformly distributed on the unit sphere $S^{n-1}$.  Then $\|v\|_S^2$ is distributed according to the beta distribution
$$ \|v\|_S^2 \sim \mathrm{Beta} \left( \frac{|S|}{2}, \frac{n - |S|}{2} \right). $$ 
In particular, this implies that $\|v\|_S^2$ has mean $\frac{|S|}{n}$ and variance $\frac{|S|(n-|S|)}{n^2(n/2+1)}$.  
\end{theorem}

As a corollary of Theorem \ref{thm:distr}, we obtain the following central limit theorem.  

\begin{theorem}[Central limit theorem] \label{thm:clt}
Fix $\delta \in (0,1)$.  For each $n \geq 2$, let $S_n \subset [n]$ with $|S_n| = \lfloor \delta n \rfloor$, and let $v_n$ be a random vector uniformly distributed on the unit sphere $S^{n-1}$.  Then
$$  \sqrt{\frac{n^3}{2|S_n|(n-|S_n|)}} \left( \|v_n\|^2_{S_n} - \frac{|S_n|}{n} \right) \longrightarrow N(0,1) $$
in distribution as $n \to \infty$.  
\end{theorem}

We also have the following concentration inequality.  

\begin{theorem}[Concentration]\label{thm:gaussian}
Let $S \subset [n]$, and let $v$ be a random vector uniformly distributed on the unit sphere $S^{n-1}$.  Then, for any $t > 0$, 
$$ \left| \| v \|_{S}^2 - \frac{|S|}{n} \right| \leq \frac{8}{n} \left( \sqrt{nt} + t \right) $$
with probability at least $1 - \exp(- c n) - 4 \exp(-t)$, where $c > 0$ is an absolute constant.  
\end{theorem}

\begin{remark}
The proof of Theorem \ref{thm:gaussian} actually reveals that, for any $t > 0$, 
$$ \left| \| v \|_{S}^2 - \frac{|S|}{n} \right| \leq \frac{4}{n} \left( \sqrt{|S|t} + t \right) + 4 \frac{|S|}{n^2} \left( \sqrt{nt} + t \right) $$
with probability at least $1 - \exp(- c n) - 4 \exp(-t)$.  
\end{remark}

We now consider the maximum and minimum order statistics 
$$ \max_{S \subset [n] : |S| = \lfloor \delta n \rfloor} \|v\|_S \quad \text{ and } \quad \min_{S \subset [n] : |S|=\lfloor \delta n \rfloor} \|v\|_S $$
for some $0 < \delta < 1$.  Recall that the $\chi^2$-distribution with $k$ degrees of freedom is the distribution of a sum of the squares of $k$ independent standard normal random variables.  Let $F$ be the cumulative distribution function of the $\chi^2$-distribution with one degree of freedom. Following the notation in \cite{CHM}, let $Q$ denote the quantile function of $F$.  That is, 
\begin{equation} \label{eq:def:Q}
	Q(s) := \inf \{ x \in \mathbb{R} : F(x) \geq s \}, \quad 0 < s \leq 1, \quad Q(0) := \lim_{s \searrow 0} Q(s). 
\end{equation}
Define 
\begin{equation} \label{eq:def:H}
	H(s) := - Q(1-s), \quad 0\le s <1. 
\end{equation}

\begin{theorem}[Extreme order statistics] \label{thm:order}
For each $n \geq 1$, let $v_n$ be a random vector uniformly distributed on the unit sphere $S^{n-1}$.  Then, for any fixed $0 < \delta < 1$, 
$$ \max_{S \subset [n] : |S|=\lfloor \delta n \rfloor} \|v_n\|_S^2 \longrightarrow - \int_{0}^{\delta} H(u)~ du $$ and
$$ \min_{S \subset [n] : |S|=\lfloor \delta n \rfloor} \|v_n\|_S^2 \longrightarrow - \int_{1-\delta}^{1} H(u)~ du $$
in probability as $n \to \infty$.  
\end{theorem}

\begin{remark} \label{delta-mass} 
The integrals on the right-hand side involving $H$ can be rewritten in terms of the standard normal distribution.  Indeed, from \eqref{eq:nor} and integration by parts, we find
\begin{align*}
- \int_{0}^{\delta} H(u)~ du &= 2 \int_{\Phi^{-1}(1-\frac{\delta}{2})}^{\infty} x^2 \Phi'(x)\,dx\\
&= {\delta} + \sqrt{\frac{2}{\pi}} \Phi^{-1} \left(1-\frac{\delta}{2} \right) \exp\left(-\frac{1}{2} \left[ \Phi^{-1}\left(1-\frac{\delta}{2} \right) \right]^2 \right),
\end{align*}
where $\Phi(x)$ is the cumulative distribution function of the standard normal distribution. We also get
\begin{align*}
- \int_{1-\delta}^{1} H(u)~ du &= 1+ \int_{0}^{1-\delta} H(u)~ du\\
&=\delta - \sqrt{\frac{2}{\pi}} \Phi^{-1}\left(\frac{1+\delta}{2} \right) \exp\left(-\frac{1}{2} \left[ \Phi^{-1} \left(\frac{1+\delta}{2} \right) \right]^2 \right).
\end{align*}
\end{remark}

\begin{remark} \label{rem:delta-mass-order}
Using the expressions in Remark \ref{delta-mass}, one can show that, as $\delta$ tends to zero, 
$$ -\int_{0}^{\delta} H(u)~ du = \Theta \left(  \delta \log (\delta^{-1})  \right) $$
and
$$ -\int_{1-\delta}^{1} H(u)~ du = \Theta (\delta^3). $$
In other words, Theorem \ref{thm:order} implies that the smallest $\delta n$ coordinates of an eigenvector contribute only $\Theta(\delta^3)$ fraction of the mass.  
\end{remark}

 Figure \ref{fig:order} depicts numerical simulations of the distribution of $\max_{S \subset [n] : |S|=\lfloor \delta n \rfloor} \|v_n\|_S^2$ when $n=600$.  The simulation shows that $\max_{S \subset [n] : |S|=\lfloor \delta n \rfloor} \|v_n\|_S^2$ is highly concentrated at the value $g(\delta):= - \int_{0}^{\delta} H(u)~ du$ as indicated in Theorem \ref{thm:order}.  Indeed, numerical calculations show 
 \begin{align*}
	g(1/4) \approx 0.7236069618, \quad g(1/3) \approx 0.8167557098, \\
	g(1/2) \approx 0.9286740823, \quad g(3/5) \approx 0.9646603703.
\end{align*}

\begin{figure}[!Ht]
   \includegraphics[width=\textwidth]{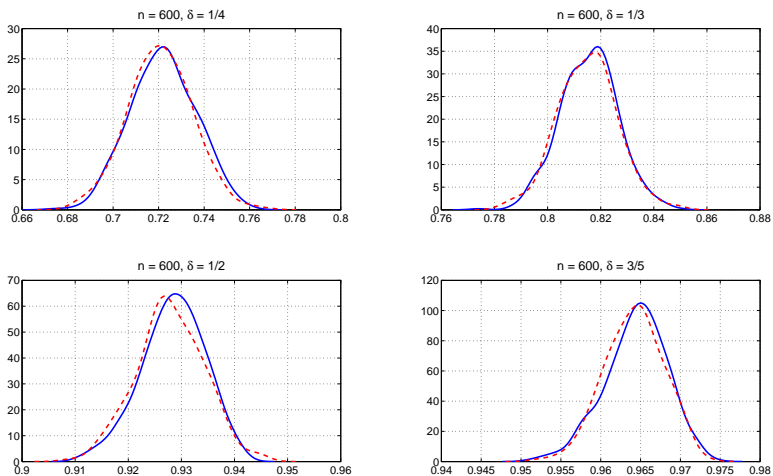}
   \caption{The probability density function of the distribution of $\max_{S \subset [n]: |S|=\lfloor \delta n \rfloor} \|v\|_S^2$ when $n=600$, based on $500$ samples. The blue curve is for $v$, a unit eigenvector of a matrix drawn from the GOE, while the red dashed curve corresponds to a unit eigenvector of a random symmetric Bernoulli matrix.}
   \label{fig:order}
\end{figure}

Our next result shows that these extreme order statistics concentrate around their expectation, even for relatively small values of $n$.  

\begin{theorem}[Concentration of the extreme order statistic]\label{thm:con-max}
Let $v$ be a random vector uniformly distributed on the unit sphere $S^{n-1}$.  Then, for any $1 \leq m \leq n$ and every $t \geq 0$, 
\begin{equation} \label{eq:conc-max}
	\Prob \Big( \Big| \max_{S \subset [n] : |S|=m} \|v \|_S - \E \max_{S \subset [n] : |S|=m} \|v \|_S \Big| > t \Big) \leq C \exp(-c nt^2) 
\end{equation}
and
\begin{equation} \label{eq:conc-min}
	\Prob \Big( \Big| \min_{S \subset [n] : |S|=m} \|v \|_S - \E \min_{S \subset [n] : |S|=m} \|v \|_S \Big| > t \Big) \leq C \exp(-c nt^2) 
\end{equation}
where $C,c > 0$ are absolute constants.  
\end{theorem}

We prove these results in Section \ref{sec:proof:GOE}.  

\section {Direct comparison theorems } \label{sec:universality}

In many cases, one can compare the eigenvectors of a Wigner matrix directly to those of the GOE.  In the random matrix theory literature, such results are often referred to as \emph{universality results}.  

Let $W$ be a Wigner matrix with atom variables $\xi, \zeta$.  We will require that $\xi$ and $\zeta$ satisfy a few moment conditions.  In particular, we say $\xi, \zeta$ are from the class $\CM_4$ if
\begin{itemize}
\item $\xi$ and $\zeta$ are sub-exponential random variables, 
\item $\E (\xi) = \E (\zeta)= \E(\xi^3) =0$, and
\item $\E(\xi^2) =1, \E(\xi^4) =3, \E(\zeta^2 )=2$.
\end{itemize} 
These conditions imply that the first four moments of the off-diagonal entries of $W$ match those of the GOE, and the first 2 moments of the diagonal entries of $W$ match those of the GOE. 

Let $\{X_n\}_{n \geq 1}$ and $\{Y_n\}_{n \geq 1}$ be two sequences of real random variables. In order to show that they have (asymptotically) the same distribution, it suffices to show that 
$$\Prob ( X_n \le t) -  \Prob( Y_n \le t) = o (1) $$ 
for all $t$. Notice that 
$$\Prob ( X_n \le t )= \E f_t(X_n),$$ 
where $f_t(X_n)$ is the indicator function of the event $\{X_n \le t\}$.  In practice, it is useful to smoothen $f_t$ a little bit to obtain a function  
with bounded derivatives, at the cost of an extra error term.  We are going to use this strategy in the next few theorems, which allow us to  compare $X_n$ with $Y_n$ by bounding 
$ \E F(X_n) - \E F(Y_n) $ for a large set of test functions $F$.  

One of the first universality results for eigenvectors is the following result.  

\begin{theorem}[Eigenvector coefficients of Wigner matrices, \cite{TVuniv-vector}]\label{thm:TVvector}  
Given $C > 0$ and random variables $\xi, \zeta$ from the class $\CM_4$, there are constants $\delta, C_0 > 0$ such that the following statement holds.  
Let $W$ be an $n \times n$ Wigner matrix with atom variables $\xi, \zeta$ and with unit eigenvectors $v_1, \ldots, v_n$.  Let $v_i(j)$ denote the $j$th entry of $v_i$.  For $1 \leq i,j \leq n$, let $Z_{i,j}$ be independent random variables with $Z_{i,j} \sim N(0,1)$ 
  for $j > 1$ and $Z_{i,1} \sim |N(0,1)|$.  Let $1 \leq k \leq n^{\delta} $, and let $1 \leq i_1 < \ldots < i_k \leq n$ and $1 \leq j_1 < \ldots < j_k \leq n$ be indices.
Then
\begin{equation}\label{sil}
	\left| \E F\left( ( \sqrt{n} v_{i_a}(j_b))_{1 \leq a,b \leq k} \right)  - \E F \left(( Z_{i_a,j_b})_{1 \leq a,b \leq k} \right) \right|   \le C_0 n^{-\delta} 
\end{equation}
whenever $F: \R^{k^2} \to \R$ is a smooth function obeying the bounds
$$ |F(x)| \leq C$$
and
$$ |\nabla^j F(x)| \leq C n^\delta $$
for all $x \in \R^{k^2}$ and $0 \leq j \leq 5$.
\end{theorem}

Theorem \ref{thm:TVvector} is essentially the first part of \cite[Theorem 7]{TVuniv-vector}. The original theorem has $o(1)$ on the right-hand side of \eqref{sil}, but one can obtain the better bound $C_0 n^{-\delta}$ using the same proof. The essence of this theorem is that any set of at most $n^{\delta}$ coordinates (which may come from different eigenvectors) behaves like a set of iid Gaussian random variables.  Because of our normalization, requiring the first non-zero coordinate of the eigenvectors of $W$ to be positive, we cannot expect this coordinate to be close to a normal random variable.  However, Theorem \ref{thm:TVvector} shows that it is close in distribution to the absolute value of a standard normal random variable.  Similar results were also obtained in \cite{KY} for the case when $k= O(1)$.  

Another way to show that an eigenvector $v$ behaves like a random vector  $u$ uniformly distributed on the unit sphere is to fix a unit vector $a$, and compare the distribution of the inner product $v \cdot a $ with $u \cdot a$. 
It is easy to show that $u \cdot a$ satisfies the central limit theorem.  Namely, if $u_n$ is a random vector uniformly distributed on the unit sphere $S^{n-1}$ and $\{a_n\}$ is a sequence of unit vectors with $a_n \in S^{n-1}$, then
$$ \sqrt{n} u_n \cdot a_n \longrightarrow N(0,1) $$
in distribution as $n \to \infty$.  We refer the reader to \cite{Gfr} and \cite[Proposition 25]{TVuniv-vector} for details and a proof of this statement.  The following is a consequence of \cite[Theorem 13]{TVuniv-vector}. 

\begin{theorem}[Theorem 13 from \cite{TVuniv-vector}]  \label{thm:TVvectro2}  
Let $\xi, \zeta$ be random variables from the class $\CM_4$, and assume $\xi$ is a symmetric random variable.  For each $n \geq 1$, let $W_n$ be an $n \times n$ Wigner matrix with atom variables $\xi, \zeta$.  Let $\{a_n\}$  be a sequence of unit vectors with $a_n \in S^{n-1}$ such that $ \lim_{n \to \infty} \| a_n \| _{\ell^\infty} = 0 $, and let $\{i_n\}$ be a sequence of indices with $i_n \in [n]$.  Then
$$ \sqrt{n} v_{i_n}(W_n) \cdot a_n \longrightarrow N(0,1) $$
in distribution as $n \to \infty$.  
\end{theorem}

In a recent paper \cite{BY}, it was showed that one can  remove the assumption that $\xi, \zeta$ are from the class $\CM_4$, with an extra restriction on the index $i_n$.

\begin{theorem}[Theorem 1.2 from \cite{BY}] \label{thm:BY}
Let $\xi, \zeta$ be sub-exponetial random variables with mean zero, and assume $\xi$ has unit variance.  For each $n \geq 1$, let $W_n$ be an $n \times n$ Wigner matrix with atom variables $\xi, \zeta$.  In addition, let $\{a_n\}$ be a sequence of unit vectors with $a_n \in S^{n-1}$.  
 Then there exists $\tilde\delta > 0$ such that, for any fixed integer $m \geq 1$ and any
$$ I_n \subset T_n := \left( [1, n^{1/4}] \cup [n^{1-\tilde\delta}, n-n^{1-\tilde\delta}] \cup [n - n^{1/4}, n] \right) \cap \mathbb{N} $$
with $|I_n| = m$, 
$$ \sqrt{n}( |a_n \cdot  v_i(W_n)| )_{i \in I_n} \longrightarrow ( |Z_i| )_{i=1}^m $$
in distribution as $n \to \infty$, where $Z_1, \ldots, Z_m$ are iid standard normal random variables.  
\end{theorem}

 One immediately obtains the following corollary 
 
\begin{corollary}[Corollary 1.3 from \cite{BY}] \label{cor:BY}
Let $W_n$, $\tilde\delta$, and $T_n$ be as in Theorem \ref{thm:BY}. For each $n \geq 1$, let $i_n \in T_n$, and let $v_{i_n}^{(n)}(j)$ be the $j$th coordinate of $v_{i_n}(W_n)$.  Assume $l$ is a fixed positive integer.  Then, for any $J_n \subset [n]$ with $|J_n| = l$,
$$ \sqrt{n} \left( \left| v_{i_n}^{(n)}(j) \right| \right)_{j \in J_n} \longrightarrow (|Z_j|)_{j=1}^l $$
in distribution as $n \to \infty$, where $Z_1, \ldots, Z_l$ are iid standard normal random variables.  
\end{corollary}

\begin{remark}
Both Theorem \ref{thm:BY} and Corollary \ref{cor:BY} hold for more general ensembles of random matrices than Wigner matrices.  Indeed, the results in \cite{BY} hold for so-called generalized Wigner matrices, where the entries above the diagonal are only required to be independent, not necessarily identically distributed; see \cite[Definition 1.1]{BY} for details.  
\end{remark}

\section {Extremal coordinates} \label{sec:deloc}

In this section,  we investigate the largest and smallest coordinates (in absolute value) of an eigenvector of a Wigner matrix. 

\subsection{The largest coordinate}
In view of Theorem \ref{thm:GOEnorm}, it is natural to conjecture that, for any eigenvector $v_j$, 
\begin{equation} \label{eq:linftyopt}
	\|v_j \| _{\ell^\infty} = O \left( \sqrt { \frac{\log n}{n}} \right) 
\end{equation}
with high probability.  The first breakthrough was made in  \cite[Theorem 5.1]{ESY1},   which provides a bound of the form $\frac{\log^{O(1)} n}{\sqrt n }$ for a large proportion of eigenvectors $v_j$, under 
some technical conditions on the distribution of the entries. 
 This result was extended to all eigenvectors under a weaker assumption in \cite{TVedge, TVuniv}, and many newer  
 papers  \cite{E, EKYY, ESY2, ESY1, ESY3, ESY4, ESYY, EY, EYY2, EYY1, EYY, TVsur, VW} give further strengthening and generalizations.  In particular,  the optimal bound in \eqref{eq:linftyopt} was recently obtained  in \cite{VW}. 
 
\begin{theorem}[Optimal upper bound; Theorem 6.1 from \cite{VW}] \label{thm:delo1} 
Let $\xi$ be a sub-gaussian  random variable with mean zero and unit variance.   Then, for any $C_1 > 0$ and any $0 < \eps < 1$, there exists a constant $C_2 > 0$ such that the following holds.  Let $W$ be an $n \times n$ Wigner matrix with atom variable $\xi$.  
\begin{itemize} 
\item (Bulk case) For any $ \eps n \le i \le  (1-\eps) n$, 
$$ \| v_i (W) \|_{\ell^\infty} \le C_2  \sqrt { \frac {\log  n } { n } } $$
with probability at least $1 -  n^{-C_1} $. 
\item (Edge case) For  $1\le i \le \eps n$ or $ (1- \eps)n \le i \le n $,  
$$ \| v_i(W) \|_{\ell^\infty} \le C_2   \frac{\log n}{\sqrt{n}} $$
with probability at least $1 -  n^{-C_1}$.  
\end{itemize}
\end{theorem} 

\begin{remark} \label{nobound} 
Theorem \ref{thm:delo1} was proved in \cite{VW} under the stronger assumption that $\xi$ is $K$-bounded for some constant $K$. One can 
easily obtain the more general result here by observing that \cite[Lemma 1.2]{VW} holds under the sub-gaussian assumption, as a special case of 
a recent result, \cite[Theorem 2.1]{RVhw}. The rest of the proof remains unchanged. 
\end{remark} 

Using Theorem \ref{thm:TVvector}, one can derive  a matching lower bound.  

\begin{theorem} [Matching lower bound] \label{thm:matching lower} 
Let $\xi, \zeta$ be random variables from the class $\CM_4$.  Then there exists a constant $c > 0$ such that the following holds.  Let $W$ be an $n \times n$ Wigner matrix with atom variables $\xi, \zeta$.  For any $1 \leq i \leq n$, 
$$ \| v_i(W) \|_{\ell^\infty} \geq c \sqrt{ \frac{\log n}{n}} $$
with probability $1 - o(1)$.  
\end{theorem} 

\vskip2mm 
\noindent {\bf Open question.} Prove the optimal bound in \eqref{eq:linftyopt} for all eigenvectors (including the edge case). 
\vskip2mm 
\noindent {\bf Open question.} Is the limiting distribution of $\| v_i \| _{\ell^\infty}$ universal, or does it depend on the atom variables $\xi$ and $\zeta$? 
\vskip2mm 

One can deduce a slightly-weaker version of Theorem \ref{thm:delo1} when the entries are only assumed to be sub-exponential with exponent $0 < \alpha < 2$ by applying \cite[Theorem 6.1]{VW} and a truncation argument.  We apply such an argument in Section \ref{sec:proof:extreme} to obtain the following corollary. 

\begin{corollary} \label{cor:delo} 
Let $\xi$ be a symmetric sub-exponential random variable with exponent $\alpha > 0$.  Assume further that $\xi$ has unit variance.  Then, for any $C_1$ and $0 < \eps < 1$, there is a constant $C_2$ such that the following holds.  Let $W$ be an $n \times n$ Wigner matrix with atom variable $\xi$. 
\begin{itemize} 
\item (Bulk case) For any $ \eps n \le i \le  (1-\eps) n$, 
$$ \| v_i (W) \|_{\ell^\infty} \le C_2  \sqrt { \frac {\log^{1 + 2/\alpha}  n } { n } } $$
with probability at least $1 - n^{-C_1} $. 
\item (Edge case) For  $1\le i \le \eps n$ or $ (1- \eps)n \le i \le n $,  
$$ \| v_i(W) \|_{\ell^\infty} \le C_2  \frac{\log^{1 + 2/\alpha} n}{\sqrt{n}} $$
with probability at least $1 - n^{-C_1}$. 
\end{itemize}
\end{corollary} 

Corollary \ref{cor:delo} falls short of the optimal bound in \eqref{eq:linftyopt}; it remains an open question to obtain the optimal bound when the entries are not sub-gaussian.  

Next, we discuss a generalization.  Notice that $\| v \|_{\ell^\infty} = \max _{1 \le i \le n} | v \cdot e_i |$, where $e_1, \ldots, e_n$ are the standard basis vectors. 
What happens if we consider the inner product $v \cdot u$, for any fixed unit vector $u$?  The theorems in the previous section show that, under certain technical assumptions, $ \sqrt  n v \cdot u $ is approximately Gaussian, which implies that $| v \cdot u |$ is typically of order $n^{-1/2} $. The following result gives a strong deviation bound.

\begin{theorem} [Isotropic delocalization, Theorem 2.16 from \cite{BEKHY}]  \label{thm:BEKHY}
Let $\xi$ and $\zeta$ be zero-mean sub-exponential random variables, and assume that $\xi$ has unit variance.  Let $W$ be an $n \times n$ Wigner matrix with atom variables $\xi, \zeta$.  Then, for any $C_1 > 0$ and $0 < \eps < 1/2$, there exists $C_2 > 0$ (depending only on $C_1, \eps$, $\xi$, and $\zeta$) such that 
$$ \sup_{1 \leq i \leq n} |v_i \cdot u| \leq \frac{n^{\eps}}{\sqrt{n}}, $$
for any fixed unit vector $u \in S^{n-1}$, with probability at least $1 - C_2 n^{-C_1}$.    
\end{theorem} 

\begin{remark} 
Theorem \ref{thm:BEKHY} and the results in \cite{BEKHY} actually hold for a larger class of so-called generalized Wigner matrices whose entries have bounded moments; see \cite[Section 2.2]{BEKHY} for details.   
\end{remark}

\subsection{The smallest coordinate} 

We now turn our attention to the smallest coordinate of a given eigenvector.  To this end, we recall the definition of $\| \cdot \|_{\min}$ given in \eqref{eq:def:min}.  

\begin{theorem}[Individual coordinates: Lower bound] \label{lower}  
Let $\xi, \zeta$ be sub-gaussian random variables with mean zero, and assume $\xi$ has unit variance.  Let $W$ be an $n \times n$ Wigner matrix with atom variables $\xi, \zeta$.  Let $v_1, \ldots, v_n$ denote the unit eigenvectors of $W$, and let $v_i(j)$ denote the $j$th coordinate of $v_i$.  Then there exist constants $C, c, c_0, c_1 > 0$ (depending only on $\xi, \zeta$) such that, for any $n^{-c_0} < \alpha < c_0$ and $\delta \geq n^{-c_0/\alpha}$, 
$$ \sup_{1 \leq i, j \leq n} \Prob \left( |v_i(j)| \leq \frac{\delta}{\sqrt{n} (\log n)^{c_1 }} \right) \leq C \frac{\delta}{\sqrt{\alpha}} + C \exp(-c \log^2 n). $$
\end{theorem} 

By the union bound, we immediately obtain the following corollary. 
\begin{corollary} \label{cor:lower}
Let $\xi, \zeta$ be sub-gaussian random variables with mean zero, and assume $\xi$ has unit variance.  Let $W$ be an $n \times n$ Wigner matrix with atom variables $\xi, \zeta$.  Let $v_1, \ldots, v_n$ denote the unit eigenvectors of $W$, and let $v_i(j)$ denote the $j$th coordinate of $v_i$.  Then there exist constants $C, c, c_0, c_1 > 0$ (depending only on $\xi, \zeta$) such that, for any $n^{-c_0} < \alpha < c_0$ and $\delta \geq n^{-c_0/\alpha}$, 
$$ \sup_{1 \leq i \leq n} \Prob \left( \| v_i(W) \|_{\min} \leq \frac{\delta}{n^{3/2} (\log n)^{c_1 }} \right) \leq C \frac{\delta}{\sqrt{\alpha}} + C \exp(-c \log^2 n). $$
\end{corollary}

In particular, Corollary \ref{cor:lower} implies that, with high probability, 
$$ \| v_i(W) \|_{\min} = \Omega \left( \frac{1}{n^{3/2} (\log n)^{c_1} } \right). $$
In view of Theorem \ref{thm:GOEnorm}, this is optimal up to logarithmic factors.

\vskip2mm 

\noindent {\bf Open question.} Is the limiting distribution of $\| v_j(W) \|_{\min}$ universal, or does it depend on the atom variables $\xi$ and $\zeta$?

\vskip2mm

We prove Corollary \ref{cor:delo} and Theorem \ref{lower} in Section \ref{sec:proof:extreme}.

\section {No-gaps delocalization} \label{sec:mass}

The results in the previous section address how much mass can be contained in a single coordinate.  
We next turn to similar estimates for the amount of mass contained on a number of  coordinates.  In particular, the following results assert that any subset of coordinates of linear size must contain a non-negligible fraction of the vector's $\ell^2$-norm.   Following Rudelson and Vershynin \cite{RVgaps}, we refer to this phenomenon as \emph{no-gaps delocalization}.

Using Corollary \ref{cor:BY}, we obtain the following analogue of Theorem \ref{thm:order}.  Recall the function $H$ defined in \eqref{eq:def:H}.  

\begin{theorem}\label{thm:generalized} 
Let $W_n$, $\tilde\delta$, and $T_n$ be as in Theorem \ref{thm:BY}.  For each $n$, let $k_n \in T_n$.  
Let $v_{k_n}$ be a unit eigenvector of $W_n$ corresponding to $\lambda_{k_n}(Y_n)$. Then, for any fixed $0< \delta <1$, 
$$ \max_{S\subset [n]: |S|=\lfloor\delta n \rfloor} \|v_{k_n}\|_S^2 \longrightarrow -\int_{0}^{\delta} H(u) \, du  $$
and
$$ \min_{S\subset [n]: |S|=\lfloor\delta n \rfloor} \|v_{k_n}\|_S^2 \longrightarrow -\int_{1-\delta}^1 H(u) \, du  $$
in probability as $n \to \infty$, where $H$ is defined in \eqref{eq:def:H}.
\end{theorem}

We prove Theorem \ref{thm:generalized}  in Section \ref{sec:proof:wigner}.  The integrals in the right-hand side involving $H$ can be expressed in a number of different ways; see Remarks \ref{delta-mass} and \ref{rem:delta-mass-order} for details.  In particular, Remark \ref{rem:delta-mass-order} shows that, for $\delta$ sufficiently small, the first integral is $\Theta(\delta \log(\delta^{-1}))$ and the second is $\Theta(\delta^3)$.  

While  this survey was written,  Rudelson and Vershynin posted the following theorem, which addresses the lower bound for more general models of random matrices. 
Their proof is more involved and very different from that of Theorem \ref{thm:generalized}.  

\begin{theorem}[Theorem 1.5 from \cite{RVgaps}] \label{thm:RVgapsgen}
Let $\xi$ be a real random variable which satisfies 
$$ \sup_{u \in \mathbb{R}} \Prob(|\xi - u| \leq 1) \leq 1 - p \quad \text{and} \quad \Prob(|\xi| > K) \leq p/2 $$
for some $K,p > 0$.  Let $W$ be an $n \times n$ Wigner matrix with atom variable $\xi$.  Let $\kappa \geq 1$ be such that the event $\mathcal{B}_\kappa := \{ \|W\| \leq \kappa \sqrt{n} \}$ holds with probability at least $1/2$.  Let $\delta \geq 1/n$ and $t \geq c_1 \delta^{-7/6} n^{-1/6} + e^{-c_2/\sqrt{\delta}}$.  Then, conditionally on $\mathcal{B}_\kappa$, the following holds with probability at least $1 - (c_3 t)^{\delta n}$: every eigenvector $v$ of $W$ satisfies 
$$ \min_{S \subset [n]: |S| \geq \delta n} \| v \|_S \geq (\delta t)^6. $$
Here, the constants $c_1, c_2, c_3 > 0$ depend on $p$, $K$, and $\kappa$.  
\end{theorem}
\begin{remark}
The results in \cite{RVgaps} hold for even more general ensembles of random matrices than what is stated above; see \cite[Assumption 1.1]{RVgaps} for details.  
For many atom variables $\xi, \zeta$, the spectral norm $\| W \|$ is strongly concentrated (see Lemma \ref{lemma:norm} or \cite[Corollary 2.3.6]{Tao-book} for examples), so the event $\mathcal{B} _\kappa $ holds with high probability, provided $\kappa$ is sufficiently large. 
\end{remark}

Compared to Theorem \ref{thm:generalized}, the estimate in Theorem \ref{thm:RVgapsgen} is not optimal. On the other hand, the probability bound 
in Theorem \ref{thm:RVgapsgen} is stronger  and thus the estimate is more applicable. It would be desirable to have a common strengthening of these two theorems. This can be done by achieving an extension of Theorem \ref{thm:con-max}. 

\vskip2mm 
\noindent{\bf Open question.} Extend Theorem \ref{thm:con-max} to eigenvectors of  Wigner matrices with non-gaussian entries. 
\vskip2mm 

We conclude this section with the following corollary, which is comparable to Theorem \ref{thm:GOElp}. 

\begin{corollary}[$\ell^p$-norm] \label{cor:lp}
Let $\xi$, $\zeta$ be real sub-gaussian random variables with mean zero, and assume $\xi$ has unit variance.  Let $W$ be an $n \times n$ Wigner matrix with atom variables $\xi$, $\zeta$.  Then, for any $1 \leq p \leq 2$, there exist constants $C, c, C_0, c_0 > 0$ (depending only on $p$ and the sub-gaussian moments of $\xi$ and $\zeta$) such that
$$ c_0 n^{1/p-1/2} \leq \min_{1 \leq j \leq n} \|v_j(W)\|_{\ell^p} \leq \max_{1 \leq j \leq n} \|v_j(W) \|_{\ell^p} \leq C_0 n^{1/p-1/2}  $$
with probability at least $1 - C \exp(-cn)$.
\end{corollary}


One-sided bounds for $\|v_j(W)\|_{\ell^p}$ were previously obtained in \cite{ESY1}.  We give a proof of Corollary \ref{cor:lp} in Section \ref{sec:proof:wigner}.


\section {Random symmetric matrices with non-zero mean} \label{sec:nonzero}
In this section, we consider random symmetric matrices with nonzero mean, which includes the adjacency matrices $A_n(p)$ and $\tilde{A}_n(p)$.  In view of Definition \ref{def:wigner}, we do not refer to these matrices as Wigner matrices.  However, such random symmetric matrices can be written as deterministic perturbations of Wigner matrices, and we will often take advantage of this fact.

 It has been observed that the unit eigenvector corresponding to the largest eigenvalue of the adjacency matrix $A_n(p)$ looks like the normalized all-ones vector $\frac{1}{\sqrt{n}} \Bj_n$, since the degree of the vertices are approximately the same (assuming $p \gg \log n/n$). 
For the rest of the spectrum, we expect the corresponding eigenvectors to be uniformly distributed on the unit sphere in $\mathbb{R}^n$.

\subsection{The largest eigenvector}  
The following result describes the unit eigenvector corresponding to the largest eigenvalue of $A_n(p)$ 

\begin{theorem}[Theorem 1 from \cite{Mitra}] \label{thm:Mitra}
Let $A_n(p)$ be the adjacency matrix of the random graph $G(n,p)$.  Let $v_n$ be a unit eigenvector corresponding to the largest eigenvalue of $A_n(p)$.  For $p \geq \frac{\log^6 n}{n}$, 
$$ \left\| v_n - \frac{1}{\sqrt{n}} \Bj_n \right\|_{\ell^\infty} \leq C \frac{\log n}{\log (np)} \frac{1}{\sqrt{n}} \sqrt{ \frac{\log n}{np}} $$
with probability $1 - o(1)$, for some constant $C > 0$.  
\end{theorem}

\subsection{Extremal coordinates} 
 Recall that $\tilde{G}(n,p)$ is the Erd\"os--R\'enyi random graph with loops on $n$ vertices and edge density $p$.  We have the following analogue of Theorem \ref{thm:delo1}.

\begin{theorem}[Theorem 2.16 from \cite{EKYY}] \label{thm:norm}  Let $\tilde{A}_n(p)$ be the adjacency matrix of the random graph $\tilde{G}(n,p)$, and let $\tilde{v}_i$ be the unit eigenvector corresponding to the $i$th smallest eigenvalue.  Fix $1+\eps_0 \le \alpha \le C_0 \log\log n$ for some constants $\eps_0, C_0>0$. Assume $c' \ge p \ge (\log n)^{6\alpha}/n$ for some constant $c'$. Then there exist constants $C,c>0$ (depending on $\eps_0$, $C_0$, and $c'$) such that 
\begin{align*}
	\max_{1\le i \le n-1} \|\tilde{v}_i\|_{\ell^\infty} \le \frac{(\log n)^{4\alpha}}{\sqrt{n}}
\end{align*}
and
\begin{align*}
	\left\|\tilde{v}_n -  \frac{1}{\sqrt{n}}\Bj_n\right\|_{\ell^\infty} \le C\frac{(\log n)^{\alpha}}{\sqrt{n}}\sqrt{ \frac{1}{np}}
\end{align*}
with probability at least $1-C \exp(-c(\log n)^{\alpha})$.
\end{theorem} 

An upper bound for the $\ell^\infty$-norm of the form $\frac{ \log^C n }{\sqrt n}$ was originally proven in \cite{TVW} for fixed values of $p$. 
Theorem \ref{thm:norm} above is an extension which applies to a wider range of values of $p$.  More generally, the results in \cite{EKYY} also apply to perturbed Wigner matrices.  

\begin{theorem}[Theorem 2.16 from \cite{EKYY}]
Let $\xi$ be a sub-exponential random variable with mean zero and unit variance.  Let $W$ be the $n \times n$ Wigner matrix with atom variable $\xi$.  Fix $\mu \in \mathbb{R}$, and consider the matrix $M := W + \mu  \BJ$, where $\BJ$ is the all-ones matrix.  Fix $1 + \eps_0 \leq \alpha \leq C_0 \log \log n$ for some constants $\eps_0, C_0 > 0$.  Then there exists $C, c > 0$ (depending on $\eps_0, C_0$, $\xi$, and $\mu$) such that
$$ \max_{1 \leq i \leq n-1} \| v_i(M) \|_{\ell^\infty} \leq \frac{ (\log n)^{4 \alpha} }{ \sqrt{n}} $$
and
$$ \left\| v_n(M) - \frac{1}{\sqrt{n}} \Bj_n \right\|_{\ell^\infty} \leq C \frac{ (\log n)^{\alpha}}{ n  } $$
with probability at least $1 - C \exp(- c (\log n)^{\alpha} )$.  
\end{theorem}

We next consider the smallest coordinates of each eigenvector of the adjacency matrix $A_n(p)$.   We prove the following  analogue of Theorem \ref{lower} 

\begin{theorem}[Individual coordinates: Lower bound] \label{thm:Gnpsmallcor}
Let $A_{n}(p)$ be the adjacency matrix of the random graph $G(n,p)$ for some fixed value of $p \in (0,1)$. Let $v_1, \ldots, v_n$ be the unit eigenvectors of $A_n(p)$, and let $v_i(j)$ denote the $j$th coordinate of $v_i$.  Then, for any $\alpha > 0$, there exist constants $C, c_1 > 0$ (depending on $p, \alpha$) such that, for any $\delta > n^{-\alpha}$,
$$ \sup_{1 \leq i \leq n-1} \sup_{1 \leq j \leq n} \Prob \left( |v_i(j)| \leq \frac{\delta}{\sqrt{n} (\log n)^{c_1}} \right) \leq C n^{o(1)} \delta + o(1). $$
Here, the rate of convergence to zero implicit in the $o(1)$ terms depends on $p, \alpha$.  
\end{theorem}

Notice that Theorem \ref{thm:Gnpsmallcor} does not address the eigenvector corresponding to the largest eigenvalue of $A_n(p$).  
Bounds for this eigenvector can easily be obtained using Theorem \ref{thm:Mitra} and are left as an exercise.

\subsection{No-gaps delocalization} 
In \cite{DLL}, the following result is developed as a tool  in the authors' study of nodal domain for the eigenvectors of $G(n,p)$. 
 
\begin{theorem}[Theorem 3.1 from \cite{DLL}]\label{thm:DLL}
Let $A_n(p)$ be the adjacency matrix of the random graph $G(n,p)$, and let $v_i$ be the unit eigenvector corresponding to the $i$th smallest eigenvalue. For every $p\in (0,1)$ and every $\varepsilon>0$, there exist $C, c, \eta> 0$ (depending on $\eps$ and $p$) such that, for every fixed subset $S\subset [n]$ of size $|S|\ge(\frac{1}{2}+\varepsilon)n$, 
$$ \min_{1\le i \le n-1} \|v_i\|_S \ge \eta $$
with probability at least $1-C\exp(-c n)$.
\end{theorem}

Theorem \ref{thm:DLL} is not nearly as strong as  Theorem \ref{thm:RVgapsgen} in Section \ref{sec:mass}, which holds for all subsets $S \subset [n]$ of specified size. 
 However, the  results from \cite{RVgaps} do not apply directly to  the adjacency matrix $A_n(p)$.  In particular, Theorem  \ref{thm:RVgapsgen} only applies when the spectral norm of the matrix is $O(\sqrt{n})$ with high probability. However, 
  the largest eigenvalue of 
 $A_n (p)$ is of order $\Theta (n)$. 

We conclude this section by discussing some variations  which do not require the spectral norm to be $O(\sqrt{n})$.  Specifically, we consider eigenvectors of matrices of the form $M := W + J$, where $W$ is a Wigner matrix and $J$ is a real symmetric deterministic low-rank matrix.  When $W$ is drawn from the GOE, it suffices to consider the case that $J$ is diagonal (since $W$ is invariant under conjugation by orthogonal matrices).  Thus, we begin with the case when $J$ is diagonal.  
\begin{theorem}[Diagonal low-rank perturbations] \label{thm:diag}
Let $\xi$, $\zeta$ be real sub-gaussian random variables with mean zero, and assume $\xi$ has unit variance.  Let $W$ be an $n \times n$ Wigner matrix with atom variables $\xi$, $\zeta$; let $J$ be $n \times n$ deterministic, diagonal real symmetric matrix with rank $k$.  Let $M:=W+J$, and consider $\lambda_j(M) \in [\lambda_1(W), \lambda_n(W)]$ and its corresponding unit eigenvector $v_j(M)$. For any constant $\delta \in (0,1)$ and any fixed set $S \subset [n]$ with size $|S| = \lfloor \delta n \rfloor$,  there exist constants $C,c > 0$ and $0 < \eta_1 \leq \eta_2 < 1$ (depending only on $\delta$, $k$, and the sub-gaussian moments of $\xi$ and $\zeta$) such that
$$ \eta_2 \geq \|v_j(M)\|_S \geq \eta_1 $$
with probability at least $1 -  C\exp \left( -c (\log n)^{c \log \log n} \right)$.
\end{theorem}

Theorem \ref{thm:diag} does not handle the eigenvectors corresponding to the extreme eigenvalues of $M$ due to the condition $\lambda_j(M) \in [\lambda_1(W), \lambda_n(W)]$.  In particular, it is possible for the eigenvectors corresponding to the largest and smallest eigenvalues to be \emph{localized} for certain choices of the diagonal matrix $J$.  A localized eigenvector is an eigenvector whose mass in concentrated on only a few coordinates.  We present an example of this phenomenon in the following theorem.

\begin{theorem} \label{thm:OVW}
Let $\xi$, $\zeta$ be real sub-gaussian random variables with mean zero and unit variance.  Then there exist constants $C, c > 0$ such that the following holds.  Let $W$ be an $n \times n$ Wigner matrix with atom variables $\xi$, $\zeta$.  Let $J = \sqrt{n} \theta e_k e_k^\mathrm{T}$ for some $\theta > C$ and $1 \leq k \leq n$, where $e_1, \ldots, e_n$ is the standard basis of $\mathbb{R}^n$.  Set $M := W + J$.  Then
$$ | v_n(M) \cdot e_k|^2 \geq 1 - \frac{C^2}{\theta^2} $$
with probability at least $1 - C \exp(-cn)$.  
\end{theorem}

Specifically, Theorem \ref{thm:OVW} implies that for $\theta$ large enough (say, $\theta = \Omega(\log n)$), most of the mass of the eigenvector $v_n(M)$ is concentrated on the $k$th coordinate.  This is in contrast to Theorem \ref{thm:diag}, which shows that, with high probability, the other eigenvectors cannot concentrate on a single coordinate.   Theorem \ref{thm:OVW} is part of a large collection of results concerning the extreme eigenvalues and eigenvectors of perturbed random matrices.  We refer the reader to \cite{BBP, BR, KY2, OVW, PSR, RS} and references therein for many other results.  

Theorem \ref{thm:OVW} follows as a simple corollary of a slightly-modified version of the Davis--Kahan $\sin \theta$ Theorem (see, for instance, \cite[Theorem 4]{OVW}) and a bound on the spectral norm of a Wigner matrix (Lemma \ref{lemma:norm}); we leave the details as an exercise. 

Theorems \ref{thm:diag} and \ref{thm:OVW} both deal with diagonal perturbations.  For more general perturbations, we have the following result.  

\begin{theorem}[General low-rank perturbations] \label{thm:perturb}
Let $\xi$, $\zeta$ be real sub-gaussian random variables with mean zero, and assume $\xi$ has unit variance.  Let $W$ be an $n \times n$ Wigner matrix with atom variables $\xi$, $\zeta$; let $J$ be $n \times n$ deterministic real symmetric matrix with rank $k$.  Let $\eps_1, \eps_0 > 0$.  Let $M:=W+J$, and consider $\lambda_j(M)$ such that $\eps_1 n \leq j \leq (1-\eps_1)n$ and its corresponding unit eigenvector $v_j(M)$. For any constant $\delta \in (0,1)$ and any fixed set $S \subset [n]$ with size $|S| = \lfloor \delta n \rfloor$,  there exist constants $C,c > 0$ and $0 < \eta < 1$ (depending only on $\delta$, $k$, $\eps_0$, $\eps_1$, and the sub-gaussian moments of $\xi$ and $\zeta$) such that 
$$ \Prob \left( \frac{1}{n^{1 - \eps_0}} \leq \| v_j(M) \|_S \leq \eta \right) \leq C \exp \left( -c (\log n)^{c \log \log n} \right). $$
\end{theorem}
\begin{remark}
Theorem \ref{thm:perturb} shows that, with high probability, either $\|v_j(M)\|_S \geq \eta$ or $\|v_j(M)\|_S \leq \frac{1}{n^{1 - \eps_0}}$.  Based on the previous results, we do not expect the later case to be a likely event.  However, it appears additional structural information about $J$ (for instance, that $J$ is diagonal as in Theorem \ref{thm:diag}) is required to eliminate this possibility.  See the proof of Theorem \ref{thm:perturb} for additional details.  
\end{remark}

In the case when $J$ has rank one, we have the following.  

\begin{theorem}[Rank one perturbations] \label{thm:singlerank1}
Let $\xi$, $\zeta$ be real sub-gaussian random variables with mean zero and unit variance, and fix $\eps_1 > 0$.  Then there exist constants $C, c > 0$ such that the following holds.  Let $W$ be an $n \times n$ Wigner matrix with atom variables $\xi$, $\zeta$.  Suppose $J = \theta u u^\mathrm{T}$, where $\theta \in \mathbb{R}$ and $u \in \mathbb{R}^n$ is a unit vector.  Set $M := W + J$.   Then, for any integer $j$ with $\eps_1 n \leq j \leq (1-\eps_1) n$, 
$$ | v_j(M) \cdot u | \leq \frac{C (\log n)^{c \log \log n}}{|\theta|} $$
with probability $1 - o(1)$.  
\end{theorem}

When $\theta := \mu n$ and $u := n^{-1/2} \Bj_n$, $J$ becomes the matrix in which every entry takes the value $\mu$.  In this case, the entries of $W + J$ have mean $\mu$ instead of mean zero.  Thus, applying Theorem \ref{thm:singlerank1}, we immediately obtain the following corollary.  

\begin{corollary}[Wigner matrices with non-zero mean]
Let $\xi$, $\zeta$ be real sub-gaussian random variables with mean zero and unit variance, and fix $\eps_1 > 0$.  Then there exist constants $C, c > 0$ such that the following holds.  Let $\mu \in \mathbb{R}$ with $\mu \neq 0$.  Let $W$ be an $n \times n$ Wigner matrix with atom variables $\xi$, $\zeta$.  Set $M := W + \mu \BJ$, where $\BJ$ is the all-ones matrix.  Then, for any integer $j$ with $\eps_1 n \leq j \leq (1-\eps_1) n$, 
$$ | v_j(M) \cdot \Bj_n | \leq \frac{C (\log n)^{c \log \log n}}{|\mu| \sqrt{n}} $$
with probability $1 - o(1)$.  
\end{corollary}

We prove Theorems \ref{thm:Gnpsmallcor}, \ref{thm:diag}, \ref{thm:perturb} and \ref{thm:singlerank1} in Section \ref{sec:proof:nonzero}.

\section{Localized eigenvectors: Heavy-tailed and band random matrices} \label{sec:ht}

As we saw in Theorem \ref{thm:OVW}, the eigenvectors corresponding to the extreme eigenvalues of a perturbed Wigner matrix can be \emph{localized}, meaning that most of the mass is contained on only a few coordinates.  For instance, in Theorem \ref{thm:OVW}, most of the mass was contained on a single coordinate.  We now discuss a similar phenomenon for heavy-tailed and band random matrices.  

\subsection{Heavy-tailed random matrices}

Most of the results from the previous sections required the atom variables $\xi$, $\zeta$ to be sub-exponential or sub-gaussian.  In particular, these conditions imply that $\xi$ and $\zeta$ have finite moments of all orders.  In other words, the atom variables have very light tails.  We now consider the case when the atom variables have heavy tails, such as when $\xi$ and $\zeta$ have only one or even zero finite moments.  In this case, the eigenvectors corresponding to the largest eigenvalues behave very differently than predicted by the results above.  

\begin{theorem}[Theorem 1.1 from \cite{BP}] \label{thm:ht}
Let $\xi$ be a real random variable satisfying
$$ \Prob (|\xi| \geq x) = L(x) x^{-\alpha} $$
for all $x > 0$, where $0 < \alpha < 2$ and $L:(0,\infty) \to (0,\infty)$ is a slowly varying function, i.e., for all $t > 0$,
$$ \lim_{x \to \infty} \frac{ L(tx)}{L(x)} = 1. $$
For each $n \geq 1$, let $W_n$ be an $n \times n$ Wigner matrix with atom variable $\xi$.  Fix an integer $k \geq 0$.  Then, for every fixed $\eps > 0$, 
$$ \| v_{n-k}(W_n) \|_{\ell^\infty} \geq \frac{1}{\sqrt{2}} - \eps $$
with probability $1 - o(1)$.  In addition,
$$ \min_{S \subset [n] : |S| = n-2} \| v_{n-k}(W_n) \|_S \longrightarrow 0 $$
in probability as $n \to \infty$.  
\end{theorem}
\begin{remark}
Theorem \ref{thm:ht} also holds when $2 \leq \alpha < 4$ provided the atom variable $\xi$ is symmetric; see \cite[Theorem 1.1]{BP} for details.  
\end{remark}

Theorem \ref{thm:ht} shows that the eigenvectors corresponding to the largest eigenvalues of $W_n$ are concentrated on at most two coordinates.  This is considerably different than the cases discussed above when $\xi$ has light tails.  

Let us try to explain this phenomenon based on the tail behavior of $\xi$.  It is well known that the largest entry of an $n \times n$ Wigner matrix with sub-gaussian entries is $O(\sqrt{\log n})$ with high probability.  However, when the tails are heavy, the maximum entry of $W_n$ can be significantly larger.  It was observed by Soshnikov \cite{Sosh, Sosh2} that, in this case, the largest eigenvalues behave like the largest entries of the matrix.  Intuitively, the eigenvector corresponding to the largest eigenvalue of $W_n$ should localize on the coordinates which match the largest entry.  Since $W_n$ is symmetric, the largest entry can appear at most twice.  Hence, we expect this eigenvector to be concentrated on at most two coordinates.  This heuristic has led to a number of results regarding the eigenvalues and eigenvectors of heavy-tailed Wigner matrices; we refer the reader to \cite{ABAP, BG, BP,BorG, BorG2, CB, Sosh, Sosh2} and references therein for further details and additional results.  

\subsection{Random band matrices}

The standard basis elements $e_1, \ldots, e_n$ of $\mathbb{R}^n$ are always eigenvectors of an $n \times n$ diagonal matrix.  In other words, the eigenvectors of a diagonal matrix are localized.  Band matrices generalize diagonal matrices by only allowing the entries on and near the diagonal to be non-zero while requiring the other entries, away from the diagonal, to be zero.  

We can form random band matrices from Wigner matrices.  Indeed, let $W$ be an $n \times n$ Wigner matrix with atom variables $\xi$, $\zeta$.  For simplicity, let us assume that $\xi$ and $\zeta$ are sub-gaussian random variables.  We can form an $n \times n$ random band matrix $T$ from $W$ with band width $L \geq 1$ by replacing the $(i,j)$-entry of $W$ by zero if and only if $|i-j| \geq L$.  Hence, the $(i,j)$-entry of $T$ is just the $(i,j)$-entry of $W$ when $|i-j| < L$.  A random band matrix with width $n$ is a Wigner matrix, and a random band matrix with band width $1$ is a diagonal matrix.  Thus, we expect a transition in the eigenvector behavior when the band width $L$ interpolates between $1$ and $n$.  Indeed, it is conjectured that for $L$ significantly smaller than $\sqrt{n}$, the eigenvectors will be localized (with localization length on the order of $L^2$).  On the other hand, for $L$ sufficiently larger than $\sqrt{n}$, it is expected that the eigenvectors of $T$ behave more like the eigenvectors of $W$.  Some partial results in this direction have been established in \cite{EK, EK2, EKYY2, Sband}.  

While random band matrices can be constructed from Wigner matrices, they are, in general, not Wigner matrices, and we will not focus on them here.  We refer the interested reader to \cite{BP, EK, EK2, EKYY2, Sband, Soxford} and references therein for results concerning the spectral properties of random band matrices.  In the discussion above, we have focused on the case when the atom variables $\xi$ and $\zeta$ are sub-gaussian.  However, Theorem \ref{thm:ht} can be extended to random band matrices constructed from heavy-tailed Wigner matrices; see \cite{BP} for details.

\section{Singular vectors and eigenvectors of non-Hermitian matrices} \label{sec:others}

In this section, we consider the singular vectors and eigenvectors of non-Hermitian random matrices. 

Let $M$ be a $p \times n$ matrix with real entries.  Recall that the singular values of $M$ are the square roots of the eigenvalues of $M M^\mathrm{T}$.  The left singular vectors are the eigenvectors of $M M^\mathrm{T}$, and the right singular vectors are the eigenvectors of $M^\mathrm{T} M$.  Following our previously introduced notation, we will write 
$$ \sqrt{ \lambda_1(M M^\mathrm{T})}, \ldots, \sqrt{ \lambda_p ( M M^\mathrm{T}) } $$ 
to denote the singular values and 
$$ v_1( M M^\mathrm{T}), \ldots, v_p(M M^\mathrm{T}) $$ 
to denote the left singular vectors of $M$.  

Let $M=M_{p,n}=(\zeta_{ij})_{1\le i\le p, 1\le j \le n}$ be a random matrix (more specifically, a sequence of random matrices) whose entries are independent real random variables with mean zero and unit variance.  Assume $\lim_{n\to \infty} p/n =y$ for some $y\in (0,1]$.  The delocalization properties of the singular vectors of $M_{p,n}$ have been explored in \cite{CAB, PY, TVcovariance, VW, W-edge}. The optimal bound of $O(\sqrt{\log n/n})$ for the $\ell^\infty$-norm was obtained recently in \cite{VW}.

\begin{theorem}[Delocalization of singular vectors, Theorem B.3 from \cite{VW}]  \label{thm:singularvector} 
Let $\zeta$ be a sub-gaussian random variable with mean zero and unit variance.  Then, for any $C_1 > 0$ and any $0 < \eps < 1$, there exists a constant $C_2 > 0$ such that the following holds.   Assume the entries of $M_{p,n}=(\zeta_{ij})_{1\le i\le p, 1\le j \le n}$ are iid copies of $\zeta$.  Let $a := (1-\sqrt{y})^2$ and $b := (1+\sqrt{y})^2$. 
\begin{itemize}
\item (Bulk case) For any $1 \leq i \leq n$ such that $\frac{1}{n} \lambda_i(M_{p,n} M_{p,n}^\mathrm{T}) \in [a + \eps, b - \eps]$, there is a corresponding left singular vector $v_i(M_{p,n} M_{p,n}^\mathrm{T})$ such that
$$ \| v_i( M_{p,n} M_{p,n}^\mathrm{T}) \|_{\ell^\infty} \leq C_2 \sqrt{\frac{\log n}{{n}}} $$
with probability at least $1 -  n^{-C_1}$.  The same also holds for right singular vectors. 
\item (Edge case) For any $1 \leq i \leq n$ such that $\frac{1}{n} \lambda_i( M_{p,n} M_{p,n}^\mathrm{T}) \in [a - \eps, a + \eps] \cup [b - \eps, b + \eps]$ if $a \neq 0$ and $\frac{1}{n} \lambda_i( M_{p,n} M_{p,n}^\mathrm{T}) \in [4 - \eps, 4]$ if $a = 0$, there is a corresponding left singular vector $v_i(M_{n,p} M_{n,p}^\mathrm{T})$ such that
$$ \| v_i(M_{n,p} M_{n,p}^\mathrm{T}) \|_{\ell^\infty} \leq C_2\frac{\log n}{\sqrt{n}} $$
with probability at least $1 - n^{-C_1}$.  The same also holds for right singular vectors. 
\end{itemize}
\end{theorem} 

Similar to Theorem \ref{thm:delo1} this was first proved under the stronger assumption that the entries of the matrix are bounded, but one can obtain this version using the same argument as in Remark \ref{nobound}. 
The analogue  of Theorem \ref{thm:singularvector} for the eigenvectors of $M_{n,n}$ was recently proved in \cite{RVeigen}, using a completely different method. 

\begin{theorem}[Delocalization of eigenvectors;  Theorem 1.1 from \cite{RVeigen}] Let $M=(\zeta_{ij})_{1\le i,j \le n}$ be an $n\times n$ random matrix whose entries are 
independent real random variables with mean zero, variance at least one, and 
$$ \sup_{p \ge 1} p^{-1/2} (\E |\zeta_{ij}|^p)^{1/p} \le K $$ 
for all $1 \leq i,j \leq n$.  Then, for any $t \geq 2$, with probability at least $1-n^{1-t}$, all unit eigenvectors $v$ of $M$ satisfy $$\|v\|_{\ell^\infty} \le \frac{C t^{3/2} 
\log^{9/2}n}{\sqrt{n}}, $$
where $C$ depends only on $K$.
\end{theorem}
\begin{remark}
The above result holds for more general matrix ensembles, e.g. random matrices with independent sub-exponential entries; see \cite[Corollary 1.5]{RVeigen} for details. 
\end{remark}

\section{Random regular graphs} \label{sec:rrg}

We now turn to the eigenvectors of the adjacency matrix of random regular graphs. Recall that a regular graph is a simple graph where each vertex has the same degree, and a $d$-regular graph is a regular graph with vertices of degree $d$.  It is well-known that a $d$-regular graph on $n$ vertices exists if and only if $n \geq d+1$ and $nd$ is even.  Let $G_{n,d}$ denote a random $d$-regular graph chosen uniformly from all $d$-regular graphs on $n$ vertices.
 It is easy to see that the adjacency matrix of $G_{n,d}$ has a trivial eigenvector $\frac{1}{\sqrt{n}} \Bj_n$ corresponding to the eigenvalue $d$. Further, it has been conjectured that every non-trivial unit eigenvector behaves like a uniform vector on the unit sphere.

In the combinatorics/computer science literature, the most interesting case to consider is when $d$ is a constant. 
This also seems to be the most difficult case to study.  In this case, the strongest delocalization result known to the authors is the following, which is a corollary of \cite[Theorem 1]{BL}

\begin{theorem}[Theorem 1 from \cite{BL}] \label{thm:BL}
Let $d$ be fixed and $\eps>0$.  Then there is a constant $\delta >0$ (depending on $d$ and $\eps$) such that the following holds. With probability $1- o(1)$, for any unit eigenvector $v$ of the adjacency matrix of $G_{n,d}$, 
any subset $S \subset [n]$ satisfying $$\|v\|_S^2 > \eps$$ must be of 
 size $|S| \ge n^{\delta}.$ \end{theorem}

From \cite{DP}, one  has the following result  for the eigenvectors in the case when $d$ grows slowly with the vertex size $n$.  
\begin{theorem}[Theorem 3 from \cite{DP}] Fix $\delta>0$. Let $d := d_n = (\log n)^\gamma$ for $\gamma>0$. Let $\eta_n := (r_n-r_n^{-1})/2$ where $r_n := \exp(d_n^{-\alpha})$ for some $0< \alpha <\min(1, \gamma^{-1})$. Let $T_n \subset [n]$ be a deterministic set of size $L_n = o(\eta_n^{-1})$. Let $\Omega_n$ be the event that some unit eigenvector $v$ of the adjacency matrix of $G_{n,d}$ satisfies $\|v\|_{T_n}^2 \ge 1-\delta$. Then, for all sufficiently large $n$,
$$\Prob(\Omega_n^c) \ge e^{-L_n \eta_n/d_n} \left(1-o(d_n^{-1}) \right)=1-o(d_n^{-1}).$$
\end{theorem}

Finally, let us mention the following recent result from \cite{BKY}, which provides a near optimal  bound 
when $d$ grows sufficiently fast with $n$.

\begin{theorem}[Corollary 1.2 from \cite{BKY}] 
There exist constants $C, C' > 0$ such that the following holds.  Let $C^{-1} \log^4 n \le d \le C n^{2/3} \log^{-4/3}n$. Then any unit eigenvector $v$ of the adjacency matrix of $G_{n,d}$ satisfies 
$$\|v\|_{\ell^\infty} \leq C' \frac{\log^2 n}{\sqrt{n}} $$
with probability at least $1-\exp(-2\log^2 n\cdot\log \log n)$.
\end{theorem}

\begin{remark}
More generally, the results in \cite{BKY} hold for other models of random regular graphs besides the uniform model $G_{n,d}$ discussed here; see \cite[Section 1.2]{BKY} for details.  
\end{remark}

\section{Proofs for the Gaussian orthogonal ensemble} \label{sec:proof:GOE}

In order to prove the results in Section \ref{sec:GOE}, we will need the following characterization of a unit vector uniformly distributed on the unit sphere $S^{n-1}$.   

\begin{lemma} \label{lemma:distr}
Let $v$ be a random vector uniformly distributed on the unit sphere $S^{n-1}$.  Then $v$ has the same distribution as 
$$ \left( \frac{ \xi_1}{\sqrt{\sum_{j=1}^n \xi_j^2}}, \ldots, \frac{ \xi_n}{\sqrt{\sum_{j=1}^n \xi_j^2}}\right) $$
where $\xi_1, \ldots, \xi_n$ are iid standard normal random variables.  
\end{lemma}
\begin{proof}
The claim follows from the fact that the Gaussian vector $(\xi_k)_{k=1}^n$ is rotationally invariant.  We refer the reader to \cite{J} for further details and other interesting results regarding entries of uniformly distributed unit vectors, and, more generally, results concerning entries of orthogonal matrices distributed according to Haar measure.  
\end{proof}

We now prove Theorem \ref{thm:GOEnorm}.  In order to do so, we will need the following result from \cite{LM}.  
\begin{lemma}[Lemma 1 from \cite{LM}] \label{lemma:chi}
Suppose $X$ is a $\chi^2$-distributed with $k$ degrees of freedom.  Then, for all $t > 0$, 
$$ \Prob( X - k \geq 2 \sqrt{kt} + 2t) \leq \exp(-t) $$
and
$$ \Prob(k - X \geq 2 \sqrt{kt} ) \leq \exp(-t). $$ 
\end{lemma}

\begin{proof}[Proof of Theorem \ref{thm:GOEnorm}] In view of Lemma \ref{lemma:distr}, it suffices to assume that
$$ v := \left( \frac{ \xi_1}{\sqrt{\sum_{j=1}^n \xi_j^2}}, \ldots, \frac{ \xi_n}{\sqrt{\sum_{j=1}^n \xi_j^2}}\right), $$
where $\xi_1, \ldots, \xi_n$ are iid standard normal random variables.  

We first verify \eqref{eq:GOEmax}.  Let $C > 1$ and $c_1=1/C<1$.  Define the events
$$ \Omega_1 := \left\{ \sqrt{ \sum_{j=1}^n \xi_j^2 } \geq c_1\sqrt{n} \right\} $$
and
$$ \Omega_2 := \left\{ \max_{1 \leq i \leq n} |\xi_i| \leq \sqrt{2C \log n} \right\}. $$
In order to verify \eqref{eq:GOEmax}, it suffices to show $\Omega_1 \cap \Omega_2$ holds with probability at least $1-2 n^{1 - C}-\exp(-\frac{(C-1)^2}{4C^2}n)$.  

As $\sum_{j=1}^n \xi_j^2$ is $\chi^2$-distributed with $n$ degrees of freedom, Lemma \ref{lemma:chi} implies that
$$ \Prob(\Omega_1^c) = \Prob \left( \sum_{j=1}^n \xi_j^2 < c_1^2 n \right) \leq \exp\left(-\frac{(1-c_1^2)^2}{4} n\right) \le \exp\left(-\frac{(1-c_1)^2}{4}n \right).$$

Since $\xi_1$ is a standard normal random variable, it follows that, for every $t \geq 0$, 
\begin{equation} \label{eq:normalbound}
	\Prob \left( |\xi_1| > t \right) \leq 2 e^{-t^2/2}; 
\end{equation}
this bound can be deduced from the exponential Markov inequality; see also \cite[Section 5.2.3]{RV}.  Thus, from \eqref{eq:normalbound}, we have
\begin{align*}
	\Prob \left( \Omega_2^c \right) \leq n \Prob(|\xi_1| > \sqrt{2C \log n}) \leq 2n \exp \left(-C \log n \right) = 2 n^{1 - C}.
\end{align*}
Combining the bounds above yields 
$$ \Prob( \Omega_1^c \cup \Omega_2^c) \leq \exp\left(-\frac{(1-c_1)^2}{4}n \right) + 2 n^{1 - C} $$
and the proof of \eqref{eq:GOEmax} is complete.  

We now verify \eqref{eq:GOEmin}.  Let $0 \leq c < 1$ and $a>1$.  Define the events
$$ \Omega_3 := \left\{ \sqrt{ \sum_{j=1}^n \xi_j^2 } \leq a \sqrt{n}\right\} $$
and
$$ \Omega_4 := \left\{ \min_{1 \leq i \leq n} |\xi_i| \geq \frac{c}{n} \right\}. $$
It suffices to show that $\Omega_3 \cap \Omega_4$ holds with probability at least 
$$ \exp \left( -2c \right) - \exp \left( - \frac{ a^2 - \sqrt{2 a^2 -1} }{2} n \right). $$ 

From Lemma \ref{lemma:chi}, we again find
\begin{equation} \label{eq:omega3}
	\Prob(\Omega_3^c) \leq \exp\left(-\frac{(\sqrt{2a^2-1}-1)^2}{4} n\right) = \exp \left( - \frac{ a^2 - \sqrt{2 a^2 -1} }{2} n \right). 
\end{equation}  

Since $\xi_1$ is a standard normal random variable, we have
\begin{equation*}
	\Prob \left( |\xi_1| \leq \frac{c}{n} \right) = \frac{2}{\sqrt{2 \pi}} \int_{0}^{c n^{-1}} e^{-t^2/2} dt \leq \int_0^{cn^{-1}} dt = \frac{c}{n}. 
\end{equation*}
Hence, we obtain
\begin{align*}
	\Prob \left( \Omega_4 \right) &= \left( 1 - \Prob \left(|\xi_1| < \frac{c}{n}\right) \right)^{n} \\
	&\geq \left( 1 - \frac{c}{n} \right)^n \\
	&= \exp \left( n \log \left(1 - \frac{c}{n} \right) \right).
\end{align*}
By expanding the Taylor series for $\log (1-x)$, it follows that
$$ \log \left(1 - \frac{c}{n} \right) \geq - \frac{c}{n} - \frac{c^2}{n^2} \frac{1}{1 - \frac{c}{n}}, $$
and hence
\begin{equation} \label{eq:omega4}
	\Prob (\Omega_4) \geq \exp(-c) \exp \left( - \frac{c^2}{n-c} \right) \geq \exp(-2c),
\end{equation}
for $n \geq 2$.

Since 
$$ \Prob(\Omega_3 \cap \Omega_4) \geq 1 - \Prob(\Omega_3^c) - \Prob(\Omega_4^c) = \Prob(\Omega_4) - \Prob(\Omega_3^c), $$
we apply \eqref{eq:omega3} and \eqref{eq:omega4} to conclude that
\begin{align*}
	\Prob(\Omega_3 \cap \Omega_4) \geq \exp(-2c) - \exp \left( - \frac{ a^2 - \sqrt{2 a^2 -1} }{2} n \right), 
\end{align*}
as desired.  
\end{proof}

We now prove Theorem \ref{thm:GOElp} using Lemma \ref{lemma:distr} and the law of large numbers.

\begin{proof}[Proof of Theorem \ref{thm:GOElp}]
In view of Lemma \ref{lemma:distr}, it suffices to assume that
$$ v := \left( \frac{ \xi_1}{\sqrt{\sum_{j=1}^n \xi_j^2}}, \ldots, \frac{ \xi_n}{\sqrt{\sum_{j=1}^n \xi_j^2}}\right), $$
where $\xi_1, \ldots, \xi_n$ are iid standard normal random variables. 

Fix $p \geq 1$, and define $c_p := \E|\xi_1|^p$.  Then
\begin{align*}
	\frac{n^{p/2}}{n} \|v \|_{\ell^p}^p - c_p &= \left( \left( \frac{n}{\sum_{j=1}^n |\xi_j|^2} \right)^{p/2} - 1 \right) \frac{1}{n} \sum_{i=1}^n |\xi_i|^p + \frac{1}{n} \sum_{i=1}^n \left( |\xi_i|^p - c_p \right).
\end{align*}
By the law of large numbers, it follows that almost surely
$$ \lim_{n \to \infty} \frac{1}{n} \sum_{i=1}^n |\xi_i|^p = c_p \quad \text{and} \quad \lim_{n \to \infty} \frac{1}{n} \sum_{j=1}^n |\xi_j|^2 = 1. $$
Hence, we conclude that almost surely
$$ \lim_{n \to \infty} \left| \frac{n^{p/2}}{n} \|v \|_{\ell^p}^p - c_p \right| = 0, $$
and the claim follows.  
\end{proof}

Let $S \subset [n]$.  It follows from Lemma \ref{lemma:distr} that $\|v\|_{S}^2$ has the same distribution as 
\begin{equation} \label{eq:xichar}
	\frac{\sum_{i=1}^{|S|} \xi_i^2}{ \sum_{j=1}^n \xi_j^2}, 
\end{equation}
where $\xi_1, \ldots, \xi_n$ are iid standard normal random variables.  
In particular, we observe that, for any $1 \leq k \leq n$, $\sum_{j=1}^k \xi_j^2$ is $\chi^2$-distributed with $k$ degrees of freedom.  Thus, the random variable in \eqref{eq:xichar} can be expressed as $\frac{X}{X+W}$ where $X$ and $W$ are independent $\chi^2$-distributed random variables with $|S|$ and $n-|S|$ degrees of freedom respectively.  Theorem \ref{thm:distr} follows from computing the distribution of this ratio; see \cite{S} for details.  

Theorem \ref{thm:clt} is an immediate consequence of Theorem \ref{thm:distr} and the following lemma. 

\begin{lemma}\label{lem:beta}
Let $(\alpha_n)_{n = 1}^\infty$ and $(\beta_n)_{n=1}^\infty$ be sequences of positive integers which satisfy
\begin{enumerate}[(i)]
\item $\alpha_n  \to \infty$ and $\beta_n \to \infty$ as $n \to \infty$, 
\item $\frac{\alpha_n}{\alpha_n + \beta_n}$ converges to a limit in $[0,1]$ as $n \to \infty$.  
\end{enumerate}
For each $n \geq 1$, let $X_n \sim \mathrm{Beta}\left( \frac{\alpha_n}{2}, \frac{\beta_n}{2} \right)$.  Then
$$ \sqrt{ \frac{ (\alpha_n + \beta_n)^3}{2 \alpha_n \beta_n} } \left( X_n - \frac{\alpha_n}{\alpha_n + \beta_n} \right) \longrightarrow N(0,1) $$
in distribution as $n \to \infty$.  
\end{lemma}

We present the proof of Lemma \ref{lem:beta} in Appendix \ref{appendix:beta}.  We now prove Theorem \ref{thm:gaussian}.   
\begin{proof}[Proof of Theorem \ref{thm:gaussian}]
Let $m:= |S|$.  In view of Lemma \ref{lemma:distr}, it suffices to bound 
$$ \left| \frac{\sum_{i=1}^m \xi_i^2}{\sum_{j=1}^n \xi_j^2} - \frac{m}{n} \right|, $$
where $\xi_1, \ldots, \xi_n$ are iid standard normal random variables.  Fix $t > 0$, and define the events
\begin{align*}
	\Omega_1 &:= \left\{ \sum_{j=1}^n \xi_j^2 \geq \frac{n}{2} \right\}, \\
	\Omega_2 &:= \left\{ \left| \sum_{j=1}^n \xi_j^2 - n \right| \leq 2 \sqrt{nt} + 2t \right\}, \\
	\Omega_3 &:= \left\{ \left| \sum_{i=1}^{m} \xi_i^2 - m \right| \leq 2 \sqrt{mt} + 2t \right\}.
\end{align*}

By the union bound and Lemma \ref{lemma:chi}, we have
\begin{align*}
	\Prob( \Omega_1^c \cup \Omega_2^c \cup \Omega_3^c) \leq \exp( -c n)  + 4 \exp(-t)
\end{align*}
for some absolute constant $c>0$.  

On the event $\Omega_1 \cap \Omega_2 \cap \Omega_3$, we observe that
\begin{align*}
	\left| \frac{\sum_{i=1}^m \xi_i^2}{\sum_{j=1}^n \xi_j^2} - \frac{m}{n} \right| &\leq \frac{ \left| \sum_{i=1}^m \xi_i^2 - m \right|}{ \sum_{j=1}^n \xi_j^2} + \left| \frac{m}{\sum_{j=1}^n \xi_j^2} - \frac{m}{n} \right| \\
	&\leq \frac{2}{n} \left| \sum_{i=1}^m \xi_i^2 - m \right| + \frac{m}{n} \left| \frac{n}{\sum_{j=1}^n \xi_j^2} - 1 \right| \\
	&\leq \frac{2}{n} \left( 2 \sqrt{mt} + 2t \right) + 2 \frac{m}{n^2} \left| n - \sum_{j=1}^n \xi_j^2 \right| \\
	&\leq \frac{4}{n} \left( \sqrt{mt} + t \right) + 4 \frac{m}{n^2} \left( \sqrt{nt} + t \right).
\end{align*}
Since $m \leq n$, the claim follows.  
\end{proof}

We conclude this section with the proofs of Theorems \ref{thm:order} and \ref{thm:con-max}.

\begin{proof}[Proof of Theorem \ref{thm:order}]
In view of Lemma \ref{lemma:distr}, it suffices to consider the vector
$$ v_n := \frac{1}{ \sqrt{ \sum_{j=1}^n \xi_j^2 } } \left( \xi_1, \ldots, \xi_n \right), $$
where $\xi_1, \xi_2, \ldots$ are iid standard normal random variables.  Thus, $\xi_1^2, \xi_2^2, \ldots$ are iid $\chi^2$-distributed random variables with one degree of freedom.  Recall that $F$ is the cumulative distribution function of $\xi_1^2$, $Q$ is the quantile function of $F$ defined in \eqref{eq:def:Q}, and $H$ is defined in \eqref{eq:def:H}.

Let $\xi_{(n,1)}^2  \leq \cdots \leq \xi_{(n,n)}^2$ denote the order statistics based on the sample $\xi_1^2, \ldots, \xi_n^2$.  Let $m_n := \lfloor \delta n \rfloor$.  Define
$$ S_{n,m_n} := \sum_{i=1}^{m_n} \xi_{(n,n-m_n+i)}^2 $$ 
to be the partial sum of the largest $m_n$ entries, and set 
$$ S_n := \sum_{i=1}^n \xi_i^2. $$

We observe that 
\begin{equation} \label{eq:charlarge}
	\max_{S \subset [n] : |S|=m_n} \|v_n\|_S^2 = \frac{S_{n,m_n}}{S_n} .
\end{equation}

Since $F$ is a special case of the gamma distribution, it follows from \cite[Theorem 1.1.8]{deHaan} and \cite[Corollary 2]{CHM} that
\begin{equation}\label{eq:trim}
\frac{S_{n,m_n}-\mu_n}{\sqrt{n}a_n} \longrightarrow N(0,1)
\end{equation}
in distribution as $n \to \infty$, where 
$$ \mu_n :=-n \int_{1/n}^{m_n/n} H(u)~ du - H\left(\frac{1}{n}\right) $$ 
and
$$ a_n := \left(\int_{1/n}^{m_n/n} \int_{1/n}^{m_n/n} (\min(u,v)-uv)~dH(u)dH(v) \right)^{1/2}. $$

Since 
$$ \lim_{x\to \infty} \frac{1-F(x)}{\sqrt{\frac{2}{\pi}} x^{-1/2} \exp(-x/2)} = 1, $$
we apply \cite[Proposition 1]{taka} to obtain that $H(\frac{1}{n}) \leq C \log n$ for some constant $C>0$. Thus, we have 
$$\lim_{n \to \infty} \frac{\mu_n}{n} = - \int_{0}^{\delta} H(u)~ du $$
and
$$\lim_{n \to \infty} a_n = \left(\int_{0}^{\delta} \int_{0}^{\delta} (\min(u,v)-uv)~dH(u)dH(v) \right)^{1/2}.$$
From \eqref{eq:trim}, we conclude that
$$ \frac{S_{n,m_n}}{n} \longrightarrow - \int_{0}^{\delta} H(u)~ du $$
in probability as $n \to \infty$.  

On the other hand, by the law of large numbers, we have
$$ \frac{S_{n}}{n} \longrightarrow 1 $$
in probability as $n \to \infty$.  Thus, by Slutsky's theorem (see Theorem 11.4 in \cite[Chapter 5]{G}) and \eqref{eq:charlarge}, we have
$$ \max_{S \subset [n] : |S|=m_n} \|v_n\|_S^2 = \frac{S_{n,m_n}}{n} \cdot \frac{n}{S_n} \longrightarrow - \int_{0}^{\delta} H(u)~ du $$
in probability as $n \to \infty$.  

For the minimum, we note that 
$$ \min_{S \subset [n]: |S|=m_n} \|v_n\|_S^2 = 1-  \max_{T \subset [n]:|T|=n-m_n} \|v_n\|_T^2, $$
and hence
$$ 1-  \max_{T \subset [n]:|T|=n-m_n} \|v_n\|_T^2 \longrightarrow 1+ \int_{0}^{1- \delta} H(u)~ du= - \int_{1-\delta}^{1} H(u)~ du $$
in probability as $n \to \infty$.  Here, we used the fact that $\int_0^1 Q(s) ~ds = - \int_{0}^1 H(u)~du=1$.
\end{proof}

\begin{proof}[Proof of Theorem \ref{thm:con-max}]
We first observe that it suffices to prove \eqref{eq:conc-max}.  Indeed, \eqref{eq:conc-min} follows immediately from \eqref{eq:conc-max} by applying the identity
$$ \max_{S \subset [n]: |S|=m} \|v\|_S + \min_{S \subset [n]: |S|=n-m} \|v\|_S = 1.  $$ 

Let $1 \leq m \leq n$, and define the function $F: \mathbb{R}^n \to \mathbb{R}$ by 
$$ F(X) :=  \max_{S \subset [n]: |S|=m} \|X\|_S. $$ 
Clearly, $|F(v) - \E F(v)| \leq 2 \| v \| = 2$.  Thus, it suffices to show 
$$ \Prob \left( |F(v) - \E F(v)| > t \right) \leq C \exp(-c t^2 n) $$
for all $t \in [0,2]$ (as opposed to all $t \geq 0$).  In addition, by taking the absolute constant $C$ sufficiently large, it suffices to prove
\begin{equation} \label{eq:conc-show}
	\Prob \left( |F(v) - \E F(v)| > t \right) \leq C \exp(-c t^2 n) 
\end{equation}
for all $t \in \left( 24 n^{-1/2}, 2 \right]$.  

Clearly, $F$ is $1$-Lipschitz.  Thus, by L\'{e}vy's lemma (see, for example, \cite[Theorem 14.3.2]{M}), we have, for all $0 \leq t \leq 1$, 
\begin{equation} \label{eq:levy}
	\Prob \left( |F(v) - \med F(v)| > t \right) \leq 4 \exp \left( - t^2 n / 2 \right). 
\end{equation}
Here $\med F(v)$ denotes the median of the random variable $F(v)$.  In addition, \cite[Proposition 14.3.3]{M} implies that
\begin{equation} \label{eq:medexp}
	|\med F(v) - \E F(v)| \leq \frac{12}{\sqrt{n}}. 
\end{equation}

We will use \eqref{eq:medexp} to replace the median appearing in \eqref{eq:levy} with expectation.  Indeed, assume $24 n^{-1/2} < t \leq 2$.  Then, from \eqref{eq:medexp}, we have
\begin{align*}
	\Prob \left( |F(v) - \E F(v)| > t \right) &\leq  \Prob \left( |F(v) - \med F(v)| > t - | \E F(v) - \med F(v)| \right) \\
		&\leq \Prob \left( |F(v) - \med F(v)| > \frac{t}{2} \right).
\end{align*}
The bound in \eqref{eq:conc-show} now follows by applying \eqref{eq:levy}.  
\end{proof}

\section{Tools required for the remaining proofs}

We now turn our attention to the remaining proofs.  In order to better organize the arguments, we start by collecting a variety of deterministic and probabilistic tools we will require in subsequent sections.  

\subsection{Tools from linear algebra}

The Courant--Fisher minimax characterization of the eigenvalues (see, for instance, \cite[Chapter III]{Bhatia}) states that
$$ \lambda_i(M) = \min_V \max_{u \in V} u^\ast M u, $$
where $M$ is a Hermitian $n \times n$ matrix, $V$ ranges over $i$-dimensional subspaces of $\mathbb{C}^n$, and u ranges over unit vectors in $V$.  

From this one can obtain \emph{Cauchy's interlacing inequalities}:
\begin{equation} \label{eq:cauchy}
	\lambda_i(M_{n}) \leq \lambda_i(M_{n-1}) \leq \lambda_{i+1}(M_n) 
\end{equation}
for all $i < n$, where $M_n$ is an $n \times n$ Hermitian matrix and $M_{n-1}$ is the top $(n-1) \times (n-1)$ minor.  One also has the following more precise version of the Cauchy interlacing inequality.

\begin{lemma}[Interlacing identity; Lemma 40 from \cite{TVuniv}] \label{lemma:interlace}
Let
$$ M_n = \begin{pmatrix} M_{n-1} & X \\ X^\ast & m_{nn} \end{pmatrix} $$
be an $n \times n$ Hermitian matrix, where $M_{n-1}$ is the upper-left $(n-1) \times (n-1)$ minor of $M_n$, $X \in \mathbb{C}^{n-1}$, and $m_{nn} \in \mathbb{R}$.  Suppose that $X$ is not orthogonal to any of the unit eigenvectors $v_j(M_{n-1})$ of $M_{n-1}$.  Then we have
\begin{equation} \label{eq:interlaceeq}
	\sum_{j=1}^{n-1} \frac{ |v_j(M_{n-1})^\ast X|^2}{\lambda_j(M_{n-1}) - \lambda_i(M_n)} = m_{nn} - \lambda_i(M_n) 
\end{equation}
for every $1 \leq i \leq n$.  
\end{lemma}

The following lemma is needed when one wants to consider the coordinates of an eigenvector.  

\begin{lemma}[\cite{ESY1}; Lemma 41 from \cite{TVuniv}] \label{lemma:coordinate}
Let
$$ M_n = \begin{pmatrix} M_{n-1} & X \\ X^\ast & m_{nn} \end{pmatrix} $$
be an $n \times n$ Hermitian matrix, where $M_{n-1}$ is the upper-left $(n-1) \times (n-1)$ minor of $M_n$, $X \in \mathbb{C}^{n-1}$, and $m_{nn} \in \mathbb{R}$.  Let $\begin{pmatrix} v \\ x \end{pmatrix}$ be a unit eigenvector of $\lambda_i(M_n)$, where $v \in \mathbb{C}^{n-1}$ and $x \in \mathbb{C}$.  Suppose that none of the eigenvalues of $M_{n-1}$ are equal to $\lambda_i(M_n)$.  Then
$$ |x|^2 = \frac{1}{1 + \sum_{j=1}^{n-1} (\lambda_j(M_{n-1}) - \lambda_i(M_n))^{-2} |v_j(M_{n-1})^\ast X|^2}, $$
where $v_j(M_{n-1})$ is a unit eigenvector corresponding to the eigenvalue $\lambda_j(M_{n-1})$. 
\end{lemma}

\subsection{Spectral norm}
We will make use of the following bound for the spectral norm of a Wigner matrix with sub-gaussian entries.  

\begin{lemma}[Spectral norm of a Wigner matrix; Lemma 5 from \cite{OT}] \label{lemma:norm}
Let $\xi$, $\zeta$ be real sub-gaussian random variables with mean zero, and assume $\xi$ has unit variance.  Let $W$ be an $n \times n$ Wigner random matrix with atom variables $\xi$, $\zeta$.  Then there exists constants $C_0, c_0 > 0$ (depending only on the sub-gaussian moments of $\xi$ and $\zeta$) such that $\|W\| \leq C_0 \sqrt{n}$ with probability at least $1 - C_0 \exp(-c_0 n)$.
\end{lemma}

\subsection{Local semicircle law}
While Wigner's semicircle law (see, for example, \cite[Theorem 2.5]{BS}) describes the global behavior of the eigenvalues of a Wigner matrix, the local semicircle law describes the fine-scale behavior of the eigenvalues.  Many authors have proved versions of the local semicircle law under varying assumptions; we refer the reader to \cite{E,EKY,EKYY,EYTV, ESY2,ESY1,ESY3,ESY4,ESYY,EY,EYY2,EYY1,EYY,LY,TVedge,TVuniv,TVmeh,TVsur} and references therein.  

The local semicircle law stated below was proved by Lee and Yin in \cite{LY}.  We let $\rho_{\mathrm{sc}}$ denote the density of the semicircle distribution defined by 
\begin{equation} \label{eq:def:rho}
	\rho_{\mathrm{sc}}(x) := \left\{
	\begin{array}{ll}
		\frac{1}{2 \pi} \sqrt{4 - x^2}, & |x| \leq 2, \\
		0, & |x| > 2.
	\end{array} \right. 
\end{equation}

\begin{theorem}[Theorem 3.6 from \cite{LY}] \label{thm:dim} 
Let $\xi$, $\zeta$ be real sub-gaussian random variables with mean zero, and assume $\xi$ has unit variance. Let $W$ be an $n \times n$ Wigner matrix with atom variables $\xi$, $\zeta$. Let $N_I$ be the number of eigenvalues of $\frac{1}{ \sqrt{n}}W$ in the interval $I$. Then, there exist constants $C,c,c' > 0$ (depending on the sub-gaussian moments of $\xi$ and $\zeta$) such that, for any interval $I \subset \mathbb{R}$,  
$$ \Prob \left( \left| N_I - n\int_{I} \rho_{\mathrm{sc}}(x) dx \right| \geq (\log n)^{c' \log \log n} \right) \leq C \exp \left(- c (\log n)^{c \log \log n} \right). $$
\end{theorem}

It will occasionally be useful to avoid the $(\log n)^{c' \log \log n}$ term present in Theorem \ref{thm:dim} in favor of a bound of the form $(\log n)^{c'}$.  In these cases, the following result will be useful.  

\begin{theorem} \label{thm:dim2} 
Let $\xi$, $\zeta$ be real sub-gaussian random variables with mean zero, and assume $\xi$ has unit variance. Let $W$ be an $n \times n$ Wigner matrix with atom variables $\xi$, $\zeta$. Let $N_I$ be the number of eigenvalues of $\frac{1}{ \sqrt{n}}W$ in the interval $I$. Then, there exist constants $C,c,c' > 0$ (depending on the sub-gaussian moments of $\xi$ and $\zeta$) such that, for any interval $I \subset \mathbb{R}$,  
$$ \Prob \left( N_I  \geq (1.1) n \int_{I} \rho_{\mathrm{sc}}(x) dx +  (\log n)^{c'} \right) \leq  C \exp \left(- c \log^2 n \right). $$ 
\end{theorem}
\begin{remark}
The constant $1.1$ appearing in Theorem \ref{thm:dim2} can be replaced by any absolute constant larger than one.  
\end{remark}

Theorem \ref{thm:dim2} follows from the arguments in \cite{TVuniv, VW} (see \cite[Proposition 66]{TVuniv} for details).

\subsection{Smallest singular value}
We will need the following result concerning the least singular value of a rectangular random matrix with iid entries.  For an $N \times n$ matrix $M$, we let $\sigma_1(M) \geq \cdots \geq \sigma_n(M)$ denote the ordered eigenvalues of $\sqrt{M^\ast M}$.  (Note that the non-zero eigenvalues of $\sqrt{M M^\ast}$ are the same as the non-zero eigenvalues of $\sqrt{M^\ast M}$.)  In particular, $\sigma_1(M) = \|M\|$, and $\sigma_n(M)$ is called the smallest singular value of $M$.  We refer the interested reader to \cite{RVsurvey} for a wonderful survey on the non-asymptotic theory of extreme singular values of random matrices with independent entires.  

\begin{theorem}[Theorem 1.1 from \cite{RVssv}] \label{thm:lsv2}
Let $B$ be an $N \times n$ random matrix, $N \geq n$, whose elements are independent copies of a mean zero sub-gaussian random variable $\xi$ with unit variance.  Then, for every $\eps > 0$, 
$$ \Prob \left( \sigma_n(B) \leq \eps \left( \sqrt{N} - \sqrt{n-1} \right) \right) \leq (C \eps)^{N - n + 1} + e^{-cN}, $$
where $C,c > 0$ depend only on the sub-gaussian moment of $\xi$.  
\end{theorem}

Similar bounds for the smallest singular value were also obtained by Vershynin \cite{V} under the assumption that $\xi$ has $4+\eps$ finite moments instead of the sub-gaussian assumption above.  We also refer the reader to \cite{RV} for bounds on the extreme singular values of random matrices with heavy-tailed rows.  

\subsection{Projection lemma}

We will also need the following bound, which follows from the Hanson--Wright inequality (see, for example, \cite[Theorem 1.1]{RVhw}).   

\begin{lemma}[Projection lemma] \label{lemma:qua}
Let $\xi$ be a sub-gaussian random variable with mean zero and unit variance.  Let $B$ be a $n \times m$ random matrix whose entries are iid copies of $\xi$.  Let $H$ be a subspace of $\mathbb{R}^n$ of dimension $d$, and let $P_H$ denote the orthogonal projection onto $H$.  Then there exist constants $C, c>0$ (depending only on the sub-gaussian moment of $\xi$) such that, for any unit vector $y \in \mathbb{R}^m$ and every $t \geq 0$, 
$$ \Prob \left( \left| \|P_H B y \|^2 - d \right| > t \right) \leq C \exp \left(-c \min \left\{ \frac{t^2}{d}, t \right\} \right). $$
\end{lemma}
\begin{proof}
Let $B_1, \ldots, B_m$ denote the columns of $B$, and set $y = (y_i)_{i=1}^m$.  Since $P_H$ is an orthogonal projection, we write $P_H = \sum_{i=1}^d u_i u_i^\mathrm{T}$, where $\{u_1, \ldots, u_d\}$ is an orthonormal basis of $H$.  

Let $X$ denote the $mn$-vector with iid entries given by 
$$ X := \begin{pmatrix} B_1 \\ \vdots \\ B_m \end{pmatrix}. $$
Define the $mn \times mn$ matrix $\mathcal{P} := \sum_{i=1}^d w_i w_i^\mathrm{T}$, where
$$ w_i := \begin{pmatrix} y_1 u_i \\ \vdots \\ y_m u_i \end{pmatrix}. $$
It follows, from the definitions above, that $X^\mathrm{T} \mathcal{P} X = \| P_H B y \|^2$.  Moreover, since $\{w_1, \ldots, w_d\}$ is an orthonormal set, $\mathcal{P}$ is an orthogonal projection.  Thus, $\|\mathcal{P}\| = 1$ and
$$ \tr \mathcal{P} = \|\mathcal{P}\|^2_2 = \rank(\mathcal{P}) = d. $$

Since the entries of $X$ are iid copies of $\xi$, we find that
$$ \E X^\mathrm{T} \mathcal{P} X = \tr \mathcal{P} = d. $$
So, by the Hanson--Wright inequality (see, for example, \cite[Theorem 1.1]{RVhw}), we obtain, for all $t \geq 0$, 
$$ \Prob \left( \left| X^\mathrm{T} \mathcal{P} X - d \right| > t \right) \leq C \exp \left( - c \min \left\{\frac{t^2}{d}, t \right\} \right), $$
where $C, c > 0$ depends only on the sub-gaussian moment of $\xi$.  
\end{proof}

In the case when $B$ is an $n \times 1$ matrix (i.e. a vector) and $y = 1$, we immediately obtain the following corollary.  

\begin{corollary} \label{cor:projection}
Let $\xi$ be a sub-gaussian random variable with mean zero and unit variance.  Let $X$ be a vector in $\mathbb{R}^n$ whose entries are iid copies of $\xi$.  Let $H$ be a subspace of $\mathbb{R}^n$ of dimension $d$, and let $P_H$ denote the orthogonal projection onto $H$.  Then there exist constants $C, c>0$ (depending only on the sub-gaussian moment of $\xi$) such that, for every $t \geq 0$, 
$$ \Prob \left( \left| \|P_H X \|^2 - d \right| > t \right) \leq C \exp \left(-c \min \left\{ \frac{t^2}{d}, t \right\} \right). $$
\end{corollary}

\subsection{Deterministic tools and the equation $Ax = By$}
We now consider the equation $Ax = By$, where $x$ and $y$ are vectors, $A$ is a rectangular matrix, 
and $B$ is a Hermitian matrix.  If $\|x\|$ is small, then $\|Ax \| = \|B y\|$ is also relatively small.  Intuitively, then,
 it must be the case that the vector $y$ is essentially supported on the eigenvectors of $B$ corresponding to small eigenvalues.  
 In the following lemmata, we quantify the structure of $y$ in terms of $\|A\|$ and the spectral decomposition of $B$.  Similar results were implicitly used in \cite{AB}.

\begin{lemma} \label{lemma:structure}
Let $A$ be a $r\times m$ matrix and $B$ be a Hermitian $r \times r$ matrix.  Let $\eps, \tau > 0$.  Let $x$ and $y$ be vectors with $\|x\| \leq \eps \tau$ and $By = Ax$.  Then $y=v+q$, where $v$ and $q$ are orthogonal, $\|q\| \leq \eps$, and 
$$ v \in \Span\{ v_i (B) : |\lambda_i(B)| \leq \tau \|A\| \}. $$
\end{lemma}

In many cases, the norm of $A$ will be too large for the above lemma to be useful.  However, if we can write $A$ as a sum of two parts, one with small norm and one with low rank, we can still obtain essentially the same conclusion using the following lemma.

\begin{lemma} \label{lemma:structure2}
Let $A$ and $J$ be $r \times m$ matrices with $\|A\| \leq \kappa$, and let $B$ be a Hermitian $r \times r$ matrix.  Let $\eps, \tau > 0$.  Let $x$ and $y$ be vectors with $\|x\| \leq \eps \tau$, $\|y\| \leq 1$, and $By = (A+J)x$.  Then
\begin{enumerate}[(i)]
\item there exists a non-negative real number $\eta$ (depending only on $B, \kappa, \tau, \eps$) such that $B-\eta I$ is invertible, 
\item \label{item:span} the vector $y$ can be decomposed as $y=v+q$, where $\|q\| \leq \eps$, 
$$ v \in V := \Span \left\{ \{v_i(B) : |\lambda_i(B)| \leq \tau \kappa \} \cup \range\left(( B-\eta I)^{-1} J \right) \right\}, $$
and 
\begin{equation} \label{eq:dimbnd}
	\dim(V) \leq |\{1 \leq i \leq r: |\lambda_i(B)| \leq \tau \kappa\}| + \rank(J). 
\end{equation}
\end{enumerate}
\end{lemma}  

\begin{remark}
One can similarly prove versions of Lemmas \ref{lemma:structure} and \ref{lemma:structure2} when $B$ is not Hermitian.  In this case, one must rely on the singular value decomposition of $B$ instead of the spectral theorem.  In particular, the set $V$ appearing in the statement of Lemma \ref{lemma:structure2} will need to be defined in terms of the singular values and singular vectors of $B-\eta I$.  
\end{remark}

We prove Lemmas \ref{lemma:structure} and \ref{lemma:structure2} in Appendix \ref{sec:structure}.  

\section{Proofs of results concerning extremal coordinates} \label{sec:proof:extreme}
This section is devoted to the proofs of Corollary \ref{cor:delo} and Theorem \ref{lower}.  

\subsection{Proof of Corollary \ref{cor:delo}}

The desired bound will follow from \cite[Theorem 6.1]{VW} and a simple truncation argument.  Let $W$ be the Wigner matrix from Corollary \ref{cor:delo} with sub-exponential atom variable $\xi$.  Then, there exist $\alpha, \beta > 0$ such that
\begin{equation} \label{eq:subbetaal}
	\Prob ( |\xi| > t ) \leq \beta \exp(- t^\alpha / \beta) 
\end{equation}
for all $t > 0$.  In addition, since $\xi$ is symmetric, it follows that $\xi$ has mean zero.  Let $C_1 > 0$ be given, and take $C > 0$ to be a large constant to be chosen later.  Define
$$ \tilde{\xi} := \xi \indicator{|\xi| \leq C \log^{1/\alpha} n}, $$
where $\oindicator{E}$ is the indicator function of the event $E$.  Since $\xi$ is symmetric it follows that $\tilde{\xi}$ has mean zero.  Similarly, we define the matrix $\tilde{W} = (\tilde{w}_{ij})_{i,j=1}^n$ by 
$$ \tilde{w}_{ij} := w_{ij} \indicator{|w_{ij}| \leq C \log^{1/\alpha} n}. $$
It follows that $\tilde{W}$ is an $n \times n$ Wigner matrix with atom variable $\tilde{\xi}$.  Moreover, $\tilde{\xi}$ has mean zero and is $(C \log^{1/\alpha} n)$-bounded.  Thus, \cite[Theorem 6.1]{VW} can be applied to $\tilde{W}$,\footnote{Technically, one also has to normalize the entries of $\tilde{W}$ to have unit variance as required by \cite[Theorem 6.1]{VW}.  However, this corresponds to multiplying $\tilde{W}$ by a positive scalar and the unit eigenvectors are invariant under such a scaling.} and we obtain the desired conclusion for $\tilde{W}$.  It remains to show that the same conclusion also holds for $W$.  From \eqref{eq:subbetaal}, we have
\begin{align*}
	\Prob( W \neq \tilde{W} ) &\leq \sum_{i \leq j} \Prob( w_{ij} \neq \tilde{w}_{ij})  \\
	&\leq n^2 \Prob( |\xi| > C \log^{1/\alpha} n ) \\
	&\leq \beta n^2 \exp \left(- \frac{C^{\alpha}}{\beta} \log n \right) \\
	&\leq \beta n^{-C_1} 
\end{align*}
by taking $C$ sufficiently large.  Therefore, on the event where $\tilde{W} = W$, we obtain the desired conclusion, and the proof is complete.

\subsection{Proof of Theorem \ref{lower}}
We will need the following result from \cite{NTV}.

\begin{theorem}[\cite{NTV}] \label{thm:NTV}
Let $W_n$ be the $n \times n$ Wigner matrix from Theorem \ref{lower}.  Let
$$ W_n = \begin{bmatrix} W_{n-1} & X \\ X^\mathrm{T} & w_{nn} \end{bmatrix}, $$
where $W_{n-1}$ is the upper-left $(n-1) \times (n-1)$ minor of $W_n$, $X \in \mathbb{R}^{n-1}$, and $w_{nn} \in \mathbb{R}$.  Then there exist constants $C, c, c_0 > 0$ such that, for any $n^{-c_0} < \alpha < c_0$ and $\delta \geq n^{-c_0 / \alpha}$, 
$$ \sup_{1 \leq i \leq n-1} \Prob \left( \lambda_i(W_{n-1}) - \lambda_i(W_n) \leq \frac{\delta}{\sqrt{n}} \right) \leq C \frac{\delta}{\sqrt{\alpha}} + C \exp(-c \log^2 n) $$
and
$$ \sup_{2 \leq i \leq n} \Prob \left( \lambda_{i}(W_n) - \lambda_{i-1}(W_{n-1}) \leq \frac{\delta}{\sqrt{n}} \right) \leq C \frac{\delta}{\sqrt{\alpha}} + C \exp(- c \log^2 n). $$
Moreover, 
\begin{equation} \label{eq:lowbndvj}
	\inf_{1 \leq j \leq n-1} |v_j(W_{n-1})^\mathrm{T} X | > 0 
\end{equation}
with probability at least $1 - C \exp(-c \log^2 n)$.  
\end{theorem}

\begin{remark} \label{rem:strictinter}
The bound in \eqref{eq:lowbndvj} implies that $X$ is not orthogonal to $v_j(W_{n-1})$ for any $1 \leq j \leq n-1$.  By \cite[Lemma 3]{OT}, this implies that none of the eigenvalues of $W_{n}$ coincide with an eigenvalue of $W_{n-1}$.  In other words, the interlacing described in \eqref{eq:cauchy} is strict.  
\end{remark}

Theorem \ref{thm:NTV} follows from \cite[Theorems 4.1 and 4.3]{NTV} and the arguments given in \cite[Section 4]{NTV}.   We will also need the following technical lemma. 

\begin{lemma} \label{lemma:techsumbnd}
Let $W_n$ be the $n \times n$ Wigner matrix from Theorem \ref{lower}.  Let
$$ W_n = \begin{bmatrix} W_{n-1} & X \\ X^\mathrm{T} & w_{nn} \end{bmatrix}, $$
where $W_{n-1}$ is the upper-left $(n-1) \times (n-1)$ minor of $W_n$, $X \in \mathbb{R}^{n-1}$, and $w_{nn} \in \mathbb{R}$.  Let $\kappa > 0$.  Then there exist constants $C, c, c_1 > 0$ such that, for any $1 \leq i \leq n$ and any $e^{-n^\kappa} < \delta < 1$, 
\begin{equation} \label{eq:sumhelpbnd}
	\sum_{j=1}^{n-1} \frac{ |v_j(W_{n-1})^\mathrm{T} X|^2 }{ | \lambda_j(W_{n-1}) - \lambda_i(W_n) |^2} \leq \frac{ (\log n)^{c_1} }{ 25 } \left[ \frac{ 1} { m_i^2 } + \frac{n}{ \delta^2} \right] 
\end{equation}
with probability at least $1 - C \exp(-c \log^2 n)$, where
$$ m_i := \min_{1 \leq j \leq n-1} |\lambda_j(W_{n-1}) - \lambda_i(W_n)|. $$  
\end{lemma}
\begin{proof}
Fix $1 \leq i \leq n$ and $e^{-n^\kappa} < \delta < 1$.  Define the event $\Omega$ to be the intersection of the events
\begin{equation} \label{eq:eventvjWn1}
	\bigcap_{j=1}^{n-1} \left\{ 0 < |v_j(W_{n-1})^\mathrm{T} X|^2 \leq \frac{1}{1000} \log^2 n \right\} 
\end{equation}
and
\begin{equation} \label{eq:normC0bnd}
	\left\{ \|W_n \| \leq C_0 \sqrt{n} \right\} 
\end{equation}
for some constant $C_0$ to be chosen later.  

We claim that, for $C_0$ sufficiently large, 
\begin{equation} \label{eq:Omegaclog2n}
	\Prob(\Omega^c) \leq C \exp(-c \log^2 n) 
\end{equation}
for some constants $C, c > 0$.  To see this, note that $|v_j(W_{n-1})^\mathrm{T} X|$ is the length of the projection of the vector $X$ onto the one-dimensional subspace spanned by $v_j(W_{n-1})$.   Hence, by Corollary \ref{cor:projection}, the union bound, and \eqref{eq:lowbndvj}, it follows that the event in \eqref{eq:eventvjWn1} holds with probability at least $1 - C \exp(-c \log^2 n)$.  In addition, the event in \eqref{eq:normC0bnd} can be dealt with by taking $C_0$ sufficiently large and applying Lemma \ref{lemma:norm}.  Combining the estimates above yields the bound in \eqref{eq:Omegaclog2n}.  

It now suffices to show that, conditionally on $\Omega$, the bound in \eqref{eq:sumhelpbnd} holds with probability at least $1 - C' \exp(-c' \log^2 n)$.  As Remark \ref{rem:strictinter} implies, on $\Omega$, the eigenvalues of $W_{n-1}$ strictly interlace with the eigenvalues of $W_n$, and hence the terms
$$ \sum_{j=1}^{n-1} \frac{ |v_j(W_{n-1})^\mathrm{T} X|^2 }{ | \lambda_j(W_{n-1}) - \lambda_i(W_n) |^2} $$
and $m_i^{-2}$ are well-defined.  Moreover, on $\Omega$, we have
$$ \sum_{j=1}^{n-1} \frac{ |v_j(W_{n-1})^\mathrm{T} X|^2 }{ | \lambda_j(W_{n-1}) - \lambda_i(W_n) |^2}  \leq \frac{\log^2 n}{1000} \sum_{j=1}^{n-1} \frac{1}{ | \lambda_j(W_{n-1}) - \lambda_i(W_n) |^2}. $$
Hence, conditionally on $\Omega$, it suffices to show that
$$ \sum_{j=1}^{n-1} \frac{1}{ | \lambda_j(W_{n-1}) - \lambda_i(W_n) |^2} \leq 10 (\log n)^{c_1} \left[ \frac{1}{m_i^2} + \frac{n}{\delta^2} \right] $$
with probability at least $1 - C' \exp(-c' \log^2 n)$.  

Define the sets
$$ T := \left\{ 1 \leq j \leq n-1 : |\lambda_j(W_{n-1}) - \lambda_i(W_{n})| < \frac{\delta}{\sqrt{n}} \right\} $$
and
$$ T_l := \left\{ 1 \leq j \leq n-1 : 2^l \frac{\delta}{\sqrt{n}} \leq |\lambda_j(W_{n-1}) - \lambda_i(W_n)| < 2^{l+1} \frac{\delta}{\sqrt{n}} \right\} $$
for $l = 0, 1, \ldots, L$, where $L$ is the smallest integer such that $2^L \frac{\delta}{\sqrt{n}} \geq 2C_0 \sqrt{n}$.  In particular, as $\delta > e^{-n^\kappa}$, we obtain $L = O(n^{-\kappa})$.  On the event $\Omega$, it follows that every index $1 \leq j \leq n-1$ is contained in either $T$ or $\cup_{l=0}^L T_l$.  By Theorem \ref{thm:dim2} and the union bound\footnote{Technically, one cannot apply Theorem \ref{thm:dim2} directly to these intervals because the intervals are defined in terms of the eigenvalues.  For instance, $|T|$ is the number of eigenvalues of $W_{n-1}$ in the open interval of radius $\frac{\delta}{\sqrt{n}}$ centered at $\lambda_i(W_n)$.  However, one can first take (say) $n^{100}$ equispaced points in the interval $[-C_0 \sqrt{n}, C_0 \sqrt{n}]$ and apply the theorem to intervals of radius $\frac{\delta}{\sqrt{n}}$ centered at each of these points.  One can then deduce the desired conclusion by approximating $|T|$ using these deterministic intervals.  A similar approximation argument works for each $T_l$.}, there exists $C', c' , c_1 > 0$ such that
$$ |T| \leq (\log n)^{c_1} $$
and 
$$ |T_l| \leq 4 (2^l ) + (\log n)^{c_1}, \quad l = 0, \ldots, L $$
with probability at least $1 - C' \exp(-c' \log^2 n)$.  Hence, on this same event, we conclude that
\begin{align*}
	\sum_{j=1}^{n-1} \frac{1}{ | \lambda_j(W_{n-1}) - \lambda_i(W_n) |^2} &\leq \sum_{j \in T} \frac{1}{ | \lambda_j(W_{n-1}) - \lambda_i(W_n) |^2} \\ 
	&\qquad\qquad\qquad + \sum_{l=0}^L \sum_{j \in T_l} \frac{1}{ | \lambda_j(W_{n-1}) - \lambda_i(W_n) |^2} \\
	&\leq \frac{|T|}{m_i^2} + \frac{n}{\delta^2} \sum_{l=0}^L \frac{|T_l|}{ 2^{2l} } \\
	&\leq \frac{ (\log n)^{c_1} }{ m_i^2 } + \frac{10n (\log n)^{c_1} }{\delta^2}
\end{align*}
for $n$ sufficiently large.  The proof of the lemma is complete.  
\end{proof}

With Theorem \ref{thm:NTV} and Lemma \ref{lemma:techsumbnd} in hand, we are now ready to prove Theorem \ref{lower}.

\begin{proof}[Proof of Theorem \ref{lower}]
Let $W_n$ be the $n \times n$ Wigner matrix with atom variables $\xi, \zeta$.  We will bound the $j$th coordinate of the unit eigenvector $v_i$ in magnitude from below.  By symmetry, it suffices to consider the case when $j=n$.  For this reason, we decompose $W_n$ as 
$$ W_n = \begin{bmatrix} W_{n-1} & X \\ X^\mathrm{T} & w_{nn} \end{bmatrix}, $$
where $W_{n-1}$ is the upper-left $(n-1) \times (n-1)$ minor of $W_n$, $X \in \mathbb{R}^{n-1}$, and $w_{nn} \in \mathbb{R}$.  

Let $c_0$ be as in Theorem \ref{thm:NTV}, and assume $n^{-c_0} < \alpha < c_0$ and $1 > \delta \geq n^{-c_0/\alpha}$ as the claim is trivial when $\delta \geq 1$.  Define the event $\Omega$ to be the intersection of the events
$$ \left\{ \inf_{1 \leq j \leq n-1} |v_j(W_{n-1})^\mathrm{T} X| > 0 \right\} \bigcap  \left\{ m_i > \frac{\delta}{\sqrt{n}} \right\} $$
and
$$ \left\{ \sum_{j=1}^{n-1} \frac{ |v_j(W_{n-1})^\mathrm{T} X|^2 }{ |\lambda_j(W_{n-1}) - \lambda_i(W_n)|^2 } \leq \frac{ (\log n)^{c_1} }{ 25 } \left[ \frac{ 1} { m_i^2 } + \frac{n}{ \delta^2} \right] \right\}, $$
where
$$ m_i := \min_{1 \leq j \leq n-1} |\lambda_j(W_{n-1}) - \lambda_i(W_n)| $$  
and $c_1$ is the constant from Lemma \ref{lemma:techsumbnd}.  If $m_i \leq \frac{\delta}{\sqrt{n}}$, then, by Cauchy's interlacing inequalities \eqref{eq:cauchy}, this implies that either 
$$ |\lambda_i(W_{n-1}) - \lambda_i(W_n)| \leq \frac{\delta}{\sqrt{n}}  $$
or 
$$ |\lambda_{i-1}(W_{n-1}) - \lambda_i(W_n)| \leq \frac{\delta}{\sqrt{n}}. $$
(Here, the first possibility can only occur if $i < n$, and the second possibility can only occur if $i > 1$.)  
Therefore, it follows from Theorem \ref{thm:NTV} and Lemma \ref{lemma:techsumbnd} that, there exists $C', c' > 0$ such that
$\Omega$ holds with probability at least $1 - C' \frac{\delta}{\sqrt{\alpha}} - C' \exp(-c' \log^2 n)$.  

Let $c_2 := c_1 / 2$.  Then
$$ \Prob \left( |v_i(n)| \leq \frac{\delta}{\sqrt{n} (\log n)^{c_2} } \right) \leq \Prob (\Omega^c) + \Prob \left( \left\{ |v_i(n)| \leq \frac{\delta}{\sqrt{n} (\log n)^{c_2} } \right\} \bigcap \Omega \right). $$
Hence, to complete the proof, we will show that the event
\begin{equation} \label{eq:eventsuffice}
	\left\{ |v_i(n)| \leq \frac{\delta}{\sqrt{n} (\log n)^{c_2} } \right\} \bigcap \Omega 
\end{equation}
is empty.  

Fix a realization in this event.  As Remark \ref{rem:strictinter} implies, the eigenvalues of $W_{n-1}$ strictly interlace with the eigenvalues of $W_n$.  Hence, we apply Lemma \ref{lemma:coordinate} and obtain
$$ \frac{1}{1 + \sum_{j=1}^{n-1} \frac{ |v_j(W_{n-1})^\mathrm{T} X |^2}{ | \lambda_j(W_{n-1}) - \lambda_i(W_n) |^2 } } = |v_i(n)|^2 \leq \frac{\delta^2}{n (\log n)^{2 c_2}}. $$
This implies, for $n$ sufficiently large, that
$$ \frac{n (\log n)^{2 c_2 }}{2 \delta^2} \leq \sum_{j=1}^{n-1} \frac{ |v_j(W_{n-1})^\mathrm{T} X |^2}{ | \lambda_j(W_{n-1}) - \lambda_i(W_n) |^2 }. $$
On the other hand, by definition of $\Omega$, we have
\begin{align*}
	\sum_{j=1}^{n-1} \frac{ |v_j(W_{n-1})^\mathrm{T} X |^2}{ | \lambda_j(W_{n-1}) - \lambda_i(W_n) |^2 } \leq \frac{2 (\log n)^{c_1} } { 25} \frac{n} {\delta^2},
\end{align*}
and thus
$$ \frac{n (\log n)^{2 c_2}}{2 \delta^2} \leq  \frac{2 (\log n)^{c_1} } { 25} \frac{n} {\delta^2}. $$
This is a contradiction since $2 c_2 = c_1$.  We conclude that the event in \eqref{eq:eventsuffice} is empty, and the proof is complete.  
\end{proof}

\begin{remark}
In the special case that one considers only $v_1(W_n)$ or $v_n(W_n)$, a simpler argument is possible by applying Lemma \ref{lemma:interlace}.  Indeed, in this case, one can exploit the fact that the terms $\lambda_j(M_{n-1}) - \lambda_i(M_n)$ appearing in the denominator of \eqref{eq:interlaceeq} are always of the same sign due to the eigenvalue interlacing inequalities \eqref{eq:cauchy}.  However, this argument does not appear to generalize to any other eigenvectors.  
\end{remark}

\section{Proofs of no-gaps delocalization results} \label{sec:proof:wigner}

This section is devoted to the proofs of Theorem \ref{thm:generalized} and Corollary \ref{cor:lp}.  

\subsection{Proof for Theorem \ref{thm:generalized}} 
We begin with a few reductions.  Let $\eps > 0$.  By definition of convergence in probability, it suffices to show
$$ \left|\max_{S\subset [n]: |S|=\lfloor\delta n \rfloor} \|v_{k_n}\|_S^2 + \int_{0}^{\delta} H(u) \, du \right| \le \varepsilon $$
and
$$ \left|\min_{S\subset [n]: |S|=\lfloor\delta n \rfloor} \|v_{k_n}\|_S^2 + \int_{1-\delta}^1 H(u) \, du \right| \le \varepsilon $$
with probability $1-o(1)$.  

Let $Z$ be a standard normal distribution with cumulative distribution function $\Phi(x)$.  Let $F(x)$ be the cumulative distribution function of $Z^2$, which has the $\chi^2$-distribution with $1$ degree of freedom.   Hence, 
$$ F(x) = \Prob(Z^2 \le x) = 2\Phi(\sqrt{x})-1 $$ 
for $x\ge 0$. By setting $x=\sqrt{F^{-1}(u)}$ in the following integrals, we observe that 
\begin{align*}
-\int_{1-\delta}^1 H(u) \, du = \int_0^{\delta} F^{-1}(u)\,du =  2 \int_{0}^{\Phi^{-1}(\frac{1+\delta}{2})} x^2 \Phi'(x)\,dx
\end{align*}
and
\begin{align}\label{eq:nor}
-\int_{0}^{\delta} H(u) \, du = \int_{1-\delta}^{1} F^{-1}(u)\,du= 2 \int_{\Phi^{-1}(1-\frac{\delta}{2})}^{\infty} x^2 \Phi'(x)\,dx.
\end{align}
Thus, it suffices to show
\begin{equation} \label{eq:gen:show1}
	\left|\max_{S\subset [n]: |S|=\lfloor\delta n \rfloor} \|v_{k_n}\|_S^2 - 2 \int_{\Phi^{-1}(1 - \frac{\delta}{2})}^{\infty} x^2 \Phi'(x)\,dx \right| \le \varepsilon 
\end{equation}
and
\begin{equation} \label{eq:gen:show2}
	\left|\min_{S\subset [n]: |S|=\lfloor\delta n \rfloor} \|v_{k_n}\|_S^2 - 2 \int_{0}^{\Phi^{-1}(\frac{1+\delta}{2})} x^2 \Phi'(x)\,dx \right| \le \varepsilon 
\end{equation}
with probability $1 - o(1)$, where $\Phi$ is the cumulative distribution function of the standard normal distribution.  

We now turn our attention to proving \eqref{eq:gen:show1} and \eqref{eq:gen:show2}.  In fact, \eqref{eq:gen:show1} follows from \eqref{eq:gen:show2} by applying the identity 
$$ \max_{S \subset [n]: |S|=\lfloor \delta n \rfloor} \|v_{k_n}\|_S^2 + \min_{S \subset [n]: |S|=\lfloor \delta n \rfloor} \|v_{k_n}\|_{S^c}^2= 1 $$
and using the fact that $2 \int_{0}^{\infty} x^2 \Phi'(x) \, dx = 1$.  (Alternatively, one can prove \eqref{eq:gen:show1} by repeating the arguments below.)  Thus, it remains to verify \eqref{eq:gen:show2}.  

Let $\delta \in (0,1)$.  For notational convenience, let $v$ denote the unit eigenvector under consideration, and let $v(j)$ denote its $j$th coordinate.   Let $c > 0$ be a (small) constant to be chosen later.  For each $k \in \mathbb{N}$, define
$$ N(c,k) := \sum_{j=1}^n \indicator{c(k-1) \leq \sqrt{n} |v(j)| < ck} $$
to be the number of coordinates of $\sqrt{n} v$ with magnitude in the interval $[c(k-1), ck)$.  

Let $Z$ be a standard normal random variable.  Define
$$ f(c,k) := n \Prob( c(k-1) \leq |Z| < ck ). $$
Since $\Phi$ is the cumulative distribution function of $Z$, it follows that
$$ f(c,k) = 2n \left( \Phi(ck) - \Phi(c(k-1)) \right). $$

From Corollary \ref{cor:BY}, it follows that, for each $1 \leq j \leq n$, 
$$ \Prob ( c(k-1) \leq \sqrt{n} |v(j)| < ck ) = \Prob( c(k-1) \leq |Z| < ck ) (1 + o(1)). $$
We now claim that this identity holds uniformly for all $1 \leq j \leq n$.  This does not follow from the formulation of Corollary \ref{cor:BY}, but instead follows from the second part of \cite[Corollary 1.3]{BY}, which gives a uniform bound on the convergence of moments.  Indeed, uniform rates of convergence follow by combining the uniform convergence of moments with the identity 
$$ \left| e^{i t X} - \sum_{l=0}^{s-1} \frac{ (i t)^l X^l }{l!} \right| \leq \frac{|t|^s |X|^s}{s!} $$
and the inequality from \cite[page 538]{Fe}.  Therefore, we obtain that
$$ \E N(c,k) = (1 + o(1)) f(c,k). $$
In addition, by a similar argument, it follows that
$$ \var( N(c,k) ) = o(n^2). $$

By Chebyshev's inequality, we have that, for any $\eps > 0$, 
$$ \Prob \left( |N(c,k) - \E N(c,k)| \geq \eps f(c,k) \right) \leq \frac{\var( N(c,k) )}{ \eps^2 f(c,k)^2 }. $$
We note that $f(c,k) \geq c' n$ for some constant $c' > 0$ depending on $c$ and $k$.  Hence, we conclude that
\begin{equation} \label{eq:Nckfck}
	N(c,k) = (1 + o(1)) f(c,k) 
\end{equation}
with probability $1 - o(1)$.  

We will return to \eqref{eq:Nckfck} in a moment.  We now approximate the distribution of $\min_{S \subset [n] : |S| = \lfloor \delta n \rfloor} \|v \|_S^2$.  We will first need to determine the value of $c$ from above and another parameter $k_0$.  Take $c > 0$ and $k_0 \in \mathbb{N}$ such that
\begin{equation} \label{eq:choicec}
	2c \left[ \Phi^{-1} \left( \frac{1 + \delta}{2} \right)\right]^{2} < \frac{\eps}{2} 
\end{equation}
and
\begin{equation} \label{eq:choicek0}
	c k_0 = \Phi^{-1} \left( \frac{1 + \delta}{2} \right). 
\end{equation}
Such choices are always possible by taking $c>0$ sufficiently small such that \eqref{eq:choicec} holds, and then (by possibly decreasing $c$ if necessary) choosing $k_0 \in \mathbb{N}$ which satisfies \eqref{eq:choicek0}.  

There are several important implications of these choices.  First, 
\begin{equation} \label{eq:2sumkk0}
	2 \sum_{k=1}^{k_0} \left( \Phi(ck) - \Phi( (c(k-1) ) \right) = 2 \left( \Phi(ck_0) - \Phi(0) \right) = \delta. 
\end{equation}
Second, we have
\begin{equation} \label{eq:approxbnd1}
\begin{aligned}
	&\left| 2 \sum_{k=1}^{k_0} c^2 (k-1)^2 \left( \Phi(ck) - \Phi(c(k-1)) \right) - 2 \int_{0}^{\Phi^{-1}( \frac{1 + \delta}{2} )} x^2 \Phi'(x) \, dx \right| \\
	&\qquad\qquad= 2 \left| \sum_{k=1}^{k_0} \left[ c^2 (k-1)^2 \left( \Phi(ck) - \Phi(c(k-1)) \right) - \int_{c(k-1)}^{ck} x^2 \Phi'(x) \, dx \right] \right| \\
	&\qquad\qquad= 2 \left| \sum_{k=1}^{k_0} \left[ ( c^2(k-1)^2 - c^2 k^2 ) \Phi(ck) + 2 \int_{c(k-1)}^{ck} x \Phi(x) \, dx \right] \right| \\
	&\qquad\qquad= 4 \left| \sum_{k=1}^{k_0} \int_{c(k-1)}^{ck} x \left( \Phi(x) - \Phi(ck) \right) \, dx \right| \\
	&\qquad\qquad\leq 2 c \left[ \Phi^{-1} \left( \frac{1 + \delta}{2} \right) \right]^{2} \\
	&\qquad\qquad< \frac{\eps}{2} 
\end{aligned}
\end{equation}
by integration by parts and \eqref{eq:choicec}.  Here the first inequality follows from the mean value theorem and the bound $0 \leq \Phi'(x) \leq 1$.  
By a similar argument, we obtain
\begin{equation} \label{eq:approxbnd2}
	\left| 2 \sum_{k=1}^{k_0} c^2 k^2 \left(\Phi(ck) - \Phi(c(k-1)) \right) - 2 \int_{0}^{\Phi^{-1}( \frac{1 + \delta}{2} )} x^2 \Phi'(x) \, dx \right| < \frac{\eps}{2}. 
\end{equation}

From \eqref{eq:Nckfck}, we have that for any $1 \leq k \leq k_0 + 1$, 
\begin{equation} \label{eq:limitNck}
	N(c,k) = (1 + o(1)) f(c,k) = 2n (1 + o(1)) \left( \Phi(ck) - \Phi(c(k-1)) \right) 
\end{equation}
with probability $1 - o(1)$.  In view of \eqref{eq:2sumkk0}, we have
$$ \sum_{k=1}^{k_0} f(c,k) = \delta n. $$
Thus, by the union bound, we obtain that, with probability $1 - o(1)$, 
$$ \sum_{k=1}^{k_0} N(c,k) = (1 + o(1)) \delta n = \lfloor \delta n \rfloor + o(n). $$

We observe that
$$ \sum_{k=1}^{k_0} c^2 (k-1)^2 N(c,k) \leq n \min_{S \subset [n] : |S| = \sum_{k=1}^{k_0} N(c,k)} \|v\|_S^2 \leq \sum_{k=1}^{k_0} c^2 k^2 N(c,k) $$
by definition of $N(c,k)$.  Thus, with probability $1 - o(1)$, there exists a sequence $(\eps_n)_{n \geq 1}$ with $\eps_n \searrow 0$ such that
\begin{align*}
	\sum_{k=1}^{k_0} &c^2 (k-1)^2 N(c,k) - \eps_n c^2 k_0^2 N(c,k_0) \\
	&\leq n  \min_{S \subset [n] : |S| = \lfloor \delta n \rfloor} \|v\|_S^2 \leq \sum_{k=1}^{k_0} c^2 k^2 N(c,k) + \eps_n c^2 (k_0 + 1)^2 N(c,k_0+1). 
\end{align*}
Applying \eqref{eq:limitNck}, we find that 
\begin{align*}
	2 \sum_{k=1}^{k_0} &c^2(k-1)^2(\Phi(ck)-\Phi(c(k-1)))(1+o(1)) \\
	&\leq \min_{S \subset [n]: |S|=\lfloor \delta n \rfloor} \|v\|_S^2 \leq 2\sum_{k=1}^{k_0} c^2 k^2(\Phi(ck)-\Phi(c(k-1)))(1+o(1))
\end{align*}
with probability $1-o(1)$.  Therefore, by \eqref{eq:approxbnd1} and \eqref{eq:approxbnd2}, we conclude that 
$$ \left|\min_{S\subset [n]: |S|=\lfloor\delta n \rfloor} \|v\|_S^2 - 2 \int_{0}^{\Phi^{-1}(\frac{1+\delta}{2})} x^2 \Phi'(x)\,dx \right| \le \varepsilon $$
with probability $1 - o(1)$, and the proof is complete.

\subsection{Proof of Corollary \ref{cor:lp}}

Before proving Corollary \ref{cor:lp}, we will need to establish the following bound.   

\begin{theorem}[Sub-gaussian entries: Lower bound] \label{thm:uniformsublow}
Let $\xi$, $\zeta$ be real sub-gaussian random variables with mean zero, and assume $\xi$ has unit variance.  Let $W$ be an $n \times n$ Wigner matrix with atom variables $\xi$, $\zeta$.  
Then there exist constants $C,c > 0$ and $0 < \eta, \delta < 1$ (depending only on the sub-gaussian moments of $\xi$ and $\zeta$) such that
$$ \min_{1 \leq j \leq n} \min_{S \subset [n] : |S| \geq \delta n} \|v_j(W) \|_S \geq \eta $$
with probability at least $1 - C \exp(-cn)$.  
\end{theorem}

Theorem \ref{thm:uniformsublow} follows directly from Theorem \ref{thm:RVgapsgen}.  However, Theorem \ref{thm:RVgapsgen} is much stronger than anything we need here.  Additionally, the proof of Theorem \ref{thm:RVgapsgen} is long and complicated, and we do not touch on it in this survey.  Instead, we provide a separate proof of Theorem \ref{thm:uniformsublow} using the uniform bounds below.  

\begin{theorem}[Uniform upper bound] \label{thm:uniformupper}
Let $\xi, \zeta$ be real random variables, and let $W$ be an $n \times n$ Wigner matrix with atom variables $\xi$, $\zeta$.  Take $1 \leq m < n$.  Then, for any $0 < \eta < 1$ and $K > 0$, 
\begin{align*}
	\Prob &\left( \exists j \in [n] \text{ and } S \subset [n] \text{ with } |S| = m \text{ such that } \|v_j(W)\|_{S}^2 \geq \eta \right)  \\
		&\qquad\qquad\qquad\leq \Prob \left( \|W\| > K\sqrt{n} \right) + n \binom{n}{m} \Prob \left( \frac{1}{n} \sigma^2_m(B) \leq \frac{4 (1-\eta) K^2}{\eta} \right),
\end{align*}
where $B$ is a $(n-m) \times m$ matrix whose entries are iid copies of $\xi$.  
\end{theorem}
\begin{proof}
We first observe that, by changing $\eta$ to $1- \eta$, it suffices to show
\begin{align*}
	\Prob &\left( \exists j \in [n] \text{ and } S \subset [n] \text{ with } |S| = m \text{ such that } \|v_j(W)\|_{S}^2 \geq 1- \eta \right)  \\
		&\qquad\qquad\qquad\leq \Prob \left( \|W\| > K\sqrt{n} \right) + n \binom{n}{m} \Prob \left( \frac{1}{n} \sigma^2_m(B) \leq \frac{4 \eta K^2}{1 - \eta} \right).
\end{align*}
For notational convenience, let $\lambda_1 \leq \cdots \leq \lambda_n$ denote the eigenvalues of $W$ with corresponding unit eigenvectors $v_1, \ldots v_n$.  Define the event
$$ \Omega_{n,m}(\eta) := \left\{ \exists j \in [n] \text{ and } S \subset [n] \text{ with } |S| = m \text{ such that } \|v_j\|_{S}^2 \geq 1- \eta \right\}. $$
Then 
\begin{equation} \label{eq:initbnd}
	\Prob (\Omega_{n,m}(\eta)) \leq \Prob \left( \Omega_{n,m}(\eta) \cap \{ \|W\| \leq K \sqrt{n} \} \right) + \Prob \left( \|W\| > K \sqrt{n} \right). 
\end{equation}
By the union bound and symmetry, we have
\begin{align*}
	\Prob &\left( \Omega_{n,m}(\eta) \cap \{ \|W\| \leq K \sqrt{n} \} \right) \\
		&\leq \sum_{j=1}^n \Prob \left( \exists S \subset [n] \text{ with } |S| = m \text{ such that } \|v_j\|_{S^c}^2 \leq \eta \text{ and } \|W\| \leq K \sqrt{n} \right) \\
		&\leq \binom{n}{m} \sum_{j=1}^n \Prob \left( \sum_{k=m+1}^n |v_j(k)|^2 \leq \eta \text{ and } \|W\| \leq K \sqrt{n} \right),
\end{align*}
where $v_j(k)$ denotes the $k$th entry of the unit vector $v_j$.  

Write
$$ W = \begin{pmatrix} A & B^\mathrm{T} \\ B & D \end{pmatrix}, $$
where $A$ is a $m \times m$ matrix, $B$ is a $(n-m) \times m$ matrix, and $D$ is a $(n-m) \times (n-m)$ matrix.  In particular, the entries of $B$ are iid copies of $\xi$.  Decompose 
$$ v_j = \begin{pmatrix} x_j \\ y_j \end{pmatrix}, $$
where $x_j$ is a $m$-vector and $y_j$ is a $(n-m)$-vector.  Then the eigenvalue equation $W v_j = \lambda_j v_j$ implies that $B x_j + D y_j = \lambda_j y_j$.  Therefore, on the event where $\|y_j\|^2 \leq \eta$ and $\|W\| \leq K \sqrt{n}$, we have
\begin{align*}
	\|B x_j\|^2 = \|\lambda_j y_j - D y_j \|^2 \leq 2 \|y_j\|^2 \left( |\lambda_j|^2 + \|D\|^2 \right) \leq 4 \eta K^2 n.
\end{align*}
Here we used the fact that the spectral norm of a matrix is not less than that of any sub-matrix.  Since $\|x_j\|^2 + \|y_j\|^2 = 1$, we find that
$$ (1-\eta) \sigma_m^2(B) \leq \|B x_j \|^2. $$
Combining the bounds above, we conclude that
\begin{align*}
	\Prob \left( \Omega_{n,m}(\eta) \cap \{ \|W\| \leq K \sqrt{n} \} \right) &\leq \binom{n}{m} \sum_{j=1}^n \Prob \left( \|y_j\|^2 \leq \eta \text{ and } \|W\| \leq K \sqrt{n} \right) \\
		&\leq \binom{n}{m} \sum_{j=1}^n \Prob \left( (1-\eta) \sigma_m^2(B) \leq 4 \eta K^2 n \right) \\
		&\leq n \binom{n}{m} \Prob \left( \frac{1}{n} \sigma_m^2(B) \leq \frac{4 \eta K^2}{1 - \eta}  \right).
\end{align*}
The proof is now complete by combining the bound above with \eqref{eq:initbnd}.  
\end{proof} 

\begin{theorem}[Uniform lower bound] \label{thm:uniformlower}
Let $\xi, \zeta$ be real random variables, and let $W$ be an $n \times n$ Wigner matrix with atom variables $\xi$, $\zeta$.  Take $1 \leq m < n$.  Then, for any $0 < \eta < 1$ and $K > 0$, 
\begin{align*}
	\Prob &\left( \exists j \in [n] \text{ and } S \subset [n] \text{ with } |S| = m \text{ such that } \|v_j(W)\|_{S}^2 \leq \eta \right)  \\
		&\qquad\qquad\qquad\leq \Prob \left( \|W\| > K \sqrt{n} \right) + n \binom{n}{m} \Prob \left( \frac{1}{n} \sigma^2_{n-m}(B) \leq \frac{4 \eta K^2}{1 - \eta} \right),
\end{align*}
where $B$ is a $m \times (n-m)$ matrix whose entries are iid copies of $\xi$.
\end{theorem}
\begin{proof}
From the relation 
$$ \|v_j(W) \|_S^2 + \|v_j(W)\|_{S^c}^2 = 1, $$
we find that
\begin{align*}
	\Prob &\left( \exists j \in [n] \text{ and } S \subset [n] \text{ with } |S| = m \text{ such that } \|v_j(W)\|_{S}^2 \leq \eta \right) \\
		&= \Prob \left( \exists j \in [n] \text{ and } S \subset [n] \text{ with } |S| = n-m \text{ such that } \|v_j(W)\|_{S}^2 \geq 1- \eta \right).
\end{align*}
Therefore, Theorem \ref{thm:uniformlower} follows immediately from Theorem \ref{thm:uniformupper}. 
\end{proof}

Before proving Theorem \ref{thm:uniformsublow}, we prove the following upper bound.

\begin{theorem}[Sub-gaussian entries: Upper bound] \label{thm:uniformsub}
Let $\xi$, $\zeta$ be sub-gaussian random variables with mean zero, and assume $\xi$ has unit variance.  Let $W$ be an $n \times n$ Wigner matrix with atom variables $\xi$, $\zeta$.  Then there exist constants $C,c > 0$ and $0 < \eta, \delta < 1$ (depending only on the sub-gaussian moments of $\xi$ and $\zeta$) such that
$$ \max_{1 \leq j \leq n} \max_{S \subset [n] : |S| \leq \delta n} \|v_j(W) \|_S \leq \eta $$
with probability at least $1 - C \exp(-cn)$.  
\end{theorem}
\begin{proof}
We observe that it suffices to show 
$$  \max_{1 \leq j \leq n} \max_{S \subset [n] : |S| = \lfloor \delta n \rfloor} \|v_j(W) \|_S \leq \eta $$
with probability at least $1 - C n \exp(-c n)$.  Moreover, in view of Theorem \ref{thm:uniformupper} and Lemma \ref{lemma:norm}, it suffices to show there exists $0 < \eta, \delta < 1$ such that
$$ n \binom{n}{\lfloor \delta n \rfloor} \Prob \left( \frac{1}{n} \sigma_{\lfloor \delta n \rfloor}^2(B) \leq \frac{4C_0^2 (1-\eta)}{\eta} \right) \leq C n \exp(-cn), $$
where $B$ is a $(n - \lfloor \delta n \rfloor) \times \lfloor \delta n \rfloor$ matrix whose entries are iid copies of $\xi$.  

Fix $0 < c_0 < 1$.  Let $c>0$ be the corresponding constant from Theorem \ref{thm:lsv2} which depends only on the sub-gaussian moment of $\xi$.  Let $0 < \delta < 1/2$ be such that
\begin{align*}
	\sqrt{n- \lfloor \delta n \rfloor} - \sqrt{ \lfloor \delta n \rfloor -1} &\geq c_0 \sqrt{n}, \\
	n-2 \lfloor \delta n \rfloor + 1 &\geq c_0 n, 
\end{align*}
and
\begin{equation} \label{eq:delta'}
	\delta \log \left( \frac{2e}{\delta} \right) \leq c_0 \frac{c}{2}.
\end{equation}
We observe that such a choice is always possible since $\delta' \log(1/\delta') \to 0$ as $\delta' \downarrow 0$.  Set $m := \lfloor \delta n \rfloor$.  

By Theorem \ref{thm:lsv2} and our choice of $\delta$, there exists $c_1 > 0$ (depending on $c_0$, $c$, and the sub-gaussian moments of $\xi$) such that
$$ \Prob \left( \sigma^2_m(B) \leq c_1 n \right) \leq 2 \exp(-c_0 c n). $$
Taking $0 < \eta < 1$ sufficiently close to $1$ (so that $\frac{4C_0^2(1-\eta)}{\eta} < c_1$), we obtain 
\begin{align*}
	n \binom{n}{m} \Prob \left( \frac{1}{n} \sigma_m^2(B) \leq \frac{4C_0^2(1-\eta)}{\eta} \right) &\leq 2 n \binom{n}{m} \exp(-c_0 cn) \\
		&\leq 2n \left( \frac{n e}{m} \right)^m \exp(-c_0 cn) \\
		&\leq 2n \exp\left( m \log \left( \frac{ne}{m} \right) - c_0 cn \right).
\end{align*}
The claim now follows since
$$ m \log \left( \frac{ne}{m} \right) \leq \delta n \log \left( \frac{2e}{\delta} \right) \leq c_0\frac{c}{2} n $$
by \eqref{eq:delta'} for $n$ sufficiently large.  
\end{proof}

Using Theorem \ref{thm:uniformsub}, we immediately obtain the proof of Theorem \ref{thm:uniformsublow}.  

\begin{proof}[Proof of Theorem \ref{thm:uniformsublow}]
Theorem \ref{thm:uniformsublow} follows immediately from Theorem \ref{thm:uniformsub} and the relation
$$ \min_{1 \leq j \leq n} \min_{\substack{S \subset [n] \\ |S| \geq m}} \|v_j(W)\|_S^2 + \max_{1 \leq j \leq n} \max_{\substack{S \subset [n] \\ |S| \leq n-m}} \|v_j(W)\|_{S}^2 = 1. $$
We also note that one can prove Theorem \ref{thm:uniformsublow} directly by applying Theorem \ref{thm:uniformlower}.  
\end{proof}

We now present the proof of Corollary \ref{cor:lp}.  Indeed, Corollary \ref{cor:lp} follows immediately from Theorem \ref{thm:uniformsublow}, Theorem \ref{thm:uniformsub}, and the two lemmas below.  

\begin{lemma}[Uniform lower bound implies upper bound on the $\ell^p$-norm] \label{lemma:upperellp}
Let $v$ be a unit vector in $\mathbb{R}^n$, and let $1 \leq m \leq n$.  If 
$$ \min_{S \subset [n] : |S| = m} \|v\|_S^2 \geq \eta, $$
then, for any $1 \leq p \leq 2$, 
$$ \|v\|_{\ell^p}^p \leq \left\lceil \frac{n}{n-m} \right\rceil \frac{ 1-\eta }{\eta^{1-p/2}} m^{1-p/2}. $$
\end{lemma}
\begin{proof}
Let $S \subset [n]$ be the subset which minimizes $\|v\|_S^2$ given the constraint $|S| = m$.  In particular, $S$ contains the $m$-smallest coordinates of $v=(v_i)_{i=1}^n$ in absolute value.  Let $l$ be such that
$$ |v_l| = \max\{|v_j| : j \in S \}. $$
Then
$$ |S| |v_l|^2 \geq \sum_{j \in S} |v_j|^2 \geq \eta $$
by assumption.  Hence, we obtain
$$ |v_l|^2 \geq \frac{\eta}{|S|}. $$
Therefore, we conclude that, for $1 \leq p \leq 2$, 
\begin{align*}
	1-\eta \geq \sum_{j \in S^c} |v_j|^2 \geq |v_l|^{2-p} \sum_{j \in S^c} |v_j|^p \geq \left( \frac{\eta}{|S|} \right)^{1 - p/2} \sum_{j \in S^c} |v_j|^p.
\end{align*}
Since $|S| = m$, and $S$ was the minimizer, we have
$$ \max_{T \subset [n] : |T| = n-m} \sum_{j \in T} |v_j|^p \leq \frac{ 1-\eta }{\eta^{1-p/2}} m^{1-p/2}. $$
Let $T_1, \ldots, T_{k_0}$ be a partition of $[n]$ such that $|T_k| \leq n-m$ for $1 \leq k \leq k_0$, and $k_0 \leq \left\lceil \frac{n}{n-m} \right\rceil$.  We then find that
$$ \sum_{j = 1}^n |v_j|^p = \sum_{k=1}^{k_0} \sum_{j \in T_k} |v_j|^p \leq \left\lceil \frac{n}{n-m} \right\rceil \frac{ 1-\eta }{\eta^{1-p/2}} m^{1-p/2}, $$
and the proof is complete.
\end{proof}

\begin{lemma}[Uniform upper bound implies lower bound on the $\ell^p$-norm] \label{lemma:lowerellp}
Let $v$ be a unit vector in $\mathbb{R}^n$, and let $1 \leq m \leq n$.  If 
$$ \max_{S \subset [n] : |S| = m} \|v\|_S^2 \leq \eta, $$
then, for any $1 \leq p \leq 2$, 
$$ \|v\|_{\ell^p}^p \geq \frac{1- \eta}{\eta^{1- p/2}} m^{1-p/2}. $$
\end{lemma}
\begin{proof}
Let $S \subset [n]$ be the subset which maximizes $\|v\|_S^2$ given the constraint $|S| = m$.  In particular, $S$ contains the $m$-largest coordinates of $v=(v_i)_{i=1}^n$ in absolute value.  Let $l$ be such that
$$ |v_l| = \min\{|v_j| : j \in S \}. $$
Then
$$ |S| |v_l|^2 \leq \sum_{j \in S} |v_j|^2 \leq \eta $$
by assumption.  Hence, we obtain
$$ |v_l|^2 \leq \frac{\eta}{|S|}. $$
Therefore, we conclude that, for $1 \leq p \leq 2$, 
\begin{align*}
	1 - \eta \leq \sum_{j \in S^c} |v_j|^2 \leq |v_l|^{2-p} \sum_{j \in S^c} |v_j|^p \leq \left( \frac{\eta}{|S|} \right)^{1 - p/2} \sum_{j=1}^n |v_j|^p.
\end{align*}
Since $|S| = m$, the proof is complete.  
\end{proof}

\section{Proofs for random matrices with non-zero mean}\label{sec:proof:nonzero}

This section is devoted to the proofs of Theorems \ref{thm:Gnpsmallcor}, \ref{thm:diag}, \ref{thm:perturb} and \ref{thm:singlerank1}.  

\subsection{Proof of Theorem \ref{thm:Gnpsmallcor}}
The proof of Theorem \ref{thm:Gnpsmallcor} is similar to the proof of Theorem \ref{lower}.  We begin with a result from \cite{NTV}.

\begin{theorem}[\cite{NTV}] \label{thm:NTVGnp}
Let $A_n$ be the adjacency matrix of $G(n,p)$ for some fixed value of $p \in (0,1)$.  Let
$$ A_n = \begin{bmatrix} A_{n-1} & X \\ X^\mathrm{T} & 0 \end{bmatrix}, $$
where $A_{n-1}$ is the upper-left $(n-1) \times (n-1)$ minor of $A_n$ and $X \in \{0,1\}^{n-1}$.  Then, for any $\alpha > 0$, there exists a constant $C$ (depending on $p, \alpha$) such that, for any $\delta > n^{-\alpha}$, 
$$ \sup_{1 \leq i \leq n-1} \Prob \left( \lambda_i(A_{n-1}) - \lambda_i(A_n) \leq \frac{\delta}{\sqrt{n}} \right) \leq C n^{o(1)} \delta + o(1) $$
and
$$ \sup_{2 \leq i \leq n} \Prob \left( \lambda_{i}(A_n) - \lambda_{i-1}(A_{n-1}) \leq \frac{\delta}{\sqrt{n}} \right) \leq C n^{o(1)} \delta + o(1). $$
Moreover, 
$$ \inf_{1 \leq j \leq n-1} |v_j(A_{n-1})^\mathrm{T} X | > 0 $$
with probability $1 - o(1)$.  Here, the rate of convergence to zero implicit in the $o(1)$ terms depends on $p$ and $\alpha$. 
\end{theorem}

Theorem \ref{thm:NTVGnp} follows from \cite[Theorem 2.7]{NTV} and the arguments given in \cite[Section 7]{NTV} and \cite{TVsimple}.   We will also need the following version of Theorem \ref{thm:dim2} generalized to the adjacency matrix $A_n(p)$.  

\begin{lemma} \label{lemma:dimGnp}
Let $A_n(p)$ be the adjacency matrix of $G(n,p)$ for some fixed value of $p \in (0,1)$.  Let $N_I$ denote the number of eigenvalues of $\frac{1}{\sqrt{n p (1-p)}} A_{n}(p)$ in the interval $I$.  Then there exist constants $C, c, c' > 0$ (depending on $p$) such that, for any interval $I \subset \mathbb{R}$, 
$$ \Prob \left( N_I \geq \frac{(1.1)}{\pi} n |I|  + (\log n)^{c'} \right) \leq C \exp(-c \log^2 n), $$
where $|I|$ denotes the length of $I$.  
\end{lemma}
\begin{proof}
We observe that $A_n(p)$ has the same distribution as $W_n + p(\BJ_n - I_n)$, where $W_n$ is the Wigner matrix whose diagonal entries are zero and whose upper-triangular entries are iid copies of the random variable defined in \eqref{eq:def:Bernoullip}, $\BJ_n$ is the all-ones matrix, and $I_n$ is the identity matrix.  Thus, it suffices to prove the result for $W_n + p(\BJ_n - I_n)$.  Accordingly, redefine $N_I$ to be the number of eigenvalues of $\frac{1}{\sqrt{n p(1-p) } } (W_n + p(\BJ_n - I_n))$ in the interval $I$.  

For $\gamma \in \mathbb{R}$ and any interval $I$, we let $I+\gamma$ denote the interval $I$ shifted to the right by $\gamma$ units.  Clearly, both $I$ and $I + \gamma$ have the same length, i.e. $|I| = |I + \gamma|$.  Set $\gamma := p / \sqrt{n p(1-p)}$, and let $I \subset \mathbb{R}$ be any interval.  It follows from Theorem \ref{thm:dim2} that the number of eigenvalues of $\frac{1}{\sqrt{n p (1-p)}} W_n$ in the interval $I + \gamma$ is at most
$$ (1.1) n \int_{I+\gamma} \rho_{\mathrm{sc}}(x) dx + (\log n)^{c'} \leq \frac{1.1}{\pi} n |I| + (\log n)^{c'} $$
with probability at least $1 - C \exp(-c \log^2 n)$.  The inequality above follows from bounding the semicircle density $\rho_{\mathrm{sc}}$ by $\frac{1}{\pi}$ (and the fact that the intervals $I$ and $I+\gamma$ have the same length).  The normalization factor $\sqrt{p(1-p)}$ ensures the entries of the matrix have unit variance as required by Theorem \ref{thm:dim2}.  

Since $\BJ_n$ is rank one, it follows from eigenvalue interlacing (see, for instance, \cite[Exercise III.2.4]{Bhatia}) that the number of eigenvalues of $\frac{1}{\sqrt{n p (1-p)}} (W_n + p \BJ_n)$ in the interval $I + \gamma$ is at most 
$$ \frac{1.1}{\pi} n |I| + (\log n)^{c'} + 1 \leq \frac{1.1}{\pi} n |I| + (\log n)^{c'+1} $$
for $n$ sufficiently large.  Subtracting $\gamma I_n$ from a matrix only shifts the eigenvalues of the matrix by $\gamma$.  Hence, we conclude that
$$ N_I \leq \frac{1.1}{\pi} n |I| + (\log n)^{c'+1} $$
with probability at least $1- C \exp(-c \log^2 n)$.  
\end{proof}

We will also need the following technical lemma. 

\begin{lemma} \label{lemma:techsumbndGnp}
Let $A_n$ be the adjacency matrix of $G(n,p)$ for some fixed value of $p \in (0,1)$.  Let
$$ A_n = \begin{bmatrix} A_{n-1} & X \\ X^\mathrm{T} & 0 \end{bmatrix}, $$
where $A_{n-1}$ is the upper-left $(n-1) \times (n-1)$ minor of $A_n$ and $X \in \{0,1\}^{n-1}$.  Let $\kappa > 0$.  Then there exists a constant $c_1 > 0$ such that, for any $1 \leq i \leq n-1$ and any $e^{-n^\kappa} < \delta < 1$, 
\begin{equation} \label{eq:sumhelpbndGnp}
	\sum_{j=1}^{n-1} \frac{ |v_j(A_{n-1})^\mathrm{T} X|^2 }{ | \lambda_j(A_{n-1}) - \lambda_i(A_n) |^2} \leq \frac{ (\log n)^{c_1} }{ 25 } \left[ \frac{ 1} { m_i^2 } + \frac{n}{ \delta^2} \right] 
\end{equation}
with probability $1 - o(1)$, where
$$ m_i := \min_{1 \leq j \leq n-1} |\lambda_j(A_{n-1}) - \lambda_i(A_n)|. $$  
Here, the rate of convergence to zero implicit in the $o(1)$ term depends on $p$ and $\kappa$.
\end{lemma}
\begin{proof}
Fix $1 \leq i \leq n-1$ and $e^{-n^\kappa} < \delta < 1$.  Define the event $\Omega$ to be the intersection of the events
\begin{align} \label{eq:eventvjWn1Gnp}
	\bigcap_{j=1}^{n-2} &\left\{ |v_j(A_{n-1})^\mathrm{T} X|^2 \leq \frac{p(1-p)}{1000} \log^2 n \right\}, \\
	&\left\{ \inf_{1 \leq j \leq n-1} |v_{j}(A_{n-1})^\mathrm{T} X |^2 > 0 \right\}, \label{eq:lowerbndipvj}
\end{align}
and
\begin{equation} \label{eq:normC0bndGnp}
	\left\{ \|A_n \| \leq 10np \right\}.
\end{equation}  

We claim that 
\begin{equation} \label{eq:Omegaclog2nGnp}
	\Prob(\Omega) = 1 - o(1). 
\end{equation}
Note that the event in \eqref{eq:normC0bndGnp} can be shown to hold with probability $1 - o(1)$ by combining the bounds in \cite{FK} with Hoeffding's inequality.  Additionally, the event in \eqref{eq:lowerbndipvj} holds with probability $1 - o(1)$ by Theorem \ref{thm:NTVGnp}.  We now consider the events in \eqref{eq:eventvjWn1Gnp}.  For $1 \leq j \leq n-2$, by the Cauchy--Schwarz inequality, we obtain
\begin{align*}
	|v_{j}(A_{n-1})^\mathrm{T} X | &\leq |v_{j}(A_{n-1})^\mathrm{T} (X - p \Bj)| + | v_{j}(A_{n-1})^\mathrm{T} (p \Bj - p n^{1/2} v_{n-1}(A_{n-1}) ) | \\
	&\leq |v_{j}(A_{n-1})^\mathrm{T} (X - p \Bj)| + p n^{1/2} \| n^{-1/2} \Bj - v_{n-1}(A_{n-1}) \| \\
	&\leq |v_{j}(A_{n-1})^\mathrm{T} (X - p \Bj)| + p n \| n^{-1/2} \Bj - v_{n-1}(A_{n-1}) \|_{\ell^\infty}
\end{align*}
since $v_j(A_{n-1})$ is orthogonal to $v_{n-1}(A_{n-1})$.  By Theorem \ref{thm:Mitra}, it follows that 
$$ p n \| n^{-1/2} \Bj - v_{n-1}(A_{n-1}) \|_{\ell^\infty} \leq C_0 \sqrt{\log n} $$
with probability $1- o(1)$.  We now observe that the vector $X - p \Bj$ has mean zero and $|v_{j}(A_{n-1})^\mathrm{T} (X - p \Bj)|$ is the length of the projection of the vector $X - p \Bj$ onto the one-dimensional subspace spanned by $v_j(A_{n-1})$.   Hence, by Corollary \ref{cor:projection} and the union bound, it follows that
$$ \sup_{1 \leq j \leq n-2} |v_{j}(A_{n-1})^\mathrm{T} (X - p \Bj)|^2 \leq \frac{p(1-p)}{1000^2} \log^2 n $$
with probability $1 - o(1)$.  Combining the estimates above yields the bound in \eqref{eq:Omegaclog2nGnp}.  

It now suffices to show that, conditionally on $\Omega$, the bound in \eqref{eq:sumhelpbndGnp} holds with probability $1 - o(1)$.  It follows (see Remark \ref{rem:strictinter} for details) that, on $\Omega$, the eigenvalues of $A_{n-1}$ strictly interlace with the eigenvalues of $A_n$, and hence the terms
$$ \sum_{j=1}^{n-1} \frac{ |v_j(A_{n-1})^\mathrm{T} X|^2 }{ | \lambda_j(A_{n-1}) - \lambda_i(A_n) |^2} $$
and $m_i^{-2}$ are well-defined.  

It follows from the results in \cite{FK} that 
$$ |\lambda_{n-1}(A_{n-1}) - \lambda_{i}(A_n) | \geq \frac{1}{100} np $$
with probability $1- o(1)$.  Hence, we obtain
$$ \frac{ |v_{n-1}(A_{n-1})^\mathrm{T} X|^2 }{ |\lambda_{n-1}(A_{n-1}) - \lambda_i(A_n) |^2 } \leq \frac{ 100^2 \|X\|^2 }{ n^2 p^2 } \leq \frac{100^2 }{ n p^2 }. $$
Thus, on $\Omega$, we have
$$ \sum_{j=1}^{n-1} \frac{ |v_j(A_{n-1})^\mathrm{T} X|^2 }{ | \lambda_j(A_{n-1}) - \lambda_i(A_n) |^2}  \leq \frac{p(1-p)}{1000} \log^2 n \sum_{j=1}^{n-2} \frac{1}{ | \lambda_j(W_{n-1}) - \lambda_i(W_n) |^2} + \frac{100^2 }{ n p^2 } $$
with probability $1-o(1)$.  Therefore, conditionally on $\Omega$, it suffices to show that
$$ \sum_{j=1}^{n-2} \frac{1}{ | \lambda_j(A_{n-1}) - \lambda_i(A_n) |^2} \leq \frac{10}{p(1-p)} (\log n)^{c_1} \left[ \frac{1}{m_i^2} + \frac{n}{\delta^2} \right] $$
with probability  $1 - o(1)$.  

Define the sets
$$ T := \left\{ 1 \leq j \leq n-2 : \frac{|\lambda_j(A_{n-1}) - \lambda_i(A_{n})|}{ \sqrt{p(1-p)} } <  \frac{\delta}{\sqrt{n}} \right\} $$
and
$$ T_l := \left\{ 1 \leq j \leq n-2 : 2^l \frac{\delta}{\sqrt{n}} \leq \frac{ |\lambda_j(A_{n-1}) - \lambda_i(A_n)|}{ \sqrt{p(1-p)} } < 2^{l+1} \frac{\delta}{\sqrt{n}} \right\} $$
for $l = 0, 1, \ldots, L$, where $L$ is the smallest integer such that $2^L \sqrt{p (1-p) } \frac{\delta}{\sqrt{n}} \geq 20 np$.  In particular, as $\delta > e^{-n^\kappa}$, we obtain $L = O(n^{-\kappa})$.  On the event $\Omega$, it follows that every index $1 \leq j \leq n-2$ is contained in either $T$ or $\cup_{l=0}^L T_l$.  By Lemma \ref{lemma:dimGnp} and the union bound, there exists $C, c , c_1 > 0$ such that
$$ |T| \leq (\log n)^{c_1} $$
and 
$$ |T_l| \leq 4 (2^l ) + (\log n)^{c_1}, \quad l = 0, \ldots, L $$
with probability at least $1 - C \exp(-c \log^2 n)$.  Hence, on this same event, we conclude that
\begin{align*}
	\sum_{j=1}^{n-2} \frac{p(1-p)}{ | \lambda_j(A_{n-1}) - \lambda_i(A_n) |^2} &\leq \sum_{j \in T} \frac{p(1-p)}{ | \lambda_j(A_{n-1}) - \lambda_i(A_n) |^2} \\ 
	&\qquad\qquad\qquad + \sum_{l=0}^L \sum_{j \in T_l} \frac{p(1-p)}{ | \lambda_j(A_{n-1}) - \lambda_i(A_n) |^2} \\
	&\leq \frac{|T|}{m_i^2} + \frac{n}{\delta^2} \sum_{l=0}^L \frac{|T_l|}{ 2^{2l} } \\
	&\leq \frac{ (\log n)^{c_1} }{ m_i^2 } + \frac{10n (\log n)^{c_1} }{\delta^2}
\end{align*}
for $n$ sufficiently large.  The proof of the lemma is complete.  
\end{proof}

With Theorem \ref{thm:NTVGnp} and Lemma \ref{lemma:techsumbndGnp} in hand, we are now ready to prove Theorem \ref{thm:Gnpsmallcor}.

\begin{proof}[Proof of Theorem \ref{thm:Gnpsmallcor}]
Let $A_n$ be the adjacency matrix of $G(n,p)$.  We will bound the $j$th coordinate of the unit eigenvector $v_i$ in magnitude from below.  By symmetry, it suffices to consider the case when $j=n$.  For this reason, we decompose $A_n$ as 
$$ A_n = \begin{bmatrix} A_{n-1} & X \\ X^\mathrm{T} & 0 \end{bmatrix}, $$
where $A_{n-1}$ is the upper-left $(n-1) \times (n-1)$ minor of $A_n$ and $X \in \{0,1\}^{n-1}$.  

Let $1 \leq i \leq n-1$.  Let $\alpha > 0$, and assume $1 > \delta > n^{-\alpha}$ as the claim is trivial when $\delta \geq 1$.  Define the event $\Omega$ to be the intersection of the events
$$ \left\{ \inf_{1 \leq j \leq n-1} |v_j(A_{n-1})^\mathrm{T} X| > 0 \right\} \bigcap  \left\{ m_i > \frac{\delta}{\sqrt{n}} \right\} $$
and
$$ \left\{ \sum_{j=1}^{n-1} \frac{ |v_j(A_{n-1})^\mathrm{T} X|^2 }{ |\lambda_j(A_{n-1}) - \lambda_i(A_n)|^2 } \leq \frac{ (\log n)^{c_1} }{ 25 } \left[ \frac{ 1} { m_i^2 } + \frac{n}{ \delta^2} \right] \right\}, $$
where
$$ m_i := \min_{1 \leq j \leq n-1} |\lambda_j(A_{n-1}) - \lambda_i(A_n)| $$  
and $c_1$ is the constant from Lemma \ref{lemma:techsumbndGnp}.  If $m_i \leq \frac{\delta}{\sqrt{n}}$, then, by Cauchy's interlacing inequalities \eqref{eq:cauchy}, this implies that either 
$$ |\lambda_i(A_{n-1}) - \lambda_i(A_n)| \leq \frac{\delta}{\sqrt{n}}  $$
or 
$$ |\lambda_{i-1}(A_{n-1}) - \lambda_i(A_n)| \leq \frac{\delta}{\sqrt{n}}. $$
(Here, the second possibility can only occur if $i > 1$.)  
Therefore, it follows from Theorem \ref{thm:NTVGnp} and Lemma \ref{lemma:techsumbndGnp} that, there exists $C > 0$ such that
$\Omega$ holds with probability at least $1 - C n^{o(1)} \delta - o(1)$.  

Let $c_2 := c_1 / 2$.  Then
$$ \Prob \left( |v_i(n)| \leq \frac{\delta}{\sqrt{n} (\log n)^{c_2} } \right) \leq \Prob (\Omega^c) + \Prob \left( \left\{ |v_i(n)| \leq \frac{\delta}{\sqrt{n} (\log n)^{c_2} } \right\} \bigcap \Omega \right). $$
Hence, to complete the proof, we will show that the event
\begin{equation} \label{eq:eventsufficeGnp}
	\left\{ |v_i(n)| \leq \frac{\delta}{\sqrt{n} (\log n)^{c_2} } \right\} \bigcap \Omega 
\end{equation}
is empty.  

Fix a realization in this event.  As discussed above (see Remark \ref{rem:strictinter}), on $\Omega$, the eigenvalues of $A_{n-1}$ strictly interlace with the eigenvalues of $A_n$.  Hence, we apply Lemma \ref{lemma:coordinate} and obtain
$$ \frac{1}{1 + \sum_{j=1}^{n-1} \frac{ |v_j(A_{n-1})^\mathrm{T} X |^2}{ | \lambda_j(A_{n-1}) - \lambda_i(A_n) |^2 } } = |v_i(n)|^2 \leq \frac{\delta^2}{n (\log n)^{2 c_2}}. $$
This implies, for $n$ sufficiently large, that
$$ \frac{n (\log n)^{2 c_2 }}{2 \delta^2} \leq \sum_{j=1}^{n-1} \frac{ |v_j(A_{n-1})^\mathrm{T} X |^2}{ | \lambda_j(A_{n-1}) - \lambda_i(A_n) |^2 }. $$
On the other hand, by definition of $\Omega$, we have
\begin{align*}
	\sum_{j=1}^{n-1} \frac{ |v_j(A_{n-1})^\mathrm{T} X |^2}{ | \lambda_j(A_{n-1}) - \lambda_i(A_n) |^2 } \leq \frac{2 (\log n)^{c_1} } { 25} \frac{n} {\delta^2},
\end{align*}
and thus
$$ \frac{n (\log n)^{2 c_2}}{2 \delta^2} \leq  \frac{2 (\log n)^{c_1} } { 25} \frac{n} {\delta^2}. $$
This is a contradiction since $2 c_2 = c_1$.  We conclude that the event in \eqref{eq:eventsufficeGnp} is empty, and the proof is complete.  
\end{proof}

\subsection{Proof of Theorem \ref{thm:diag}}

To begin, we introduce $\varepsilon$-nets as a convenient way to discretize a compact set.  

\begin{definition}[$\eps$-net]
Let $(X,d)$ be a metric space, and let $\eps > 0$.  A subset $\mathcal{N}$ of $X$ is called an $\eps$-net of $X$ if every point $x \in X$ can be approximated to within $\eps$ by some point $y \in \mathcal{N}$.  That is, for every $x \in X$ there exists $y \in \mathcal{N}$ so that $d(x,y) \leq \eps$.  
\end{definition}

The following estimate for the maximum size of an $\varepsilon$-net of a sphere is well-known. 

\begin{lemma} \label{lemma:net}
A unit sphere in $d$ dimensions admits an $\varepsilon$-net of size at most 
$$ \left(1+\frac{2}{\varepsilon} \right)^d.$$
\end{lemma}
\begin{proof}
Let $S$ be the unit sphere in question.  Let $\mathcal{N}$ be a maximal $\varepsilon$-separated subset of $S$.  That is, $\|x-y\| \geq \varepsilon$ for all distinct $x,y \in \mathcal{N}$ and no subset of $S$ containing $\mathcal{N}$ has this property.  Such a set can always be constructed by starting with an arbitrary point in $S$ and at each step selecting a point that is at least $\varepsilon$ distance away from those already selected.  Since $S$ is compact, this procedure will terminate after a finite number of steps.  

We now claim that $\mathcal{N}$ is an $\varepsilon$-net.  Suppose to the contrary.  Then there would exist $x \in S$ that is at least $\varepsilon$ from all points in $\mathcal{N}$.  In other words, $\mathcal{N}\cup\{x\}$ would still be an $\varepsilon$-separated subset of $S$.  This contradicts the maximal assumption above.  

We now proceed by a volume argument.  At each point of $\mathcal{N}$ we place a ball of radius $\varepsilon/2$.  By the triangle inequality, it is easy to verify that all such balls are disjoint and lie in the ball of radius $1+\varepsilon/2$ centered at the origin.  Comparing the volumes give
$$ |\mathcal{N}| \leq \frac{(1+\varepsilon/2)^d}{(\varepsilon/2)^d} = \left( 1 + \frac{2}{\varepsilon} \right)^d. $$
\end{proof}

We will need the following lemmata in order to prove Theorem \ref{thm:diag}.  

\begin{lemma} \label{lemma:dimbnd}
Let $\xi$, $\zeta$ be real sub-gaussian random variables with mean zero, and assume $\xi$ has unit variance.  Let $W$ be an $n \times n$ Wigner matrix with atom variables $\xi$, $\zeta$.  Let $\tau \geq \tau_0 > 0$ and $K > 1$.  Then there exists constants $C,c,c' > 0$ (depending only on $\tau_0$, $K$, and the sub-gaussian moments of $\xi$ and $\zeta$) such that
$$ \sup_{\lambda \in [-K \sqrt{n}, K \sqrt{n}]} \left| \left\{1 \leq i \leq n : |\lambda_i(W) - \lambda| \leq \tau \sqrt{n} \right\} \right| \leq 4 \tau n + 2(\log n)^{c' \log \log n} $$
with probability at least 
$$ 1 - C \exp\left( - c (\log n)^{c \log \log n} \right). $$
\end{lemma}
\begin{proof}
The proof relies on Theorem \ref{thm:dim}.  In particular, for any interval $I \subset \mathbb{R}$, Theorem \ref{thm:dim} implies that 
\begin{align} \label{eq:dim2}
	\Prob &\left( \left| N_I - n\int_{I} \rho_{\mathrm{sc}}(x) dx \right| \geq (\log n)^{c' \log \log n} \right) \leq C \exp \left(- c (\log n)^{c \log \log n} \right) ,
\end{align}
where $N_I$ denotes the number of eigenvalues of $\frac{1}{\sqrt{n}} W$ in $I$.  Here $C, c, c' > 0$ depend only on the sub-gaussian moments of $\xi$ and $\zeta$.  

For each $\lambda \in \mathbb{R}$, let $I_{\lambda}$ be the interval
$$ I_{\lambda} := \left[ \frac{\lambda}{\sqrt{n}} - \tau, \frac{\lambda}{\sqrt{n}} + \tau \right]. $$
Thus, using the notation from above, the problem reduces to showing
$$ \sup_{\lambda \in [-K \sqrt{n}, K \sqrt{n}]} N_{I_\lambda} \leq 4 \tau n + 2(\log n)^{c' \log \log n} $$
with sufficiently high probability.  

Let $\mathcal{N}$ be a $n^{-1}$-net of the interval $[-100 K \sqrt{n}, 100 K \sqrt{n}]$.  Then $|\mathcal{N}| \leq C' n^{3/2}$, where $C' > 0$ depends only on $K$.  In addition, a simple net argument reveals that
$$ \sup_{\lambda \in [-K \sqrt{n}, K \sqrt{n}]} N_{I_\lambda}  \leq 2 \sup_{\lambda \in \mathcal{N}} N_{I_{\lambda}} $$
for all $n \geq n_0$, where $n_0$ depends only on $\tau_0$.  Thus, we have
\begin{align*}
	\Prob &\left( \sup_{\lambda \in [-K \sqrt{n}, K \sqrt{n}]} N_{I_\lambda} > 4 \tau n + 2 (\log n)^{c' \log \log n} \right) \\
	&\qquad\qquad\leq \Prob \left( \sup_{\lambda \in \mathcal{N}} N_{I_{\lambda}} > 2 \tau n + (\log n)^{c' \log \log n} \right) \\
	&\qquad\qquad\leq \sum_{\lambda \in \mathcal{N}} \Prob \left( \left | N_{I_\lambda} - n \int_{I_\lambda} \rho_{\mathrm{sc}}(x) dx \right| > (\log n)^{c' \log \log n} \right) 
\end{align*}
by the union bound.  The claim now follows by applying \eqref{eq:dim2}.  
\end{proof}

\begin{remark} \label{rem:perturb}
If $J$ is a real symmetric matrix with rank $k$, then, by \cite[Theorem III.2.1]{Bhatia}, it follows that 
$$ \left| \left\{1 \leq i \leq n : |\lambda_i(W+J) - \lambda| \leq \tau \sqrt{n} \right\}\right| \leq \left| \left\{1 \leq i \leq n : |\lambda_i(W) - \lambda| \leq \tau \sqrt{n} \right\}\right| + k. $$
\end{remark}

\begin{lemma} \label{lemma:largenet}
Let $\xi$, $\zeta$ be real sub-gaussian random variables with mean zero, and assume $\xi$ has unit variance.  Let $W$ be an $n \times n$ Wigner matrix with atom variables $\xi$, $\zeta$.  Let $k$ be a non-negative integer.  Let $J$ be a $n \times n$ deterministic real symmetric matrix with rank at most $k$.  In addition, let $H$ be a subspace of $\mathbb{R}^n$, which may depend only on $W$, that satisfies $\dim(H) \leq k$ almost surely.  Let $\tau_1 \geq \tau \geq \tau_0 > 0$.  For each $\lambda \in \mathbb{R}$, define the subspaces
$$ V_{\lambda} := \Span \left\{ v_i(W+J) : |\lambda_i(W+J) - \lambda| \leq \tau \sqrt{n} \right\} $$
and 
$$ H_{\lambda} := \Span \left\{ V_\lambda \cup H \right\}. $$
Then there exists constants $C, c, c' > 0$ (depending only on $\tau_0$, $\tau_1$, and the sub-gaussian moments of $\xi$ and $\zeta$) such that, with probability at least 
$$ 1 - C \exp\left( - c (\log n)^{c \log \log n} \right), $$
the following holds.  For every $\eps > 0$, there exists a subset $\mathcal{N}$ of the unit sphere in $\mathbb{R}^n$ (depending only on $W$, $J$, and $\eps$) such that
\begin{enumerate}[(i)]
\item \label{item:epsnet} for every $\lambda \in \mathbb{R}$ and every unit vector $w \in H_{\lambda}$, there exists $w' \in \mathcal{N}$ such that $|w - w'| \leq \eps$,
\item $|\mathcal{N}| \leq n^2 \left( 1 + \frac{2}{\eps} \right)^{4 \tau n + 2 (\log n)^{c' \log \log n} + 2k}$.
\end{enumerate}
\end{lemma}
\begin{proof}
By Lemma \ref{lemma:norm}, $\|W\| \leq C_0 \sqrt{n}$ with probability at least $1 - C_0 \exp(-c_0 n)$.  On this event, there exists some constant $K > 1$ (depending only on $C_0$ and $\tau_1$) such that if $|\lambda| \geq K \sqrt{n}$, then 
$$ \left| \left\{ 1 \leq i \leq n : |\lambda_i(W) - \lambda| \leq \tau \sqrt{n} \right\} \right| = 0. $$  
Thus, by applying Lemma \ref{lemma:dimbnd}, we conclude that
$$ \sup_{\lambda \in \mathbb{R}} \left| \left\{ 1 \leq i \leq n : |\lambda_i(W) - \lambda| \leq \tau \sqrt{n} \right\} \right| \leq 4 \tau n + 2 (\log n)^{c' \log \log n} $$
with probability at least 
$$ 1 - C \exp\left( - c (\log n)^{c \log \log n} \right), $$
where $C,c,c' > 0$ depend only on $\tau_0$, $\tau_1$, and the sub-gaussian moments of $\xi$ and $\zeta$.  In view of Remark \ref{rem:perturb}, we have
\begin{equation} \label{eq:suplambnd}
	\sup_{\lambda \in \mathbb{R}} \dim(V_\lambda) \leq 4 \tau n + 2 (\log n)^{c' \log \log n} + k
\end{equation}
on the same event.  For the remainder of the proof, we fix a realization in which $\dim(H) \leq k$ and the bound in \eqref{eq:suplambnd} holds.  

Observe that the collection $\{V_{\lambda}\}_{\lambda \in \mathbb{R}}$ contains at most $n^2$ distinct subspaces.  This follows since each $V_{\lambda}$ has the form
$$ V_{\lambda} = \Span \left\{ v_{i}(W+J) : a_{\lambda} \leq \lambda_i(W+J) \leq b_{\lambda} \right\} $$
for some $a_\lambda, b_\lambda \in \mathbb{R}$.  Let $V_{\lambda_1}, \ldots, V_{\lambda_N}$ be the distinct subspaces, where $N \leq n^2$. 

From \eqref{eq:suplambnd}, we have 
$$ \sup_{1 \leq j \leq N} \dim(H_{\lambda_j}) \leq 4 \tau n + 2 (\log n)^{c' \log \log n} + 2k. $$
For each $1 \leq j \leq N$, let $\mathcal{N}_j$ be an $\eps$-net of the unit sphere in $H_{\lambda_j}$.  In particular, by Lemma \ref{lemma:net}, we can choose $\mathcal{N}_j$ such that
$$ |\mathcal{N}_j| \leq \left( 1+ \frac{2}{\eps} \right)^{4 \tau n + 2 (\log n)^{c' \log \log n} + 2k} . $$
Set $\mathcal{N} := \cup_{j=1}^N \mathcal{N}_j$.  Then, by construction,
$$ |\mathcal{N}| \leq N \left( 1+ \frac{2}{\eps} \right)^{4 \tau n + 2 (\log n)^{c' \log \log n} + 2k} \leq n^2 \left( 1+ \frac{2}{\eps} \right)^{4 \tau n + 2 (\log n)^{c' \log \log n} + 2k}. $$

It remains to show that $\mathcal{N}$ satisfies property \eqref{item:epsnet}.  Let $w$ be a unit vector in $H_{\lambda}$ for some $\lambda \in \mathbb{R}$.  Then $w \in H_{\lambda_j}$ for some $1 \leq j \leq N$.  Thus, there exists $w' \in \mathcal{N}_j \subseteq \mathcal{N}$ such that $|w - w'| \leq \eps$, and the proof is complete.   
\end{proof}

We now prove Theorem \ref{thm:diag}.

\begin{proof}[Proof of Theorem \ref{thm:diag}]
From the identity 
$$ \| v_j(W+J) \|_S^2 + \|v_j(W+J) \|_{S^c}^2 = 1, $$
it suffices to prove the lower bound for $\|v_j(W+J)\|_S$.  Let $m := \lfloor \delta n \rfloor$ and $r := n -m$.  By symmetry, it suffices to consider the case when $S := \{1, \ldots, m\}$.  

We decompose $W+J$ as 
$$ M := W + J = \begin{pmatrix} A & B \\ B^\mathrm{T} & D \end{pmatrix} + \begin{pmatrix} J_A & 0 \\ 0 & J_D \end{pmatrix}, $$
where $A$ and $J_A$ are $m \times m$ matrices, $D$ and $J_D$ are $r \times r$ matrices, and $B$ is a $m \times r$ matrix.  In particular, $A$ is an $m$-dimensional Wigner matrix, $D$ is an $r$-dimensional Wigner matrix, and $B^\mathrm{T}$ is an $r \times m$ matrix whose entries are iid copies of $\xi$.  In addition, $J_A$ and $J_D$ are diagonal matrices both having rank at most $k$.  

Fix $1 \leq j \leq n$ such that $\lambda_j(M) \in [\lambda_1(W), \lambda_n(W)]$, and write 
$$ v_j(M) = \begin{pmatrix} x_j \\ y_j \end{pmatrix}, $$
where $x_j$ is an $m$-vector and $y_j$ is an $r$-vector.  From the eigenvalue equation $M v_j(M) = \lambda_j(M) v_j(M)$, we obtain 
\begin{align}
	A x_j + B y_j + J_A x_j &= \lambda_j(M) x_j \label{eq:diageig1}\\
	B^\mathrm{T} x_j + D y_j + J_D y_j &= \lambda_j(M) y_j. \label{eq:diageig2}
\end{align}

In order to reach a contradiction, assume $\|x_j\| \leq \tau \eps$ for some constants $0 < \tau, \eps < 1$ (depending only on $\delta$, $k$, and the sub-gaussian moments of $\xi$ and $\zeta$) to be chosen later.  Define the event
$$ \Omega_{n,1} := \{ \|W\| \leq C_0 \sqrt{n} \}. $$
In particular, on the event $\Omega_{n,1}$, $|\lambda_j(M)| \leq C_0 \sqrt{n}$.  

From \eqref{eq:diageig2}, we find
$$ (D + J_D - \lambda_j(M)I) y_j = - B^\mathrm{T} x_j. $$
Thus, on the event $\Omega_{n,1}$,  Lemma \ref{lemma:structure} implies that $y_j = v_j + q_j$, where $v_j$, $q_j$ are orthogonal, $\|q_j\| \leq \eps$, and 
$$ v_j \in V_j := \Span\{ v_i (D + J_D) : |\lambda_i(D + J_D) - \lambda_j(M)| \leq C_0 \tau \sqrt{n} \}. $$

Let $0 < c_0 < 1$ be a constant (depending only on $\delta$, $k$, and the sub-gaussian moments of $\xi$ and $\zeta$) to be chosen later.  By Lemma \ref{lemma:largenet}, there exists a subset $\mathcal{N}$ of the unit sphere in $\mathbb{R}^r$ such that
\begin{enumerate}[(i)]
\item \label{item:everyunit} for every unit vector $v \in V_j$ there exists $v' \in \mathcal{N}$ such that $\|v - v'\| \leq c_0$,
\item \label{item:netsize} we have
$$ |\mathcal{N}| \leq n^2 \left( 1 + \frac{2}{c_0} \right)^{12 \tau (\sqrt{n} / \sqrt{r}) n + 2 (\log n)^{c' \log \log n} + 2k}, $$
\item $\mathcal{N}$ depends only on $D$, $J_D$, and $c_0$.  
\end{enumerate}
We now condition on the sub-matrix $D$ such that properties \eqref{item:everyunit} and \eqref{item:netsize} hold.  By independence, this conditioning does not effect the matrices $A$ and $B$.  Since $\mathcal{N}$ only depends on $D$, $J_D$, and $c_0$, we now treat $\mathcal{N}$ as a deterministic set.  

Let $H := \range(J_A)^{\perp}$, and let $P_H$ denote the orthogonal projection onto $H$.  Then, from \eqref{eq:diageig1}, we have
\begin{align*}
	2C_0 \sqrt{n} \|x_j\| \geq \| (A - \lambda_j(M)) x_j\| = \|B y_j + J_A x_j \| &\geq \|P_H B y_j\| \\
		&\geq \| P_H B v_j \| - C_0 \eps \sqrt{n} \\
		&\geq \|v_j\| \inf_{\substack{v \in V_j \\ \|v\|=1}} \|P_H B v \| - C_0 \eps \sqrt{n} 
\end{align*}
on the event $\Omega_{n,1}$.  As $\|q_j\|^2 + \|v_j\|^2 = \|y_j\|^2 = 1 - \|x_j\|^2$, we have
$$ \|v_j \|^2 \geq 1 - \eps^2 \tau^2 - \eps^2, $$
and hence, we conclude that
\begin{equation} \label{eq:epstaubnd}
	2C_0 \sqrt{n} \|x_j\| \geq \sqrt{1 - \eps^2 \tau^2 - \eps^2} \inf_{\substack{v \in V_j \\ \|v\|=1}} \|P_H B v \| - C_0 \eps \sqrt{n}.  
\end{equation}

We now obtain a lower bound for 
$$ \inf_{\substack{v \in V_j \\ \|v\|=1}} \|P_H B v \|. $$
Indeed, Since $J_A$ has rank at most $k$, $\dim(H) \geq n-k$.  Thus, by taking $c_0$ sufficiently small, we find that
\begin{align*}
	\Prob &\left( \inf_{\substack{v \in V_j \\ \|v\|=1}} \|P_H B v \| \leq \frac{1}{100} \sqrt{m-k} \right) \\
	&\qquad\leq \Prob \left( \inf_{v \in \mathcal{N}} \|P_H B v\| \leq \frac{1}{100} \sqrt{m-k} + c_0 \|B\| \right) \\
	&\qquad\leq \Prob \left( \left\{ \inf_{v \in \mathcal{N}} \|P_H B v\| \leq \frac{1}{100} \sqrt{m-k} + c_0 \|B\| \right\} \cap \Omega_{n,1} \right) + \Prob \left( \Omega_{n,1}^c \right) \\
	&\qquad\leq \Prob \left( \inf_{v \in \mathcal{N}} \|P_H B v\| \leq \frac{1}{100} \sqrt{m-k} + C_0 c_0 \sqrt{n} \right) + \Prob \left( \Omega_{n,1}^c \right) \\
	&\qquad\leq \Prob \left( \inf_{v \in \mathcal{N}} \|P_H B v\| \leq \frac{1}{2} \sqrt{m-k} \right) + \Prob \left( \Omega_{n,1}^c \right) \\
	&\qquad\leq \sum_{v \in \mathcal{N}} \Prob \left( \| P_H B v \| \leq \frac{1}{2} \sqrt{m-k} \right) + \Prob \left( \Omega_{n,1}^c \right).
\end{align*}
Therefore, in view of Lemma \ref{lemma:qua}, we obtain
\begin{align*}
	\Prob &\left( \inf_{\substack{v \in V_j \\ \|v\|=1}} \|P_H B v \| \leq \frac{1}{100} \sqrt{m-k} \right) \\
	&\leq 2 n^2 \left( 1 + \frac{2}{c_0} \right)^{12 \tau (\sqrt{n}/\sqrt{r}) n + 2 (\log n)^{c' \log \log n} + 2k} \exp(-c (m-k)) + \Prob \left( \Omega_{n,1}^c \right), 
\end{align*}
where $c > 0$ depends only on the sub-gaussian moment of $\xi$.  Therefore, by taking $\tau$ sufficiently small and applying Lemma \ref{lemma:norm}, we have
$$ \inf_{\substack{v \in V_j \\ \|v\|=1}} \|P_H B v \| \geq \frac{1}{100} \sqrt{m-k} $$
with probability at least 
$$ 1 - C_2 \exp \left( -c_2 (\log n)^{c_2 \log \log n} \right), $$
where $C_2, c_2 > 0$ depend only on $\delta$, $k$, and the sub-gaussian moments of $\xi$ and $\zeta$.  

On this event, \eqref{eq:epstaubnd} implies that
$$ \eps \tau \geq \|x_j\| \geq \frac{ \sqrt{1 - \eps^2 \tau^2 - \eps^2}}{200 C_0} \sqrt{ \frac{m-k}{n} } - \frac{\eps}{2}, $$
a contradiction for $\eps$ sufficiently small.  Therefore, on the same event, we conclude that $\|x_j\| \geq \eps \tau$.  
\end{proof}

\subsection{Proof of Theorem \ref{thm:perturb}}

Unsurprisingly, the proof of Theorem \ref{thm:perturb} is very similar to the proof of Theorem \ref{thm:diag}.  

\begin{proof}[Proof of Theorem \ref{thm:perturb}]
Let $m := \lfloor \delta n \rfloor$ and $r := n -m$.  By symmetry, it suffices to consider the case when $S := \{1, \ldots, m\}$.  

We decompose $W+J$ as 
$$ M := W + J = \begin{pmatrix} A & B \\ B^\mathrm{T} & D \end{pmatrix} + \begin{pmatrix} J_A & J_B \\ J_B^\mathrm{T} & J_D \end{pmatrix}, $$
where $A$ and $J_A$ are $m \times m$ matrices, $B$ and $J_B$ are $m \times r$ matrices, and $D$ and $J_D$ are $r \times r$ matrices.  In particular, $A$ is an $m$-dimensional Wigner matrix, $D$ is an $r$-dimensional Wigner matrix, and $B^\mathrm{T}$ is an $r \times m$ matrix whose entries are iid copies of $\xi$.  

Since the rank of any sub-matrix is not more than the rank of the original matrix (see, for example, \cite[Section 0.4.5]{HJ}), it follows that $J_A$, $J_B$, and $J_D$ all have rank at most $k$.  

Fix $\eps_1 n \leq j \leq (1 - \eps_1) n$.  Then, by \cite[Theorem III.2.1]{Bhatia}, for $n$ sufficiently large (in terms of $k$ and $\eps_1$), 
\begin{equation} \label{eq:eiginter}
	\lambda_{j+k}(W) \leq \lambda_j(M) \leq \lambda_{j-k}(W). 
\end{equation}
Express $v_j(M)$ as 
$$ v_j(M) = \begin{pmatrix} x_j \\ y_j \end{pmatrix}, $$
where $x_j$ is an $m$-vector and $y_j$ is an $r$-vector.  From the eigenvalue equation $M v_j(M) = \lambda_j(M) v_j(M)$, we obtain
\begin{align}
	A x_j + B y_j + J_A x_j + J_B y_j &= \lambda_j(M) x_j \label{eq:eigen1} \\
	B^\mathrm{T} x_j + D y_j + J_B^\mathrm{T} x_j + J_D y_j &= \lambda_j(M) y_j. \label{eq:eigen2}
\end{align}

In order to reach a contradiction, assume $\frac{1}{n^{1 - \eps_0}} \leq \|x_j\| \leq \tau \eps$ for some constants $0 < \tau, \eps < 1$ (depending only on $\delta$, $k$, $\eps_0$, $\eps_1$, and the sub-gaussian moments of $\xi$ and $\zeta$) to be chosen later.  Define the event
$$ \Omega_{n,1} := \{ \|W\| \leq C_0 \sqrt{n} \}. $$
In particular, in view of \eqref{eq:eiginter}, on the event $\Omega_{n,1}$, $|\lambda_j(M)| \leq C_0 \sqrt{n}$.   

Let $\gamma_j$ denote the classical location of the $j$th eigenvalue of a Wigner matrix.  That is, $\gamma_j$ is defined by
$$ n \int_{-\infty}^{\gamma_j} \rho_{\mathrm{sc}}(x)dx = j, $$
where $\rho_{\mathrm{sc}}$ is defined in \eqref{eq:def:rho}.  Let $\alpha > 0$ be a small parameter (depending on $\eps_0$) to be chosen later.  From \eqref{eq:eiginter} and \cite[Theorem 3.6]{LY}, we conclude that the event
$$ \Omega_{n,2} := \left\{ \left| \lambda_j(M) - \sqrt{n} \gamma_j \right| \leq \frac{n^{\alpha}}{n^{1/2}} \right\} $$
holds with probability at least $1 - C \exp \left( - c (\log n)^{c \log \log n} \right)$, where $C, c > 0$ depend on $\alpha$, $k$, and the sub-gaussian moments of $\xi$ and $\zeta$.  

From \eqref{eq:eigen2}, we obtain
$$ (D + J_D - \lambda_j(M) I ) y_j = - (B^\mathrm{T} + J_B^\mathrm{T}) x_j, $$
which we rewrite as
$$ (D + J_D - \sqrt{n} \gamma_j I ) y_j = - ( B^\mathrm{T} + (\sqrt{n} \gamma_j - \lambda_j(M)) Q_j + J_B^\mathrm{T}) x_j, $$
where $Q_J$ is the rank one matrix given by
$$ Q_j := \frac{ y_j x_j^\mathrm{T} }{\|x_j\|^2}. $$
In particular, $Q_j x_j = y_j$.  Since $\|x_j\| \geq \frac{1}{n^{1 - \eps_0}}$, we have
\begin{align*}
	\| (\sqrt{n} \gamma_j - \lambda_j(M)) Q_j \| \leq \frac{n^{\alpha}}{n^{1/2}} \frac{1}{\|x_j\|} \leq n^{1/2 - \eps_0/2} 
\end{align*}
on the event $\Omega_{n,2}$ by taking $\alpha$ sufficient small.  Thus, on the event $\Omega_{n,1} \cap \Omega_{n,2}$, we have 
$$ \| B^\mathrm{T} + (\sqrt{n} \gamma_j - \lambda_j(M)) Q_j \| \leq C_0 \sqrt{n} + n^{1/2 - \eps_0/2}. $$
Thus, on the same event, we apply Lemma \ref{lemma:structure2} and obtain the following.
\begin{enumerate}[(i)]
\item There exists $\eta > 0$ such that $D + J_D - (\sqrt{n} \gamma_j + \eta) I$ is invertible.
\item $y_j = v_j + q_j$, where $\|q_j\| \leq \eps$, $v_j \in \Span\{V_j \cup H_j\}$, 
$$ V_j := \Span \left\{ v_i(D + J_D) : | \lambda_i(D + J_D) - \sqrt{n} \gamma_j| \leq C_0 \tau \sqrt{n} + \tau n^{1/2 - \eps_0/2} \right \} $$
and
$$ H_j := \range \left( (D+ J_D - (\sqrt{n} \gamma_j + \eta) I )^{-1} J_B^\mathrm{T} \right). $$
\end{enumerate}
In particular, $V_j$ and $H_j$ depend only on $D$, and $H_j$ has dimension at most $k$. 

We pause a moment to note the following.  We introduced the classical location $\gamma_j$ so that the subspace $H_j$ depends only on $D$ and not on the entire matrix $W$.  If we had not introduced $\gamma_j$, then $H_j$ would depend on $\lambda_j(M)$ (and hence on $W$).  Since $H_j$ only depends on $D$, Lemma \ref{lemma:largenet} is applicable, and the net $\mathcal{N}$, introduced below, also depends only on $D$.  

Let $c_0 > 0$ be a constant (depending on $\delta$, $k$, and the sub-gaussian moments of $\xi$ and $\zeta$) to be chosen later.  Then, by Lemma \ref{lemma:largenet}, there exists a subset $\mathcal{N}$ of the unit ball in $\mathbb{R}^r$ such that
\begin{enumerate}[(i)]
\item for every unit vector $v \in \Span\{V_j \cup H_j\}$, there exists $v' \in \mathcal{N}$ such that $\|v - v'\| \leq c_0$,
\item we have
$$ |\mathcal{N}| \leq n^2 \left( 1 + \frac{2}{c_0} \right)^{12 \tau \frac{\sqrt{n}}{\sqrt{r}}n + 4 \tau \frac{ \sqrt{n}}{\sqrt{r} n^{\eps_0/2}}n + 2 (\log n)^{c' \log \log n} + 2k}, $$
\item $\mathcal{N}$ depends only on $D$, $J_D$, $J_B^\mathrm{T}$, and $c_0$.  
\end{enumerate}
We now condition on the sub-matrix $D$ such that the first two properties hold.  By independence, this conditioning does not effect the matrices $A$ and $B$.  We now treat $\mathcal{N}$ as a deterministic set. 

Returning to \eqref{eq:eigen1}, on the event $\Omega_{n,1}$, we have
\begin{align*}
	2 C_0 \sqrt{n} \|x_j \| \geq \| (A - \lambda_j(M) I ) x_j \| &= \|B y_j + J_B y_j + J_A x_j \| \\
		&= \|B y_j + J' v_j(M) \| \\
		& \geq \| P_H B y_j \| \\
		& \geq \| P_H B v_j \| - C_0 \eps \sqrt{n},
\end{align*}
where $J' = \begin{pmatrix} J_A & J_B \end{pmatrix}$, $H = (\range J')^\perp$, and $P_H$ is the orthogonal projection onto $H$.  (Note that $H$ and $H_j$ are different subspaces.)  Thus, we conclude that
\begin{equation} \label{eq:6nxjbnd}
	2 C_0 \sqrt{n} \|x_j \| \geq \left( \sqrt{1 - \eps^2 \tau^2} - \eps \right) \inf_{\substack{v \in \Span\{V_j \cup H_j\} \\ \|v\| = 1}} \|P_H B v \| - C_0 \eps \sqrt{n}. 
\end{equation}

We now obtain a lower bound for 
$$ \inf_{\substack{v \in \Span\{V_j \cup H_j\} \\ \|v\| = 1}} \|P_H B v \| $$
as in the proof of Theorem \ref{thm:diag}.  Indeed, for $c_0$ sufficiently small, we find 
\begin{align*}
	\Prob &\left( \inf_{\substack{v \in \Span\{V_j \cup H_j\} \\ \|v\| = 1}} \|P_H B v \| \leq \frac{1}{100} \sqrt{m-k} \right) \\
		&\leq 2 n^2 \left( 1 + \frac{2}{c_0} \right)^{12 \tau \frac{\sqrt{n}}{\sqrt{r}}n + 4 \tau \frac{ \sqrt{n}}{\sqrt{r} n^{\eps_0/2}}n + 2 (\log n)^{c' \log \log n} + 2k} \exp(-c(m-k)) + \Prob\left(\Omega_{n,1}^c \right).
\end{align*}
Therefore, from \eqref{eq:6nxjbnd} and by taking $\tau$ sufficiently small, we find that
$$ \eps \tau \geq \|x_j \| \geq \frac{ \sqrt{1 - \eps^2 \tau^2} - \eps }{ 200 C_0 } \sqrt{ \frac{m-k}{n} } - \frac{\eps}{2} $$
with probability at least $1 - C_2 \exp \left( -c_2 (\log n)^{c_2 \log \log n} \right)$.  This is a contradiction for $\eps$ sufficiently small.  Thus, we conclude that, on the same event, either $\|x_j\| \geq \eps \tau$ or $\|x_j\| \leq \frac{1}{n^{1 - \eps_0}}$.  
\end{proof}

\subsection{Proof of Theorem \ref{thm:singlerank1}}

In order to prove Theorem \ref{thm:singlerank1}, we will apply Theorem \ref{thm:BY}.  

Without loss of generality, we assume $\theta > 0$.  If $\theta < 0$, one can consider $-W - J$ instead of $W + J$.  Indeed, $-W$ is a Wigner matrix with atom variable $-\xi, -\zeta$.  In addition, $-W - J$ and $W+J$ have the same eigenvectors while the eigenvalues only differ by sign.  

Recall that $M := W + J$.  Beginning with the eigenvalue equation
$$ M v_j(M) = \lambda_j(M) v_j(M), $$
we multiply on the left by $v_j(W)^\mathrm{T}$ to obtain
$$ \lambda_j(W) v_j(W)^\mathrm{T} v_j(M) + v_j(W)^\mathrm{T} J v_j(M) = \lambda_j(M) v_j(W)^\mathrm{T} v_j(M). $$
Thus, we find that
$$ |v_j(W)^\mathrm{T} J v_j(M) | \leq |\lambda_j(M) - \lambda_j(W)|. $$
As $J$ has rank $1$, eigenvalue interlacing (see, for instance, \cite[Exercise III.2.4]{Bhatia}) implies that
$$ |v_j(W)^\mathrm{T} J v_j(M) | \leq \lambda_{j+1}(W) - \lambda_j(W). $$
Since $J = \theta u u^\mathrm{T}$, we conclude that
$$ |v_j(W) \cdot u| |v_j(M) \cdot u| \leq \frac{\lambda_{j+1}(W) - \lambda_j(W)}{\theta}. $$

We will bound $\lambda_{j+1}(W) - \lambda_j(W)$ above using the rigidity of eigenvalues estimate from \cite{EYY}.  Indeed, by \cite[Theorem 2.2]{EYY}, it follows that
$$ \lambda_{j+1}(W) - \lambda_j(W) \leq \frac{C (\log n)^{c \log \log n}}{\sqrt{n}} $$
with probability $1 - o(1)$ for some constants $C, c > 0$.  Thus, we obtain that
$$ \sqrt{n} |v_j(W) \cdot u| |v_j(M) \cdot u| \leq \frac{C (\log n)^{c \log \log n}}{\theta} $$
with probability $1 - o(1)$.  

From Theorem \ref{thm:BY}, we have
$$ \sqrt{n} |v_j(W) \cdot u| \geq \frac{1}{\log n} $$
with probability $1 - o(1)$.  Thus, we conclude that
$$ |v_j(M) \cdot u| \leq \frac{C (\log n)^{(c+1) \log \log n}}{\theta} $$
with probability $1 - o(1)$, and the proof is complete.

\appendix

\section{Proof of Lemma \ref{lem:beta}}\label{appendix:beta}

For each $n \geq 1$, let $V_n$ and $W_n$ be independent $\chi^2$-distribution random variables with $\alpha_n$ and $\beta_n$ degrees of freedom respectively.  Then
$$ \frac{V_n}{V_n + W_n} \sim \mathrm{Beta} \left( \frac{\alpha_n}{2}, \frac{\beta_n}{2} \right). $$
This can be verified by computing the distribution of the ratio directly; see \cite{S} for details.  

Thus, in order to prove Lemma \ref{lem:beta}, it suffices to show that
\begin{equation} \label{eq:beta:show}
	\sqrt{ \frac{(\alpha_n + \beta_n)^3}{2 \alpha_n \beta_n} } \left( \frac{V_n}{V_n + W_n} - \frac{\alpha_n}{\alpha_n + \beta_n} \right) \longrightarrow N(0,1) 
\end{equation}
in distribution as $n \to \infty$.  

We decompose the left-hand side of \eqref{eq:beta:show} as
\begin{align*}
	\sqrt{ \frac{(\alpha_n + \beta_n)^3}{2 \alpha_n \beta_n} }& \left( \frac{V_n}{V_n + W_n} - \frac{\alpha_n}{\alpha_n + \beta_n} \right) \\
	&= \sqrt{ \frac{\alpha_n + \beta_n}{2 \alpha_n \beta_n} } \left( \frac{ \beta_n V_n - \alpha_n W_n}{\alpha_n + \beta_n} \right) \frac{\alpha_n + \beta_n}{V_n + W_n} \\
	&= \sqrt{ \frac{\alpha_n + \beta_n}{2 \alpha_n \beta_n} } \left( \frac{ \beta_n \sqrt{\alpha_n} \left( \frac{V_n - \alpha_n}{\sqrt{\alpha_n}} \right) - \alpha_n \sqrt{\beta_n} \left( \frac{W_n - \beta_n}{\sqrt{\beta_n}} \right) }{ \alpha_n + \beta_n} \right) \frac{\alpha_n + \beta_n}{V_n + W_n} \\
	&= \left[ \sqrt{ \frac{\beta_n}{\alpha_n + \beta_n}} \left( \frac{V_n - \alpha_n}{\sqrt{2 \alpha_n}} \right) - \sqrt{\frac{\alpha_n}{\alpha_n + \beta_n}} \left( \frac{W_n - \beta_n}{\sqrt{2 \beta_n}} \right) \right]  \frac{\alpha_n + \beta_n}{V_n + W_n}.
\end{align*}

By definition of the $\chi^2$-distribution, $V_n$ has the same distribution as the sum of $\alpha_n$ independent squared standard normal random variables.  Similarly, $W_n$ has the same distribution as the sum of $\beta_n$ independent squared standard normal random variables.  Thus, by the central limit theorem, 
$$ \frac{V_n - \alpha_n}{\sqrt{2 \alpha_n}} \longrightarrow N(0,1) \quad \text{and} \quad \frac{W_n - \beta_n}{\sqrt{2 \beta_n}} \longrightarrow N(0,1) $$
in distribution as $n \to \infty$.  Moreover, since $V_n$ and $W_n$ are independent, we conclude that 
$$ \sqrt{ \frac{\beta_n}{\alpha_n + \beta_n}} \left( \frac{V_n - \alpha_n}{\sqrt{2 \alpha_n}} \right) - \sqrt{\frac{\alpha_n}{\alpha_n + \beta_n}} \left( \frac{W_n - \beta_n}{\sqrt{2 \beta_n}} \right) \longrightarrow N(0,1) $$
in distribution as $n \to \infty$.  Here we used Slutsky's theorem (see Theorem 11.4 in \cite[Chapter 5]{G}) since both 
$$ \sqrt{ \frac{\beta_n}{\alpha_n + \beta_n}} \quad \text{and} \quad \sqrt{\frac{\alpha_n}{\alpha_n + \beta_n}} $$
converge to limits in $[0,1]$ by supposition.  

Finally, by the law of large numbers, we observe that
$$ \frac{\alpha_n + \beta_n}{V_n + W_n} \longrightarrow 1 $$
almost surely as $n \to \infty$.  Thus, by another application of Slutsky's theorem, we conclude that
$$ \left[ \sqrt{ \frac{\beta_n}{\alpha_n + \beta_n}} \left( \frac{V_n - \alpha_n}{\sqrt{2 \alpha_n}} \right) - \sqrt{\frac{\alpha_n}{\alpha_n + \beta_n}} \left( \frac{W_n - \beta_n}{\sqrt{2 \beta_n}} \right) \right]  \frac{\alpha_n + \beta_n}{V_n + W_n} \longrightarrow N(0,1) $$
in distribution as $n \to \infty$, and the proof is complete.

\section{Proof of Lemmas \ref{lemma:structure} and \ref{lemma:structure2}} \label{sec:structure}

We now present the proof of Lemmas \ref{lemma:structure} and \ref{lemma:structure2}.

\begin{proof}[Proof of Lemma \ref{lemma:structure}]
By the spectral theorem, the eigenvectors $v_1(B), \ldots, v_r(B)$ of $B$ form an orthonormal basis.  Write $y= \sum_{i=1}^r \alpha_i v_i(B)$.  Define
$$ q := \sum_{i : |\lambda_i(B)| > \tau \|A\|} \alpha_i v_i(B), $$
and set $v := y-q$.  Clearly, $v$ and $q$ are orthogonal.  Moreover, 
$$ v \in \Span\{ v_i(B) : |\lambda_i(B)| \leq \tau \|A\| \} $$
by construction.  It remains to show $\|q\| \leq \eps$.  

We now utilize the equation $By = Ax$.  We first consider the case when $\|A\| = 0$.  In this case, $By = 0$; in other words, $y$ is in the null space of $B$.  Thus, the vector $q$ is zero, and the claim follows.  

Assume $\|A \| > 0$.  Since $\|Ax \| \leq \|A \| \|x\| \leq \eps \tau \|A\|$, we have $\|B y\| \leq \eps \tau \|A\|$.  Thus, by the spectral theorem, we obtain
$$ \|B y\|^2 = \sum_{i=1}^r \lambda_i^2(B) |\alpha_i|^2 \leq \eps^2 \tau^2 \|A\|^2. $$
By definition of $q$, this implies
$$ \tau^2 \|A\|^2 \|q\|^2 = \tau^2 \|A\|^2 \sum_{i : |\lambda_i(B)| > \tau \|A\|} |\alpha_i|^2 \leq \sum_{i : |\lambda_i(B)| > \tau \|A\|} \lambda_i^2(B) |\alpha_i|^2 \leq \eps^2 \tau^2 \|A\|^2. $$
Therefore, we conclude that 
$$ \|q\|^2 \leq \frac{\eps^2 \tau^2 \|A\|^2}{\tau^2 \|A\|^2} = \eps^2, $$
and the proof is complete.  
\end{proof}

\begin{proof}[Proof of Lemma \ref{lemma:structure2}]
Let $\{\eta_k\}_{k =1}^\infty$ be a non-increasing sequence of non-negative real numbers such that $\lim_{k \to \infty} \eta_k = 0$ and $B-\eta_k I$ is invertible for all $k \geq 1$.  It is always possible to find such a sequence since $B$ has at most $r$ distinct eigenvalues.  

Without loss of generality, we assume $0 < \eps < 1$.  Indeed, if $\eps \geq 1$, part \eqref{item:span} of the lemma follows by taking $q := y$ and $v := 0$.  

Define the value
\begin{equation} \label{eq:defL}
	L:= \frac{\kappa \eps \tau + \eta_k}{\eps} 
\end{equation}
and the subspace
$$ W := \Span \{ v_i(B) : |\lambda_i(B)| \leq \tau \kappa\}. $$

We make two simple observations related to the subspace $W$.
\begin{itemize}
\item If $|\lambda_i(B)| \leq \tau \kappa$, then, by the triangle inequality, 
$$ |\lambda_i(B) - \eta_k| \leq \tau \kappa + \eta_k \leq L. $$
Here the last inequality follows from the fact that $0 < \eps < 1$.  
\item If $|\lambda_i(B) - \eta_k| \leq L$, then
$$ |\lambda_i(B)| \leq L + \eta_k = \kappa \tau + \eta_k \left( 1 + \frac{1}{\eps} \right). $$
\end{itemize}
From the two observations above, we conclude that, for all $k$ sufficiently large (i.e. $\eta_k$ sufficiently small),
$$ \{ 1 \leq i \leq r : |\lambda_i(B)| \leq \tau \kappa \} = \{1 \leq i \leq r : |\lambda_i(B) - \eta_k| \leq L \}. $$
Fix $k$ sufficiently large so that the above equality holds.  Then 
\begin{equation} \label{eq:charW}
	W = \Span \{ v_i(B) : |\lambda_i(B) - \eta_k| \leq L \}. 
\end{equation}
Let 
$$ V := \Span \left\{ W \cup \range\left(( B-\eta_k I)^{-1} J \right) \right\}. $$
The bound in \eqref{eq:dimbnd} follows from the fact that the rank of $( B - \eta_k I)^{-1} J$ is no larger than the rank of $J$ (see \cite[Section 0.4]{HJ}).    

Since $B$ is Hermitian, we have
$$ \lambda_i(B - \eta_k I) = \lambda_i(B) - \eta_k, \quad v_i(B - \eta_k I) = v_i(B) $$
for all $1 \leq i \leq r$.  In particular, by the spectral theorem, we have
\begin{equation} \label{eq:sdBinv}
	(B - \eta_k I)^{-1} = \sum_{i=1}^r \frac{1}{\lambda_i(B) - \eta_k} v_i(B) v_i(B)^\ast. 
\end{equation} 

We rewrite $By = (A+J)x$ as
$$ (B- \eta_k I) y = Ax - \eta_k y + Jx, $$
and hence obtain
$$ y = (B- \eta_k I)^{-1}(Ax - \eta_k y) + (B- \eta_k I)^{-1} Jx. $$
Decompose
$$ Ax - \eta_k y = \sum_{i=1}^r \alpha_i v_i(B). $$
Then, by \eqref{eq:sdBinv}, we have
\begin{align*}
	(B-\eta_k I)^{-1} (Ax - \eta_k y) = \sum_{i=1}^r \frac{1}{\lambda_i(B) - \eta_k} \alpha_i v_i(B) = q + w,
\end{align*}
where
$$ q := \sum_{i : |\lambda_i(B) - \eta_k| > L} \frac{1}{\lambda_i(B) - \eta_k} \alpha_i v_i(B) $$
and 
$$ w := \sum_{i : |\lambda_i(B) - \eta_k| \leq L} \frac{1}{\lambda_i(B) - \eta_k} \alpha_i v_i(B). $$
Set $v := w + (B- \eta_k I)^{-1} Jx$.  From \eqref{eq:charW}, we find that $w \in W$, and hence $v \in V$.  It remains to show $\|q\| \leq \eps$.  

By definition of $q$, we obtain
$$ \|q\|^2 = \sum_{i : |\lambda_i(B) - \eta_k| > L} \frac{1}{ |\lambda_i(B) - \eta_k|^2} |\alpha_i|^2 \leq \frac{1}{L^2} \|A x - \eta_k y\|^2. $$
By supposition, we have $\|Ax - \eta_k  y\| \leq \kappa \eps \tau + \eta_k$.  Thus, in view of \eqref{eq:defL}, we conclude that 
$$ \|q\|^2 \leq \frac{(\kappa \eps \tau + \eta_k)^2}{L^2} = \eps^2, $$
and the proof of the lemma is complete.  
\end{proof}

\section*{Acknowledgments}
The first author thanks Professor Roman Vershynin for clarifications.  The third author would like to thank Professor Tiefeng Jiang for many useful discussions and constant encouragement.  The authors also thank the anonymous referee for valuable comments.


\begin{thebibliography}{10}

\bibitem{AGZ}
G.~W. Anderson, A.~Guionnet, and O.~Zeitouni.
\newblock {\em An introduction to random matrices}, volume 118 of {\em
  Cambridge Studies in Advanced Mathematics}.
\newblock Cambridge University Press, Cambridge, 2010.

\bibitem{AB}
S.~Arora and A.~Bhaskara.
\newblock Eigenvectors of random graphs: delocalization and nodal domains.
\newblock Preprint, available at {\verb+http://www.cs.princeton.edu/~bhaskara/files/deloc.pdf+}.

\bibitem{ABAP}
A.~Auffinger, G.~Ben Arous, and S. P\'{e}ch\'{e}.
\newblock Poisson convergence for the largest eigenvalues of heavy tailed random matrices.
\newblock {\em Ann. Inst. H. Poincar\'{e} Probab. Statist.}, Volume 45, Number 3 (2009), 589--610.

\bibitem{BS}
Z.~Bai and J.~W. Silverstein.
\newblock {\em Spectral analysis of large dimensional random matrices}.
\newblock Springer Series in Statistics. Springer, New York, second edition,
  2010.
  
\bibitem{BBP}
J.~Baik, G.~Ben Arous, and S. P\'{e}ch\'{e}.
\newblock Phase transition of the largest eigenvalue for nonnull complex sample covariance matrices.
\newblock {\em  Ann. Probab.}, Volume 33, Number 5 (2005), 1643--1697.

\bibitem{BKY}
R.~Bauerschmidt, A.~Knowles, and H.-T. Yau.
\newblock Local semicircle law for random regular graphs.
\newblock {\em Preprint, \href{http://arxiv.org/abs/1503.08702}{arXiv:1503.08702}}, 2015.

\bibitem{BG}
F.~Benaych-Georges and A.~Guionnet.
\newblock Central limit theorem for eigenvectors of heavy tailed matrices. 
\newblock{\em Electron. J. Prob.}, Vol. 19 (2014), no. 54, 1--27.

\bibitem{BP} 
F.~Benaych-Georges and S.~P\'{e}ch\'{e}.
\newblock Localization and delocalization for heavy tailed band matrices.
\newblock {\em Ann. Inst. H. Poincar\'{e} Probab. Statist.}, Volume 50, Number 4 (2014), 1385--1403.

\bibitem{BR}
F.~Benaych-Georges and R. Rao Nadakuditi.
\newblock Eigenvalues and eigenvectors of finite, low rank perturbation of large random matrices.
\newblock{ \em Advances in Mathematics}, 227(1):494--521, 2009.

\bibitem{Bhatia}
R.~Bhatia.
\newblock {\em Matrix analysis}, volume 169 of {\em Graduate Texts in
  Mathematics}.
\newblock Springer-Verlag, New York, 1997.

\bibitem{BEKHY}
A.~Bloemendal, L.~Erdos, A.~Knowles, H.-T. Yau, and J.~Yin.
\newblock Isotropic local laws for sample covariance and generalized wigner
  matrices.
\newblock {\em Electron. J. Probab}, 19(33):1--53, 2014.

\bibitem{BorG}
C.~Bordenave and A.~Guionnet. 
\newblock Localization and delocalization of eigenvectors for heavy-tailed random matrices. 
\newblock {\em Probability Theory and Related Fields}, Vol. 157 (2013), no. 3, 885--953.

\bibitem{BorG2}
C.~Bordenave and A.~Guionnet.
\newblock Delocalization at small energy for heavy-tailed random matrices.
\newblock {\em Preprint, \href{http://arxiv.org/abs/1603.08845}{arXiv:1603.08845}}, 2016.  

\bibitem{BY}
P.~Bourgade and H.-T. Yau.
\newblock The eigenvector moment flow and local quantum unique ergodicty.
\newblock {\em Preprint, \href{http://arxiv.org/abs/1312.1301}{arXiv:1312.1301}}, 2013.

\bibitem{BL}
S.~Brooks and E.~Lindenstrauss.
\newblock Non-localization of eigenfunctions on large regular graphs.
\newblock {\em Israel Journal of Mathematics}, 193(1):1--14, 2013.

\bibitem{CAB}
C.~Cacciapuoti, A.~Maltsev, and B.~Schlein.
\newblock Local marchenko-pastur law at the hard edge of sample covariance
  matrices.
\newblock {\em Preprint, \href{http://arxiv.org/abs/1206.1730}{arXiv:1206.1730}}, 2012.

\bibitem{CB}
P.~Cizeau and J.~P.~Bouchaud. 
\newblock Theory of L\'{e}vy matrices.
\newblock {\em Phys. Rev. E} 50 (1994).  

\bibitem{CHM}
S.~Cs{{\"o}}rg{\H{o}}, E.~Haeusler, and D.~M. Mason.
\newblock The asymptotic distribution of extreme sums.
\newblock {\em Ann. Probab.}, 19(2):783--811, 1991.

\bibitem{deHaan}
L.~de~Haan and A.~Ferreira.
\newblock {\em Extreme value theory: An introduction}.
\newblock Springer Series in Operations Research and Financial Engineering.
  Springer, New York, 2006.

\bibitem{DLL}
Y.~Dekel, J.~R. Lee, and N.~Linial.
\newblock Eigenvectors of random graphs: nodal domains.
\newblock {\em Random Structures \& Algorithms}, 39(1):39--58, 2011.

\bibitem{DP}
I.~Dumitriu and S.~Pal.
\newblock Sparse regular random graphs: spectral density and eigenvectors.
\newblock {\em Annals of Probability}, 40(5):2197--2235, 2012.

\bibitem{E}
L.~Erd{\H{o}}s.
\newblock Universality of {W}igner random matrices: a survey of recent results.
\newblock {\em Uspekhi Mat. Nauk}, 66(3(399)):67--198, 2011.

\bibitem{EK} 
L.~Erd\"{o}s and A. Knowles. 
\newblock Quantum Diffusion and Eigenfunction Delocalization in a Random Band Matrix Model.
\newblock {\em Communications in Mathematical Physics}, 303 (2011), no. 2, 509--554.

\bibitem{EK2}
L.~Erd\"{o}s and A. Knowles.
\newblock Quantum Diffusion and Delocalization for Band Matrices with General Distribution
\newblock {\em Annales Henri Poincar\'{e}}, 12 (2011), no. 7, 1227--1319.  

\bibitem{EKY}
L.~Erd{\H{o}}s, A.~Knowles, and H.-T. Yau.
\newblock Averaging fluctuations in resolvents of random band matrices.
\newblock {\em Ann. Henri Poincar{\'e}}, 14(8):1837--1926, 2013.

\bibitem{EKYY}
L.~Erd{\H{o}}s, A.~Knowles, H.-T. Yau, and J.~Yin.
\newblock Spectral statistics of Erd{\H o}s-R{\'e}nyi graphs I: Local
  semicircle law.
\newblock {\em Ann. Probab.}, 41(3B):2279--2375, 2013.

\bibitem{EKYY2} 
L.~Erd{\H{o}}s, A.~Knowles, H.-T. Yau, and J.~Yin.
\newblock Delocalization and Diffusion Profile for Random Band Matrices. 
\newblock {\em Communications in Mathematical Physics}, 323 (2013), no. 1, 367--416.

\bibitem{EYTV}
L.~Erd{\H{o}}s, J.~Ram{\'{\i}}rez, B.~Schlein, T.~Tao, V.~Vu, and H.-T. Yau.
\newblock Bulk universality for {W}igner {H}ermitian matrices with
  subexponential decay.
\newblock {\em Math. Res. Lett.}, 17(4):667--674, 2010.

\bibitem{ESY2}
L.~Erd{\H{o}}s, B.~Schlein, and H.-T. Yau.
\newblock Local semicircle law and complete delocalization for {W}igner random
  matrices.
\newblock {\em Comm. Math. Phys.}, 287(2):641--655, 2009.

\bibitem{ESY1}
L.~Erd{\H{o}}s, B.~Schlein, and H.-T. Yau.
\newblock Semicircle law on short scales and delocalization of eigenvectors for
  {W}igner random matrices.
\newblock {\em Ann. Probab.}, 37(3):815--852, 2009.

\bibitem{ESY3}
L.~Erd{\H{o}}s, B.~Schlein, and H.-T. Yau.
\newblock Wegner estimate and level repulsion for {W}igner random matrices.
\newblock {\em Int. Math. Res. Not. IMRN}, (3):436--479, 2010.

\bibitem{ESY4}
L.~Erd{\H{o}}s, B.~Schlein, and H.-T. Yau.
\newblock Universality of random matrices and local relaxation flow.
\newblock {\em Invent. Math.}, 185(1):75--119, 2011.

\bibitem{ESYY}
L.~Erd{\H{o}}s, B.~Schlein, H.-T. Yau, and J.~Yin.
\newblock The local relaxation flow approach to universality of the local
  statistics for random matrices.
\newblock {\em Ann. Inst. Henri Poincar{\'e} Probab. Stat.}, 48(1):1--46, 2012.

\bibitem{EY}
L.~Erd{\H{o}}s and H.-T. Yau.
\newblock Universality of local spectral statistics of random matrices.
\newblock {\em Bull. Amer. Math. Soc. (N.S.)}, 49(3):377--414, 2012.

\bibitem{EYY2}
L.~Erd{\H{o}}s, H.-T. Yau, and J.~Yin.
\newblock Universality for generalized {W}igner matrices with {B}ernoulli
  distribution.
\newblock {\em J. Comb.}, 2(1):15--81, 2011.

\bibitem{EYY1}
L.~Erd{\H{o}}s, H.-T. Yau, and J.~Yin.
\newblock Bulk universality for generalized {W}igner matrices.
\newblock {\em Probab. Theory Related Fields}, 154(1-2):341--407, 2012.

\bibitem{EYY}
L.~Erd{\H{o}}s, H.-T. Yau, and J.~Yin.
\newblock Rigidity of eigenvalues of generalized {W}igner matrices.
\newblock {\em Adv. Math.}, 229(3):1435--1515, 2012.

\bibitem{Fe}
W.~Feller.
\newblock {\em An introduction to probability theory and its applications.
  {V}ol. {II}.}
\newblock Second edition. John Wiley \& Sons, Inc., New York-London-Sydney,
  1971.
  
\bibitem{FK}
Z. F\"{u}redi and J. Koml\'{o}s.
\newblock The eigenvalues of random symmetric matrices.
\newblock {\em Combinatorica}, 1(3):233--241, 1981.
  
\bibitem{Gfr}
L.~Gallardo.
\newblock Au sujet du contenu probabiliste d'un lemme d'Henri Poincar\'{e}. 
\newblock Saint-Flour Probability Summer Schools (Saint-Flour, 1979/1980).
{\em Ann. Sci. Univ. Clermont-Ferrand II Math.} No. 19 (1981), 185--190.

\bibitem{G}
A.~Gut.
\newblock {\em Probability: a graduate course}.
\newblock Springer Texts in Statistics. Springer, New York, 2005.

\bibitem{HJ}
R.~A. Horn and C.~R. Johnson.
\newblock {\em Matrix analysis}.
\newblock Cambridge University Press, Cambridge, second edition, 2013.

\bibitem{J}
T.~Jiang.
\newblock How many entries of a typical orthogonal matrix can be approximated
  by independent normals?
\newblock {\em Ann. Probab.}, 34(4):1497--1529, 2006.

\bibitem{KY}
A.~Knowles and J.~Yin.
\newblock Eigenvector distribution of Wigner matrices.
\newblock {\em Probability Theory and Related Fields}, 155(3-4):543--582, 2013.

\bibitem{KY2}
A.~Knowles and J.~Yin.
\newblock The outliers of a deformed Wigner matrix
\newblock{\em Ann. Probab.}, Volume 42, Number 5 (2014), 1980--2031.

\bibitem{LM}
B.~Laurent and P.~Massart.
\newblock Adaptive estimation of a quadratic functional by model selection.
\newblock {\em Ann. Statist.}, 28(5):1302--1338, 2000.

\bibitem{LY}
J.~O. Lee and J.~Yin.
\newblock A necessary and sufficient condition for edge universality of
  {W}igner matrices.
\newblock {\em Duke Math. J.}, 163(1):117--173, 2014.

\bibitem{M}
J.~Matou{\v{s}}ek.
\newblock {\em Lectures on discrete geometry}, volume 212 of {\em Graduate
  Texts in Mathematics}.
\newblock Springer-Verlag, New York, 2002.

\bibitem{McSherry}
F.~McSherry.
\newblock Spectral partitioning of random graphs.
\newblock In {\em 42nd {IEEE} {S}ymposium on {F}oundations of {C}omputer
  {S}cience ({L}as {V}egas, {NV}, 2001)}, pages 529--537. IEEE Computer Soc.,
  Los Alamitos, CA, 2001.

\bibitem{Mitra}
P.~Mitra.
\newblock Entrywise bounds for eigenvectors of random graphs.
\newblock {\em Electron. J. Combin.}, 16(1):Research Paper 131, 18, 2009.

\bibitem{Newman}
M.~E. Newman.
\newblock Finding community structure in networks using the eigenvectors of
  matrices.
\newblock {\em Physical review E}, 74(3):036104, 2006.

\bibitem{Newman1}
M.~E. Newman.
\newblock Modularity and community structure in networks.
\newblock {\em Proceedings of the National Academy of Sciences},
  103(23):8577--8582, 2006.

\bibitem{NTV}
H.~Nguyen, T.~Tao, and V.~Vu.
\newblock Random matrices: tail bounds for gaps between eigenvalues.
\newblock {\em Preprint, \href{http://arxiv.org/abs/1504.00396}{arXiv:1504.00396}}, 2015.

\bibitem{OT}
S.~O'Rourke and B.~Touri.
\newblock Controllability of random systems: Universality and minimal
  controllability.
\newblock {\em Preprint, \href{http://arxiv.org/abs/1506.03125}{arXiv:1506.03125}}, 2015.

\bibitem{OVW}
S.~O'Rourke, V.~Vu, and K.~Wang.
\newblock Random perturbation of low rank matrices: Improving classical bounds.
\newblock {\em Preprint, \href{http://arxiv.org/abs/1311.2657}{arXiv:1311.2657}}, 2016.

\bibitem{PBMW}
L.~Page, S.~Brin, R.~Motwani, and T.~Winograd.
\newblock The pagerank citation ranking: bringing order to the web.
\newblock {\em Technical Report. Stanford InfoLab}, 1999.

\bibitem{PY}
N.~S. Pillai and J.~Yin.
\newblock Universality of covariance matrices.
\newblock {\em Ann. Appl. Probab.}, 24(3):935--1001, 2014.

\bibitem{PSR}
A.~Pizzo, D.~Renfrew, and A.~Soshnikov.
\newblock On Finite Rank Deformations of Wigner Matrices.
\newblock {\em Annales de l'Institut Henri Poincar\'{e} Probabilites et Statistiques}, vol. 49, no. 1, 64--94, (2013). 

\bibitem{PSL}
A.~Pothen, H.~D. Simon, and K.-P. Liou.
\newblock Partitioning sparse matrices with eigenvectors of graphs.
\newblock {\em SIAM J. Matrix Anal. Appl.}, 11(3):430--452, 1990.
\newblock Sparse matrices (Gleneden Beach, OR, 1989).

\bibitem{RS}
D.~Renfrew and A.~Soshnikov.
\newblock On Finite Rank Deformations of Wigner Matrices II: Delocalized Perturbations.
\newblock {\em Random Matrices: Theory and Applications}, vol. 2, no. 1 (2013).

\bibitem{RVssv}
M.~Rudelson and R.~Vershynin.
\newblock Smallest singular value of a random rectangular matrix.
\newblock {\em Comm. Pure Appl. Math.}, 62(12):1707--1739, 2009.

\bibitem{RVsurvey}
M.~Rudelson and R.~Vershynin.
\newblock Non-asymptotic theory of random matrices: extreme singular values.
\newblock In {\em Proceedings of the {I}nternational {C}ongress of
  {M}athematicians. {V}olume {III}}, pages 1576--1602. Hindustan Book Agency,
  New Delhi, 2010.

\bibitem{RVeigen}
M.~Rudelson and R.~Vershynin.
\newblock Delocalization of eigenvectors of random matrices with independent
  entries.
\newblock {\em Preprint, \href{http://arxiv.org/abs/1306.2887}{arXiv:1306.2887}}, 2013.

\bibitem{RVhw}
M.~Rudelson and R.~Vershynin.
\newblock Hanson-{W}right inequality and sub-{G}aussian concentration.
\newblock {\em Electron. Commun. Probab.}, vol. 18, no. 82, 1--9, 2013.

\bibitem{RVgaps}
M.~Rudelson and R.~Vershynin.
\newblock No-gaps delocalization for general random matrices.
\newblock {\em Preprint, \href{http://arxiv.org/abs/1506.04012}{arXiv:1506.04012}}, 2015.

\bibitem{Sband}
J.~Schenker. 
\newblock Eigenvector Localization for Random Band Matrices with Power Law Band Width.
\newblock {\em Communications in Mathematical Physics}, 290 (2009), no. 3, 1065--1097.  

\bibitem{SM}
J.~Shi and J.~Malik.
\newblock Normalized cuts and image segmentation.
\newblock {\em Pattern Analysis and Machine Intelligence, IEEE Transactions
  on}, 22(8):888--905, 2000.
  
\bibitem{Sosh}
A.~Soshnikov.
\newblock Poisson statistics for the largest eigenvalues of Wigner random matrices with heavy tails.
\newblock {\em Elect. Commun. in Probab.}, 9 (2004), 82--91.

\bibitem{Sosh2} 
A.~Soshnikov.
\newblock Poisson statistics for the largest eigenvalues in random matrix ensembles.  
\newblock In {\em Mathematical physics of quantum mechanics}, vol. 690 of {\em Lecture Notes in Phys.}, Springer, Berlin, 2006, pp. 351--364.  

\bibitem{Soxford}
T.~Spencer.
\newblock Random banded and sparse matrices.
\newblock In {\em The Oxford handbook of random matrix theory}.  Oxford University Press (2011), 471--488.  

\bibitem{S}
M.~D. Springer.
\newblock {\em The algebra of random variables}.
\newblock John Wiley \& Sons, New York-Chichester-Brisbane, 1979.
\newblock Wiley Series in Probability and Mathematical Statistics.

\bibitem{taka}
R.~Takahashi.
\newblock Normalizing constants of a distribution which belongs to the domain
  of attraction of the {G}umbel distribution.
\newblock {\em Statist. Probab. Lett.}, 5(3):197--200, 1987.

\bibitem{Tao-book}
T.~Tao.
\newblock {\em Topics in random matrix theory}, volume 132.
\newblock American Mathematical Soc., 2012.

\bibitem{TVedge}
T.~Tao and V.~Vu.
\newblock Random matrices: universality of local eigenvalue statistics up to
  the edge.
\newblock {\em Comm. Math. Phys.}, 298(2):549--572, 2010.

\bibitem{TVuniv}
T.~Tao and V.~Vu.
\newblock Random matrices: universality of local eigenvalue statistics.
\newblock {\em Acta Math.}, 206(1):127--204, 2011.

\bibitem{TVmeh}
T.~Tao and V.~Vu.
\newblock The {W}igner-{D}yson-{M}ehta bulk universality conjecture for
  {W}igner matrices.
\newblock {\em Electron. J. Probab.}, 16(77):2104--2121, 2011.

\bibitem{TVcovariance}
T.~Tao and V.~Vu.
\newblock Random covariance matrices: Universality of local statistics of
  eigenvalues.
\newblock {\em The Annals of Probability}, 40(3):1285--1315, 2012.

\bibitem{TVuniv-vector}
T.~Tao and V.~Vu.
\newblock Random matrices: Universal properties of eigenvectors.
\newblock {\em Random Matrices: Theory and Applications}, 1(01):1150001, 2012.

\bibitem{TVsur}
T.~Tao and V.~Vu.
\newblock Random matrices: the universality phenomenon for {W}igner ensembles.
\newblock In {\em Modern aspects of random matrix theory}, volume~72 of {\em
  Proc. Sympos. Appl. Math.}, pages 121--172. Amer. Math. Soc., Providence, RI,
  2014.

\bibitem{TVsimple}
T.~Tao and V.~Vu.
\newblock Random matrices have simple spectrum.
\newblock Available at {\tt arXiv:1412.1438}.

\bibitem{TVW}
L.~V. Tran, V.~H. Vu, and K.~Wang.
\newblock Sparse random graphs: eigenvalues and eigenvectors.
\newblock {\em Random Structures Algorithms}, 42(1):110--134, 2013.

\bibitem{V}
R.~Vershynin.
\newblock Spectral norm of products of random and deterministic matrices.
\newblock {\em Probab. Theory Related Fields}, 150(3-4):471--509, 2011.

\bibitem{RV}
R.~Vershynin.
\newblock Introduction to the non-asymptotic analysis of random matrices.
\newblock In {\em Compressed sensing}, pages 210--268. Cambridge Univ. Press,
  Cambridge, 2012.

\bibitem{vonL}
U.~von Luxburg.
\newblock A tutorial on spectral clustering.
\newblock {\em Stat. Comput.}, 17(4):395--416, 2007.

\bibitem{VW}
V.~Vu and K.~Wang.
\newblock Random weighted projections, random quadratic forms and random
  eigenvectors.
\newblock {\em Random Structures \& Algorithms}, to appear.

\bibitem{W-edge}
K.~Wang.
\newblock Random covariance matrices: universality of local statistics of
  eigenvalues up to the edge.
\newblock {\em Random Matrices: Theory and Applications}, 1(01):1150005, 2012.

\end{thebibliography}

\end{document}